\documentclass[11pt]{article}
\usepackage{amsmath,amssymb,amsthm,amsfonts,amstext,amsbsy,amscd,color}
\usepackage{fullpage}
\usepackage{graphicx}
\usepackage{caption}
\usepackage{subcaption}
 
\usepackage{eucal}

 \usepackage[applemac]{inputenc}

\def\<{\langle}
\def\>{\rangle}

\def\Chi{\raise .3ex
\hbox{\large $\chi$}}

\def\({\Bigl (}
\def\){\Bigr )}
\newcommand{\be}{\begin{equation}}
\newcommand{\p}{\partial}
\newcommand{\1}{{\mathchoice {\rm 1\mskip-4mu l} {\rm 1\mskip-4mu l}{\rm 1\mskip-4.5mu l} {\rm 1\mskip-5mu l}}}
\newcommand{\f}{\frac}
\newcommand{\ee}{\end{equation}}
\newcommand{\bea}{$$ \begin{array}{lll}}
\newcommand{\eea}{\end{array} $$}
\newcommand{\bi}{\begin{itemize}}
\newcommand{\ei}{\end{itemize}}

\newcommand{\LB}{\LL_B}
\newcommand{\NB}{\NN_B}
\newcommand{\GB}{\GG_B}

\newcommand{\NN}{\mathcal N}
\newcommand{\GG}{\mathcal G}
\newcommand{\LL}{\mathcal L}

\newtheorem{prop}{Proposition}

\newtheorem{corollary}{Corollary}
\newtheorem{lemma}{Lemma}

\newtheorem{remark}{Remark}

\DeclareMathOperator{\R}{{\mathbb R}}

\DeclareMathOperator{\argmin}{argmin}
\DeclareMathOperator{\argmax}{argmax}

\renewenvironment{proof}{\noindent{\bf Proof.}}{\hfill
  $\blacksquare$\par\noindent}

\usepackage[colorlinks=true,citecolor=blue,urlcolor=blue]{hyperref}

\title{Estimating the division rate from indirect measurements \\  of single cells}
\author{Marie Doumic\footnote{Sorbonne Universit\'{e}, Inria, Universit\'{e} Paris-Diderot, CNRS,  Laboratoire Jacques-Louis Lions, F-75005 Paris, France. Email adress: marie.doumic@inria.fr} \and Ad\'ela\"ide Olivier\footnote{Laboratoire de Math\'ematiques d'Orsay, Univ. Paris-Sud, CNRS, Universit\'e Paris-Saclay, 91405 Orsay, France. Email adress: adelaide.olivier@u-psud.fr} \and Lydia Robert}

\begin{document}


\maketitle


\begin{abstract}
Is it possible to estimate the dependence of a growing and dividing population on a given trait in the case where this trait is not directly accessible by experimental measurements, but making use of measurements of another variable? This article adresses this general question for a very recent and popular model describing  bacterial growth, the so-called \emph{incremental} or \emph{adder} model. In this model, the division rate  depends on the \emph{increment of size} between birth and division, whereas the most accessible trait is the size itself. We prove that estimating the division rate from size measurements is possible, we state a reconstruction formula in a deterministic and then in a statistical setting, and solve numerically the problem on simulated and experimental data. Though this represents a severely ill-posed inverse problem, our numerical results prove to be satisfactory.
\end{abstract}
%

\section*{Introduction}
The field of structured population equations has attracted much interest for more than sixty years, leading to substantial progress  in their mathematical understanding. These equations describe a population dynamics in terms of well-chosen traits, which offer  a relevant characterization of the individual behaviour.  More recently, thanks to  considerable progress in experimental measurements, the question of estimating the parameters from single-cell measurements also attracts a growing interest, since it finally allows  comparing models model and data, and thus investigating which variable is biologically relevant as a structuring variable - see for instance~\cite{DHKR2} for the application to age-structured and size-structured models for bacterial growth. 

However, the so-called \emph{structuring} variable of the model may be quite abstract ("maturity", "satiety"...), and/or not directly measurable, whereas the quantities that are effectively measured may be linked to the structuring one in an unknown or intricate manner. As an illustration of this idea, we can cite the interesting  series of articles by H.T. Banks and co-authors, concerning the estimation of the division rate in data sets where the measured quantity was the fluorescence (carboxyfluorescein succinimidyl ester (CFSE)) of the cells. Initially, they designed a fluorescence-structured model~\cite{banks2011estimation}, but then the estimated division rates  appeared  difficult to interpret biologically. Indeed, the fluorescence was artificially added to the cells, thus it was not \emph{structuring}: the difficulty was to find out \emph{which} variable was really structuring, and how it was related to the measured quantities. This was done successfully by this group by building a model structured in both the true structuring variable - the so-called "cyton model" - and in the label, \emph{i.e.} the measured quantity~\cite{banks2012division}.

From such considerations we can formulate a general question: is it possible to estimate the dependence of a population on a given variable, which is not experimentally measurable, by taking advantage of the measurement of  another variable? 

In this article, we address this  question in a specific setting, namely the growth and division of bacteria. Recently, it was evidenced that for several types of bacteria and yeast cells, the "increment of size", \emph{i.e.} the increase of size of a cell between its birth and its division, provides a better-fitted model than age- or size-structured models~\cite{campos2014constant,eun2018archaeal,fievet2015single,Jun_2015}. 
These studies were based on data obtained by time-lapse microscopy and consisting in measurements of single-cells growing and dividing. Such data allows estimating for each cell its lifetime, its size at birth and at division, and its size-evolution through time. We refer to this kind of data as "measurements of dividing cells".

Comparison of models and data, such as performed in the above-mentioned studies, requires time-lapse microscopy data obtained in finely controlled conditions ensuring stable, steady-state growth. In addition, precise image analysis is also required to obtain accurate size measurements of numerous single-cells.
Obtaining such data is therefore not straight-forward and can be time-consuming. This can represent an important limitation, for instance for screening strategies where data has to be obtained in many different bacterial strains or experimental conditions. Here we consider the case of data consisting only in instantaneous size measurements of single-cells in a population. Such measurements can be more easily obtained, by microscopy snapshots, or using a flow cytometer or a coulter counter which both allow high-throughput acquisition.

From such data, the question of estimating the division rate in a size-structured model has been studied in a series of papers,  in a deterministic~\cite{BDE,DPZ,PZ} or statistical~\cite{DHRR} setting. The rates of convergence for the estimates have been proved to correspond to an inverse problem of degree of ill-posedness one, hence worse than the rate of convergence obtained from measurements of dividing cells (corresponding, in a deterministic setting, to a degree of ill-posedness zero, see~\cite{DHKR1} for a discussion of this heuristics). 

In view of the new biological evidence in favour of the incremental model~\cite{cadart2018size,deforet2015cell, eun2018archaeal,soifer2016single,Jun_2015}, this article is devoted to the same question as in~\cite{PZ} and following articles, but for the incremental model:  Can we  estimate an increment-dependent division rate from a measurement of the size-distribution of cells? Though formulated in a similar way, this new problem is much more complex, since the observed variable (instantaneous size) is not the structuring variable (size increment). 

Let us now give a mathematical definition of the problem under study. First of all, we  recall the increment-structured model.

\subsection*{The incremental model for  bacterial growth}

Let us denote $u(t,a,x)$ the density of cells at time $t$ of size $x$ which have an increment $a=x-y$ between their actual size $x$ and their size at birth $y$. We denote this increment $a$ as it may be viewed as a kind of  age, since it increases monotonically and starts at zero at birth - but an age that would have a link with the size: if $g(x)$ denotes the growth rate of a cell of size $x$, its "aging rate" is also $g(x).$ We have the following increment-and-size model, as proposed in Taheri-Araghi {\it et al.}~\cite{Jun_2015} for bacteria, and also designed in a different context by Hall {\it et al.} in~\cite{Hall1991} :
\begin{eqnarray}\label{eq:main}
\f{\p}{\p t} u(t,a,x) + \f{\p}{\p a} (g(x)u(t,a,x))+\f{\p}{\p x} (g(x)u(t,a,x)) + g(x)B(a,x)u(t,a,x)=0, \\ \nonumber \\
g(x)u(t,0,x)=4 \int\limits_0^\infty g(2x) B(a,2x)u(t,a,2x)da,\qquad g(0)u(t,a,0)=0,\qquad u(0,a,x)=u^{in}(a,x).
\end{eqnarray}
The instantaneous probability to divide is, as in~\cite{Hall1991}, $g(x)B(a,x)$ for a cell of size-increment $a$ and size $x$. From a modelling point of view, writing the division rate as the product of $g$ and a function $B$ allows us to interpret $B$ as the instantaneous probability to divide \emph{in a unit of growth} instead of a unit of time. This is coherent with the fact that the cell may ignore the "time" and use its growth as a clock. As will be explained below, in the case proposed by  
 Taheri-Araghi {\it et al.}~\cite{Jun_2015} where $B(a,x)=B(a)$, this is also  coherent with a much simpler and more natural underlying piecewise deterministic Markov process (PDMP), where the time of division is a simple renewal process of jump rate $B(a)$, so that the increments of dividing cells are mutually independent and distributed according to the density $f_B(a) = B(a) \exp\big(-\int\limits_0^a B(s) ds \big)$.

\subsection*{Asymptotic behaviour of the incremental model}
Under suitable assumptions on the coefficients $g$ and $B$  (for instance the theorem 3.7. of \cite{D} may be  adapted and provides a proof for smooth growth and division rates, and most recently~\cite{gabriel:hal-01742140} studies the case $g(x)=x$, with fairly general division rate $B$),  
we  have a dominant eigentriplet $(\lambda,U,\phi)$ unique solution of
\begin{eqnarray}\label{eq:U}
 \lambda U(a,x) + \f{\p}{\p a} (g(x)U(a,x))+\f{\p}{\p x} (g(x)U(a,x)) + g(x)B(a)U(a,x)=0, \label{eq:U}\\ \nonumber \\
g(x)U(0,x)=4 \int\limits_0^\infty g(2x) B(a)U(a,2x)da,\quad g(0)U(a,0)=0,\quad  \label{eq:boundU}\\ \nonumber \\
\lambda \phi(a,x) - g(x)\f{\p}{\p a} \phi(a,x) -g(x) \f{\p}{\p x}  \phi(a,x) + g(x)B(a)\phi(a,x)=2 g(x)B(a)\phi(0,\f{x}{2}),\label{eq:phi}\quad\\ \nonumber \\
\lambda>0,\quad \phi\geq 0,\quad U\geq 0, \qquad \iint U(a,x)dadx=1,\qquad\iint U(a,x)\phi(a,x)dadx=1.\label{eq:intUphi}
\end{eqnarray}
Moreover, using for instance  the general relative entropy inequalities, we  have under some more assumptions (theorem 4.5. in~\cite{D})
$$\iint \vert u(t,a,x)e^{-\lambda t}- \bigl(\iint u^{in}(a,x)\phi(a,x)dadx \bigr) U(a,x)\vert \phi(a,x)dadx \to_{t\to\infty} 0.$$ 
Let us however notice that under the assumption made in~\cite{Jun_2015}, namely that $g(x)=x,$  this precise asymptotic behaviour fails to happen and a cyclic behaviour is observed, as already proved for the growth-fragmentation equation~\cite{bernard2016cyclic,Greiner198879}. This corresponds to an idealised case; in reality, there is always some variability, leading to a growth slightly different from perfectly exponential and to a division into two slightly unequal parts, see for instance~\cite{DHKR2}. From a numerical perspective, this slight "imperfection" has to be included, else the cyclic behaviour will perturb the results, see Section~\ref{sec:num} for more details.

\subsection*{Mathematical formulation of the inverse problem}
From now on, we denote $U$ by $U_B$ to underline the dependence in the unknown division rate.
Let us assume that we measure  the steady size-distribution, which is modeled by the marginal $U_{B,x}(x)=\int\limits_0^\infty U_B(a,x)da$. Such a measurement may be done for instance via a  sample of $n$ cells for which we measure their sizes $(x_1,\cdots,x_n),$ assumed to be the realization of an i.i.d. sample distributed along $U_{B,x},$ in the spirit of Doumic, Hoffmann, Reynaud-Bouret and Rivoirard~\cite{DHRR}. 

We also assume division into two equally-sized daughters, as modelled in~\eqref{eq:main}, and that $g(x)$ and $\lambda$ are already known from an independent measurement or previous knowledge. Typically, $\lambda$ may be measured through the time evolution of the total mass, which is classical in biology~\cite{monod1949growth}, and under the consensual assumption of exponential growth $g(x)=\tau x$,  we have $\lambda=\tau.$ In this article, we do not consider the noise in the measurement of $\lambda$ and $g,$ which could be included in future work.

The problem we want to solve in order to have a fully determined model is:

\begin{center}
\fbox{\begin{minipage}{0.6\textwidth}\begin{center}
Given $\big(\lambda,g(\cdot)\big)$  and given measurements of $x\to U_{B,x}(x),$

 can we estimate the division rate $a\to B(a)$?
\end{center}
\end{minipage}}
\end{center}

\

In Section~\ref{sec:rec}, we provide an explicit though intricate formula for the estimation of $B$ from $U_B$, without taking the noise into account. We then provide a statistical estimator in Section~\ref{sec:stat}, numerically implemented in Section~\ref{sec:num} both on simulated data and real data. These numerical results provides us with clues concerning directions for future work, which we comment in the discussion (Section \ref{sec:disc}).

\section{Reconstruction formula in a deterministic setting}
Before providing the reconstruction formula for $B,$ let us introduce some useful notation. As standard in the field of renewal processes, we introduce the probability distribution function $f_B$ and the survival function $S_B$ of the increments of dividing cells:
\begin{equation}\label{eq:deff_B}
f_B(a):=B(a)\exp(-\int\limits_0^a B(s)) ds,\qquad S_B(a):=\int\limits_a^\infty f_B(s)ds=e^{-\int\limits_0^a B(s) ds}.
\end{equation}
Symetrically, we introduce the size-distribution of dividing cells, that we denote $\LB$
\begin{equation} \label{eq:sizediv_def}
\LB(x) = \int\limits_0^x g(x)B(a) U_B(a,x)da.
\end{equation}
As shown below, the function~$\LB$ is an important intermediate to formulate $B$ from the measurement of $U_{B,x}.$ Though we cannot write it as an explicit function of $U_{B,x},$ it  can be obtained in a similar way as the distribution of dividing cells for the size-structured equation, see~\cite{DPZ}.
\begin{lemma}\label{lem:LB}
Let $g$ be a positive continuous function on $(0,\infty)$, $\lambda\geq 0,$ and $U_{B,x}$ a positive function on $(0,\infty)$ such that  $\lambda U_{B,x}+\f{d}{dx} (gU_{B,x}) \in L^2(x^p dx)$ with $p\in [0,\infty)\setminus \{3\}$.
Then there exists a unique solution $\LB \in L^2(x^p dx)$ such that
\begin{eqnarray} \label{eq:sizediv2}
\lambda U_{B,x}(x) + \f{d}{dx}(gU_{B,x}\big)(x) = 4 \LB(2x) - \LB(x),
\end{eqnarray}
and there exists $C_p>0$ such that
$$\Vert \LB\Vert_{L^2 (x^pdx)} \leq C_p \Vert \lambda U_{B,x}+\f{d}{dx} (gU_{B,x}) \Vert_{L^2 (x^pdx)} .$$
If moreover there exists $U_B \in W^{1,1}((0,\infty)\times (0,\infty))$  solution of the system~\eqref{eq:U}-\eqref{eq:boundU}  such that $U_{B,x}=\int\limits_0^\infty U_{B}(a,x)da,$ then this unique solution coincides with the size-distribution of dividing cells defined by~\eqref{eq:sizediv_def}.
\end{lemma}
\begin{proof}
The existence, uniqueness and continuity part directly follows from~\cite{DPZ} Proposition A.1. If $U_B$ is solution to~\eqref{eq:U}-\eqref{eq:boundU} , we integrate Equation~\eqref{eq:U} along the increment $a$, use the boundary condition~\eqref{eq:boundU},  and obtain Equation~\eqref{eq:sizediv2} with $\LB$ defined by~\eqref{eq:sizediv_def}.
\end{proof}

With this lemma, we see  that the estimation of $\LB$ from $U_{B,x}$ is an inverse problem of degree of ill-posedness~1 when stated in a space $L^2(x^p dx)$ (degree 3/2 in the framework of a statistical noise~\cite{DHRR,NP}), as already known from~\cite{PZ}. Interestingly, we remark that this would remain true also if the division rate would depend on other structuring variables, as soon as the growth rate and the division kernel are known: this allows one to reconstruct the size distribution of dividing cells from the size distribution of all cells, in any framework. 

%
%

We are now ready to  formulate  $B$ in terms of $U_{B,x}$, $\LB,$ and the parameters $\lambda$ and $g.$ 
This is done in the next proposition.

In all what follows, we denote by $f^*$ the Fourier transform of a function $f$:
$$
f^*(\xi) = \int\limits_{-\infty}^{+\infty} f(x) e^{\bold{i}x\xi} dx.
$$

\begin{prop}\label{prop:rec}
Let $B$ and $g$ be such that there exists a unique positive eigentriplet  $(\lambda, U_B,\phi_B)$ solution of the eigenproblem~\eqref{eq:U}--\eqref{eq:intUphi}. Let us furthermore assume $\lambda U_{B,x}+\f{d}{dx} (gU_{B,x}) \in L^2(x^p dx)$ and define $\LB$ as the unique solution of Equation~\eqref{eq:sizediv2} given by Lemma~\ref{lem:LB}. 

We define $f_B$ and $S_B$ by~\eqref{eq:deff_B}. We define two intermediate functions $\NB$ and ${\GB}$  by
\begin{equation} \label{eq:G_def}
\GB (y)=4 e^{\lambda G(y)} \LB(2y), \qquad \NB (y)=g(y) e^{\lambda G(y)} U_{B,x}(y),
\end{equation} 
with $G(x)$ an anti-derivative of $1/g(x).$ We assume that $\NB^{*}$ and ${\GB^*}$ are the well-defined Fourier transforms of $\NB$ and ${\GB}$.
 
We have the following reconstruction formula for $B$ in terms of $\lambda,$ $U_{B,x}$ and $g:$
\begin{equation}\label{eq:B}
B(a) 
= \frac{f_B(a)}{S_B(a)} 
= \frac{\int\limits_{-\infty}^{+\infty}  \Big( 1 + \bold{i} \xi \frac{\NB^{*}(\xi)}{\GB^{*}(\xi)} \Big)e^{-\bold{i}a\xi} d\xi}{\int\limits_{a}^{+\infty} \bigg( \int\limits_{-\infty}^{+\infty}  \Big( 1 + \bold{i} \xi \frac{\NB^{*}(\xi)}{\GB^{*}(\xi)} \Big)e^{-\bold{i} s \xi} d\xi \bigg) ds},
\end{equation}
provided that all the inverse Fourier transforms are well defined and that neither $\GB^*$ nor the denominator vanishes.
\end{prop}
\begin{corollary}Under the assumptions of Proposition~\ref{prop:rec}, if $g(x)=\tau x$ we have $\lambda=\tau$, $\GB(y)=4y\LB(2y)$ and $\NB(y)=\tau y^2 U_{B,x} (y).$\label{coroll1}
\end{corollary}

\begin{remark} At this stage, the reconstruction formula is formal: to give it a rigorous meaning and ensure its validity, we would have to prove that all the quantities are well-defined, in particular that the Fourier transform $\GB^*$ never vanishes. This requires  a full study \emph{per se},  and is beyond the scope of this work: in another case study, this has been done for instance for the estimation of the fragmentation kernel of the growth-fragmentation equation in the article~\cite{doumic:hal-01501811}, using the Cauchy integral to prove that a Mellin transform never vanishes, proof adapted to another case in~\cite{hoang:hal-01623403}. For these two cases however, the proofs used strongly an explicit formulation of the solution with the use of Mellin or Fourier transforms, thanks to the fact that $B$ was a power law in~\cite{doumic:hal-01501811}, and constant in~\cite{hoang:hal-01623403}. We let it for future work.
\end{remark}

\

\begin{proof}
The aim is to use the classical formula
$B(a)=\f{f_B(a)}{S_B(a)},$
and to find a formulation for $f_B$ in terms of $U_{B,x}$, then express  $S_B$  as its integral. \\

\noindent {\it First step. Formulating $U_{B,x}$ in terms of $\lambda,$ $g$, $\LB$ and $B$.}

As done for the study of the eigenvalue problem carried out for instance in~\cite{D,gabriel:hal-01742140}, we can classically obtain a formulation of $U_B(a,x)$ in terms of $\LB.$ 
We first write Equation~\eqref{eq:boundU}  under the form
\begin{equation}\label{eq:Ubord}
g(x)U_B(0,x)=4 \LB(2x),
\end{equation}
and then use the method of characteristics to solve~\eqref{eq:U} and \eqref{eq:Ubord}. We define an intermediate function $C(a,x)=g(x)U_B(a,x)$ solution of
$$
\f{\p}{\p a} C(a,x)+\f{\p}{\p x} C(a,x) + \big(\f{\lambda}{g(x)} + B(a)\big)C(a,x)=0, 
$$
we define $\tilde C(a,x) =C(a,x+a)e^{\int\limits_0^a \bigl(\f{\lambda}{g(x+s)} + B(s)\bigr) ds},$ which satisfies
$\f{\p}{\p a}\tilde C(a,x)=0,$
so that
$$C(a,x+a)e^{\int\limits_0^a \bigl(\f{\lambda}{g(x+s)}+B(s)\big)ds}=C(0,x),$$
which gives for $U_B$:
\begin{equation}\label{eq:UBLB}
g(y)U_B(a,y)=g(y-a)U_B(0,y-a)e^{-\int\limits_0^a \bigl(\f{\lambda}{g(y-a+s)}+B(s)\bigr)ds} = 4\LB \big(2(y-a)\big) e^{-\int\limits_0^a \bigl(\f{\lambda}{g(y-a+s)}+B(s)\bigr)ds},
\end{equation}
using \eqref{eq:Ubord} for the last equality. We  integrate the equation~\eqref{eq:UBLB} in $a$ and obtain
$$
g(y)U_{B,x}(y)=\int\limits_0^y 4 \LB \big(2(y-a)\big) e^{-\int\limits_0^a \bigl( \f{\lambda }{g(y-a+s)}+B(s)\bigr)ds} da.
$$

\

\noindent  {\it Second step. Formulating two deconvolution problems for $S_B$ and $f_B$.}

Denoting by $G$ an antiderivative of $1/g$, and defining an intermediate function $\NB,$ the previous formula is equivalent to
\begin{equation}
\NB(y):=g(y) e^{\lambda G(y)} U_{B,x}(y) = 4 \int\limits_0^y e^{\lambda G(y-a)} \LB \big(2(y-a)\big)  e^{-\int\limits_0^a B(s) ds} da.\label{eq:N_def}
\end{equation}
We define $\GB$ by \eqref{eq:G_def}, thus Equation \eqref{eq:N_def} is nothing but 
\begin{equation} \label{eq:convol1}
\NB (x) = \big[ \GB \star S_B \big] (x).
\end{equation}
This is a deconvolution problem, where  $S_B$ is the unknown. If we find estimators of $\NB$ and $\GB$, we can reconstruct $S_B$. Since $f_B=-\f{d}{da} S_B,$  integrating by parts we can transform~\eqref{eq:convol1} into a deconvolution problem for $f_B$.

\

%

\noindent {\it Third step. Solution of the deconvolution problems by Fourier transform.}

Extending all the functions on $\R_-$ by zero (with a slight abuse of notation, we keep the same notation for the function and for its extension), we  rewrite \eqref{eq:convol1} in the Fourier domain, and for $\xi \in \mathbb R$ such that $\GB^*(\xi) \neq 0$:
\begin{align*}
\NB^{*} &= \GB^*  S_B^* \quad \implies S_B^*(\xi) = \frac{\NB^{*}}{\GB^*}(\xi) 
. 
\end{align*}
Since $f_B$ is a probability density  for which we assume continuity around $0$ and $f_B(0=0,$ we can extend it continuously by $0$ on $\R_-$, we have $f_B^*(0)=1,$ and  and since $f_B=-\f{d S_B}{da}$ for $a>0$ we get
$$f_B^*=1+\bold{i}\xi S_B^*= 1 + \bold{i} \xi \frac{\NB^{*}(\xi)}{\GB^{*}(\xi)},$$
for $\xi \in \mathbb R$ such that $\GB^*(\xi) \neq 0$. The $1$ expresses the Fourier transform of the discontinuity of the prolongation of $S_B$ in $0,$ since $S_B(0^+)=1$; however, this term should be compensated by $\bold{i} \xi \frac{\NB^{*}(\xi)}{\GB^{*}(\xi)},$ since we have assumed $f_B(0)=0$.

\

\noindent {\it Fourth step. Inverse Fourier transforms.}  $f_B$ and $S_B$ are given by the formulae 
\begin{align*}
f_B(a)  = \frac{1}{2\pi} \int\limits_{-\infty}^{+\infty}  \Big( 1 + \bold{i} \xi \frac{\NB^{*}(\xi)}{\GB^{*}(\xi)} \Big)e^{-\bold{i}a\xi} d\xi \quad \textrm{and} \quad S_B(a)  = \int\limits_{a}^{+\infty} f_B(s) ds,
\end{align*}
{provided all these quantities exist,} and we have proved the formula~\eqref{eq:B}.
%
\end{proof}

\begin{remark} \label{rk:alternative} 
An alternative formula for $B(a)$ would be obtained using a direct formula for the survival function, $$S_B(a) = \frac{1}{2\pi} \int\limits_{-\infty}^{+\infty}   \frac{\NB^{*}}{\GB^*}(\xi)   e^{-\bold{i}a\xi} d\xi .$$
\end{remark}

\label{sec:rec}

\section{Statistical setting: estimation procedure}
\label{sec:stat}

Let us assume that we have a sample $X_1,\ldots, X_n$ independent and identically distributed according to $U_{B,x}$. This idealizes the case where, for instance, a picture of all individuals at a given time is taken, and their sizes experimentally measured.
We give here a procedure to estimate $B$ from such a sample.

%

%


\subsection{Estimation of $\GB$}\label{sec:estim:GB}

\noindent {\it Step 1.} Estimation of $U_{B,x}$ by a kernel estimator
$$
\widehat{U}_{n,x}(y) = \frac{1}{n} \sum_{j=1}^n K_{h_1}(y-X_j)
$$
with $K_{h_1}(\cdot) = h_1^{-1} K(h_1^{-1} \cdot)$.  Estimation of $D(y) = \big(gU_{B,x}\big)'(y)$  by
$$
\widehat{D}_n(y) = \frac{1}{n} \sum_{j=1}^n g(X_j) K'_{h_2}(y-X_j).
$$
For the choice of the regularization parameters $h_1=h_{1,n}$ and $h_2=h_{2,n}$, see Section~\ref{sec:num}.\\

\noindent {\it Step 2.} Inversion of \eqref{eq:sizediv2} replacing the left-hand side by $\lambda \widehat{U}_{n,x}(y) + \widehat{D}_n(y)$. We obtain $\widehat{\LL}_n(y)$ an estimator of $\LB(y)$. For this step, we follow~\cite{BDE}, and concatenate the inverse given in Lemma~\ref{lem:LB} in $L^2(dx)$ for $x\leq \bar x$, that we denote $\widehat{\LL}_{n,l}(y)$, with the inverse in $L^2(x^4dx)$, denoted $\widehat{\LL}_{n,r}(y)$, for $x\geq \bar x,$ for a given (to be determined numerically) $\bar x>0$: we set
$$\widehat{\LL}_n(y)=\widehat{\LL}_{n,l}(y)\1_{x\leq \bar x} + \widehat{\LL}_{n,r}(y)\1_{x> \bar x}.$$

\

\noindent {\it Step 3.} We deduce an estimator  of $\GB(y)$ using its definition \eqref{eq:G_def}:\\
$$
\widehat{\GG}_n(y) = 4 e^{\lambda G(y)} \widehat{\LL}_n(2y).
$$ 

\subsection{Estimation of the Fourier transforms $\GB^*$ and $\NB^*$}

\noindent {\bf Estimation of $\NB^{*}$.}
Recall Equation \eqref{eq:N_def} giving the definition of $\NB$. Then
$$
\NB^*(\xi) = \int\limits_0^\infty  g(x) e^{\lambda G(x)} e^{\bold{i}x\xi} U_{B,x}(x) dx,
$$
which can be estimated by
$$
\widehat{\NN_n^*}(\xi) = \frac{1}{n} \sum_{j=1}^n g(X_j) e^{\lambda G(X_j)} e^{\bold{i} X_j \xi}.
$$
When growth is exponential {\it i.e.} in the specific case $g(x) = \tau x$, we have $\lambda  = \tau$ and $G(x) = \ln(x)/\tau$. Thus the previous formula simplifies into
$$
\widehat{\NN_n^*}(\xi) = \frac{\tau}{n} \sum_{j=1}^n (X_j)^2 e^{\bold{i} X_j \xi}.
$$

\noindent {\bf Estimation of $\GB^{*}$.}  We compute the Fourier transform of $\widehat{\GG}_n(y)$ (see Section~\ref{sec:num} for the practical details). It gives us an estimator $\widehat \GG_n^*$ of $\GB^*(y)$.\\

\subsection{Estimation of the Fourier transform $f_B^*$} 
\label{sec:estim:fB}
We have an estimator of the Fourier transform of $f_B$ by
$$
\widehat{f_n^*}(\xi) = 1 + \bold{i} \xi \frac{\widehat{\NN_n^*}(\xi)}{\widehat{\GG_n^*}(\xi)} {\bf 1}_{\Omega_n}(\xi)
$$
with $\Omega_n = \{ \xi \in \mathbb R ; |\widehat{\GG_n^*}(\xi)| \geq \underline{\xi} \}$, with $ \underline{\xi}=\underline{\xi_n}$ a well-adapted threshold.\\

\subsection{Inverse Fourier transforms to estimate $S_B$ and $f_B$.} 
Estimators of $f_B$ and $S_B$ are
\begin{align*}
\widehat{f}_{n,h_3}(a) & = \frac{1}{2\pi} \int\limits_{-1/{h_3}}^{1/{h_3}} \widehat{f_n^*}(\xi) e^{-\bold{i}a\xi} d\xi
%
\quad \textrm{and} \quad \widehat{S}_{n,h_3}(a)  = \int\limits_{a}^{\infty} \widehat{f}_{n,h_3}(s) ds
\end{align*}
with $h_3=h_{3,n}$ a parameter of regularization to be well-chosen (see Section~\ref{sec:num}).\\

\subsection{Reconstruction of $B$.} Finally the division rate B can be estimated at a given point $a$ by
$$
\widehat{B}_{n,h}(a) 
= \frac{\widehat{f}_{n,h}(a)}{\widehat{S}_{n,h}(a) \vee \varpi_2} 
= \frac{\int\limits_{-1/h}^{1/h}  \Big( 1 + \bold{i} \xi \frac{\widehat{\NN_n^*}(\xi)}{\widehat{\GG_n^*}(\xi)} {\bf 1}_{\Omega_n}(\xi) \Big)   e^{-\bold{i}a\xi} d\xi}
{\int\limits_{s}^{+\infty} \bigg( \int\limits_{-1/h}^{1/h}  \Big( 1 + \bold{i} \xi \frac{\widehat{\NN_n^*}(\xi)}{\widehat{\GG_n^*}(\xi)} {\bf 1}_{\Omega_n}(\xi) \Big)   e^{-\bold{i}s\xi} d\xi \bigg) ds \vee \varpi}
$$
with $\varpi =\varpi_{n}$ a well-adapted threshold.

\begin{remark} 
Following Remark~\ref{rk:alternative}, an alternative estimator of $B$ would be obtained using  $$\widehat{S}_{n,h_4}(a) = \frac{1}{2\pi} \int\limits_{-1/{h_4}}^{1/{h_4}} \widehat{S_n^*}(\xi) e^{-\bold{i}a\xi} d\xi \quad \textrm{where} \quad \widehat{S_n^*}(\xi) = \frac{\widehat{\NN_n^*}}{\widehat{\GG_n^*}}(\xi) {\bf 1}_{\Omega_n}(\xi)$$
with $h_4=h_{4,n}$ a regularization parameter to be well-chosen. Note that such a procedure would not guarantee the decaying property of $S_B,$ or yet the fact that $f_B=-S_B'$, $S_B(0)=1$ and $S_B(\infty)=0$.
\end{remark}


\section{Numerical study}\label{sec:num}
\subsection{Numerical implementation}

For a given estimator $\hat g$ of a function $g$ (with real or complex values), we evaluate $\hat g$ on a regular grid with mesh $\Delta t$ and compute the empirical error as 
$$
e = \frac{\| \hat g - g \|_{2,\Delta t}}{\| g \|_{2,\Delta t}}
$$
where $\| \cdot\|_{2,\Delta t}$ is the $L_2$-discrete norm over the numerical sampling.

\subsubsection*{Choice of the regularization parameters} 
In the aboveseen formulae, we have distinguished five successive steps, and several regularization parameters which need careful implementation: $h_1$,  $h_2$ and $\bar x$ for the estimation of $\LB$; $h_3$ for the integration domain of the inverse Fourier transform; and the thresholds $\underline{\xi}$ and $\varpi$ to avoid explosion when the denominators vanish in the formulae. Each of the five steps have been tested separately, and we discuss below practical ways to determine such regularization parameters.

\begin{itemize}
\item 
The parameter $h_1$ is either automatically chosen by the kernel smoothing function 
{\tt ksdensity} of Matlab, or chosen via an adaptive method such as Goldenschluger and Lepski, or Penalized Comparison to Overfitting (PCO) most recently introduced~\cite{PCO1,PCO2}. Similarly for $h_2$, we can use an adaptive method or choose it \emph{a priori}, in relation with $h_1$ -  for instance, if we have the a priori that $U_{B,x}\in H^1$, the order of an optimal choice for $h_1$ will be $O(n^{-\f{1}{3}})$, for $h_2$ it is $O(n^{-\f{1}{5}})$, so that we can choose \emph{a priori} $h_2=h_1^{\f{3}{5}}.$ 

\item To compute $\bar x$, we refer to the previous studies~\cite{BDE,DHKR2}. Since for any estimation $\widehat U_{B,x} \in H^1((1+x^p)dx)\cap W^{1,\infty}$ with $p>3$ the unique solution $\widehat {\cal L}_B^0$ in $L^2(dx)$ to Equation~\eqref{eq:sizediv2} is also in $L^1\cap L^\infty$, which is not the case for the unique solution $\widehat {\cal L}_B^p$ in $L^2(x^pdx),$ we keep the solution in $L^2(x^pdx)$ only for the right-hand tail of the distribution,  which leads us to define $\bar x$ as 
$$\bar x:= \argmin\limits_{x\geq x_{max}^0} \vert \widehat{\cal L}_B^p -\widehat{\cal L}_B^0 \vert,\qquad x_{max}^0:= \argmax \limits_{x\geq 0} \widehat {\cal L}_B^0.$$
\item The key parameter $h_3$ is chosen in an oracle way, that is to say that we minimise in $h$ the criterion
$$
e(h) = \frac{\| \hat f_{n,h} - f_B \|_{2,\Delta a}}{\| f_B \|_{2,\Delta a}}.
$$
This oracle choice requires the knowledge of $f_B$, which is impossible in practice, but our aim here is to learn how well our procedure can do (in ideal conditions, when the tuning parameters can be chosen perfectly).
In practice we have set $\underline{\xi} = 0$, since the regularization by $h_3$ suffices numerically. We chose $\varpi_n = 1/n$.\\

\end{itemize}
\subsubsection*{Use of the regularity properties of the functions}

The protocol described in~Section~\ref{sec:stat}  does not guarantee important properties of the functions, such that the positivity of $f_B$ and $B$, and $\int f_B da=1.$  All these characteristics will be enforced in the numerical study to improve qualitatively the results. 
%

\subsection{Numerical results on simulated data}

In order to evaluate the quality of the reconstruction and the influence of each step of the protocol, we first studied separately each estimation necessary for the reconstruction of $B.$ For the simulations, growth is exponential with rate 1, $g(x)=x$, so that $\lambda_B = 1$. We choose for the division rate $B(a)=a^2$ for $a\geq 0$. We thus immediately deduce formulae for the density $f_B$ and the survival function $S_B$. We compute numerically the Fourier transform $f_B^*$. All Fourier and Inverse Fourier transforms were computed by an integral using the same scheme as in \cite{BDE} (see Technical aspects of Section 4.1 in~\cite{BDE}).

\subsubsection*{Numerical solution for $U_{B,x}$}

To compute numerically the first eigenvector $U$ solution of \eqref{eq:U}--\eqref{eq:intUphi}, we follow the classical first order finite volume scheme proposed for instance in~\cite{brikci2008analysis,doumic2007analysis} for similar equations, renormalize the solution at each time step, and stop the iterations when the distribution has converged, the error between two successive steps being smaller than the precision desired. An important point is to allow the scheme to be slightly dissipative - which is the case by  choosing a regular grid - contrarily to the one proposed for instance in~\cite{bernard2016cyclic}, which would give rise to an oscillatory behaviour.

The solution $U_{B,x}$ is computed in our framework ($g(x)=x$ and $B(a)=a^2$) with regular grids: size ranges from 0 to 6 and increment ranges form 0 to 3 with meshes $\Delta x = \Delta a = 6/500$. We require a precision of $0.01$\%. Besides the stationary distribution in size we compute the size-distribution of dividing cells $\mathcal L_B$. By formula \eqref{eq:G_def}, we obtain numerically $\mathcal G_B$ and $\mathcal N_B$. (And we compute numerically the Fourier transforms $\mathcal G_B^*$ and $\mathcal N_B^*$.)

\subsubsection*{Protocol 1 -- Reconstruction of $B$ when both $U_{B,x}$ and $\LB$ are given with highest accuracy.}

For the reconstruction of $B$ we use directly the numerical and high-resolution solutions of $U_{B,x}$ and $\LB$ of the previous step. The noise is thus limited to the numerical error, itself very limited thanks to the high resolution of the grid and to the requirement of a very small error in the long-time asymptotics.

See Figure \ref{fig:etape1} for the different steps of the protocol. See Figure \ref{fig:resultats1et2} (red curves) and Tables~\ref{tab:resultats1et2} and~\ref{tab:resultats1et2_B} for results. 
\subsubsection*{Protocol 2 -- Reconstruction of $B$ when $U_{B,x}$ is given with highest accuracy but $\LB$ is unknown.}

See Figure \ref{fig:etape2} for the different steps of the protocol. See Figure~\ref{fig:resultats1et2} (yellow curves) and Tables~\ref{tab:resultats1et2} and~\ref{tab:resultats1et2_B} for results.\\

Both of the Protocols 1 and 2 give a satisfactory reconstruction of the division rate $B$ on the range $[0;2]$, with an error around 8\% (Table~\ref{tab:resultats1et2_B}). The estimation deteriorates for an increment of size higher than 2 since the probability for a cell to exceed this increment is less than 10\%. Computing the error on the wider range $[0;2.5]$, we surprisingly observe that Protocol 2 (with an error around 13\%) gives a more robust estimation than Protocol 1, which includes fewer statistical unknowns (error around 20\%).

Coming back to the first steps (Table~\ref{tab:resultats1et2}) we observe that Protocol 2 achieves errors below 5\%, whereas Protocol 1 leads to 10\% error for the reconstruction of the density $f_B$. Is this due to error compensation when computing the ratio in the reconstruction formula of Section \ref{sec:estim:fB}? Figure~\ref{fig:resultats1et2}d shows that the reconstruction of the Fourier transform of the density is good in modulus for frequency $|\xi|<5$ for Protocol 2, whereas the reconstruction by Protocol 1 deteriorates from smaller frequencies (around $\xi=\pm 3$).
Both protocols underestimate the maximum of the density $f_B$, but this is amplified using Protocol 1.
As a consequence of the bias in the estimation of the density $f_B$, we observe a bias in the estimation of the division rate $B$. It is slightly overestimated for increments of size lower than 1 and underestimated beyond. 

\subsubsection*{Protocol 3 -- Reconstruction of $B$ when $U_{B,x}$ is reconstructed from $X_1,\ldots,X_n$ i.i.d. $\sim$ $U_{B,x}$ and $\LB$ is given with highest accuracy.}

See Figure \ref{fig:etape3} for the different steps of the protocol. See Figures~\ref{fig:resultats3_IC}, \ref{fig:resultats34_IC_n} and~\ref{fig:resultats34_vitesse} for results. 

\subsubsection*{Protocol 4 -- Reconstruction of $B$ from  $X_1,\ldots,X_n$ i.i.d. $\sim$ $U_{B,x}$.}

See Figure \ref{fig:etape4} for the different steps of the protocol. See Figures~\ref{fig:resultats4_IC}, \ref{fig:resultats34_IC_n} and~\ref{fig:resultats34_vitesse} for results.\\

Protocols 3 and 4 have been repeated $M=100$ times for each tested $n$ ranging from $500$ to $50~000$. This enables us to obtain empirical confidence intervals (CI) for the estimation of the division rate $B$ and for the different intermediate reconstructions. As expected the computed 95\%-CI shrink as $n$ grows (Figure~\ref{fig:resultats34_IC_n}). The reconstruction of $B$ is satisfactory on the range $[0;1.75]$ when $n=500$, and slightly beyond 2 when $n=50~000$. We observe the same bias as the one already mentioned for Protocols~1 and~2. It seems even amplified looking at the mean of the 100 reconstructions, to such an extent that the true division rate $B(a)=a^2$ is on the fringe of the 95\%-CI when $n=50~000$.

One can plot the mean error over the $M=100$ reconstructions versus the sample size $n$ in log-log scale (Figure~\ref{fig:resultats34_vitesse}). Doing so linear curves are obtained and the extracted-slopes give us the speeds in the decrease of the error (with respect to $n$) for our different reconstructions. The speeds are surprisingly slightly better for Protocol~4 than for Protocol~3, which is in line with the comparison of Protocols~1 and~2. 

As regards Protocol~4, the speed for the estimation of $U_{B,x}$ is close to $n^{-0.4}$, which is expected (indeed $(n_0+1)/(2(n_0+1)+1) = 0.4$ with $n_0=1$ the order of a Gaussian kernel). For the estimation of $U_{B,x}^{'}$ we expect $(n_0+1)/(2(n_0+1)+2) \approx 0.33$ and we obtain worse (slope of $-0.25$). The inversion step in order to obtain $\mathcal L_B$ (and $\mathcal G_B$) does not deteriorate this speed hugely (slope of $-0.23$). After the computation of the Fourier transform, the speed for the estimation of $\mathcal G_B$ is of the same order (slope of $-0.25$). For the estimation of $\mathcal N_B$ we obtain the expected slope $-0.5$ which corresponds to a parametric speed in $n^{-1/2}$. It is not possible to predict {\it a priori} the speed of a quotient estimator, and we obtain a slope of $-0.24$ for the estimation of $f_B^*$. We obtain a final speed of $n^{-0.16}$ for the estimation of $B$ (for Protocol~4). This is much more difficult to interpret these last speeds since the regularity of $\mathcal N_B$ and $\mathcal G_B$ comes in. We refer to the Discussion below and to Johannes \cite{johannes09} for general results on the quality of density estimators in a deconvolution problem when the "noise" law is unknown, generalizing results such as Fan \cite{fan91} when the "noise" law is known. See also the study of Belomestny and Goldenshluger \cite{BG17} in the case of a multiplicative measurement error leading to the use of Mellin transform techniques (instead of an additive error leading to the use of Fourier transform ones).

Last but not least we observe a saturation of the error for very large $n$ ($n \geq 20~000$ in Protocol~3 and $n\geq 30~000$ in Protocol~4). Thus the slopes were computed taking into account this effect when necessary (removing the last points).


\subsection{Numerical results on experimental data}

We now turn to experimental data of bacterial growth to test the method. 
In the corresponding experiments cells are followed through time and the joint distribution of instantaneous size and size increment is estimated, as well as the joint distribution of size and size increment of  dividing cells. This allows us to compare our results obtained through our indirect method with a direct estimation of the division rate $B(a)$ from kernel density estimation of $f_B(a)$ and $S_B(a).$

The dataset we analysed comes from a single-cell experimental study on \emph{E. coli}
growth, performed by Stewart {\it et al.}~\cite{Eric}, and we used the data analysis performed in~\cite{DHKR2} (see Methods - data analysis). Following the results of~\cite{DHKR2}, we can assume here that  all cells grow with approximately the same growth rate $g(x)=\tau x$ with $\tau=0.0275 min^{-1}.$  Corollary~\ref{coroll1} then states that we have ${\cal G}_B (y) = 4y {\cal L}_B(2y)$ and ${\cal N}_B(y)=\tau y^2 U_{B,x} (y)$.

The experimental sample contains $n=31,333$ measurements of cell sizes. We perform a kernel density estimation with $h_1=0.125$ on Figure~\ref{fig:Nexper} to display it (Step~1 of the estimation procedure, Section~\ref{sec:estim:GB}). 

The second step consists in the estimation $\widehat {\cal L}_n$ of the size distribution of dividing cells  ${\cal L}_B$, through the numerical solution of Equation~\eqref{eq:sizediv2} (Step~2 of the procedure, Section~\ref{sec:estim:GB}). In our case of a rich dataset, we also have access to a sample of $n_d=1,679$ dividing cells, so that we can compare our estimation with the kernel density estimation of this sample (done here with $h_d=0.167$): we denote this estimate $\widehat {\cal L}_{n_d}^{h_d}$, and both are displayed in Figure~\ref{fig:BNexp}. We see that this approximation is relatively satisfactory, though far from being perfect. The distance between the two distributions has  two main reasons. First, the sample of dividing cells may be noisier than the sample of all cells: since the measurement is done only on time intervals of $2min$, there is an error due to the size growth of the cell during these $2min$. Second,  the cells may not all grow at exactly the same growth rate $\tau$ (see Figure~S4 in~\cite{DHKR2}) so that Equation~\eqref{eq:sizediv2} is an approximation, see for instance~\cite{DHKR1} for a more complete model including growth rate variability. However, we also note that this way of computing the size distribution of cells remains valid even if the division rate would depend on other variables, so that Equation~\eqref{eq:sizediv2} allows one to estimate the size distribution of dividing cells from the measurement of a size sample, provided the approximation of homogeneous growth rate is valid, and if this growth rate is measured independently.

For the following steps,
we do not have a direct way to compare the Fourier transforms of the intermediate functions ${\cal G}_B$ and ${\cal N}_B$  to directly measured quantities, but we can estimate $f_B$ (and so $S_B=\int\limits_a^\infty f_B(s)ds$ and $B=\f{f_B}{S_B},$ as classically done for renewal processes) from the increment-and-size experimental distribution  of dividing cells. Let us recall that $f_B$ is \emph{not} equal to the increment distribution of dividing cells, due to the well known bias selection effect~\cite{hoffmann:hal-01254203,DHKR2} of keeping the two daughter cells at each division step. However, for the specific case of a linear growth rate, we have an easy relation between the increment-and-size distribution of dividing cells and $f_B$. We notice that the function $U_1(a,x):=\f{xU(a,x)}{\iint x U(a,x) da dx}\geq 0,$ is solution of the system 
\begin{eqnarray}\label{eq:U1}
  \f{\p}{\p a} (g(x)U_1(a,x))+\f{\p}{\p x} (g(x)U(a,x)) + g(x)B(a)U_1(a,x)=0,  \\ \nonumber \\
g(x)U_1(0,x)=2 \int\limits_0^\infty g(2x) B(a)U_1(a,2x)da,\quad g(0)U_1(a,0)=0, \\ \nonumber \\
 U\geq 0, \qquad \iint U(a,x)dadx=1,\label{eq:intU1}
\end{eqnarray}
which defines the increment-and-size distribution of all cells in the conservative case, \emph{i.e.} when only one child is kept at each division. We thus have 
$$f_B(a)= \f{\int \tau x B(a)  U_1 (a,x) dx}{\iint \tau x B(a) U_1 (a,x) dadx}=\f{ \int x (x B(a)  U (a,x)) dx}{\iint x^2 B(a)  U (a,x) dadx},$$
and we notice  that this formula is nothing but a weighted average of  $\f{xB(a)U(a,x)}{\iint x B(a) U(a,x)dadx}$, which is the increment-and-size distribution of dividing cells  taken in experimental conditions, that is, when the two children are kept at each division (for a more detailed explanation on the links between the two points of view, we refer to~\cite{DHKR1}, Section 4.2). All these considerations
 provide us with the following estimate of $f_B$ from the increment-and-size sample of dividing cells: let us denote $(A_j,X_j)_{1\leq j\leq n_d}$ the 2-dimensional sample of increment $A$ and size $X$ at division, that we assume to behave as if $(A_j,X_j)$ were independent identically distributed according to the probability law  $\f{xB(a)U(a,x)}{\iint x B(a) U(a,x)dadx}$: we have

$$\widehat f_{B,n_d}(a):=\f{1}{n_d}\sum\limits_{j=1}^{n_d} K_{h_d} (a-A_j) X_j.$$
In Figure~\ref{fig:fBexper} we compare $\widehat{f}_{B,n}$  to a kernel density estimation of the increment distribution of dividing cells; and finally in Figure~\ref{fig:Bexper} we compare our estimator $\widehat{B}_{n}$ to $\widehat B_{n_d}=\f{\widehat f_{B,n_d}}{\int\limits_a^\infty \widehat f_{B,n_d} (s) ds \vee \varpi_2}$. If the results may be viewed as qualitatively in agreement, they are not fully satisfactory, and especially not comparable with the results obtained on simulated data. The reasons may be twofold. First, modeling errors: 
the incremental model is in good agreement with the data but cannot completely describe the full complexity of the biological process, including all potential fluctuations

 Second, our problem being severely-ill posed, even if the results on simulated data are of good quality, we did not take into account 
 the experimental noise, which is due to errors and/or imprecision of image analysis, leading to noise in size measurement and division timing, and to the sampling in time (which affects only the measurements on dividing cells), with a time step that is only ten times less than the cell generation time.
 This means that relatively important differences in the increment distribution of dividing cells / in the increment-structured division rate, may lead to differences on the size distribution of all cells which, for this level of noise, are not significant.

\section{Discussion} \label{sec:disc}

In this article, we have proposed an explicit reconstruction formula to estimate the increment-structured division rate of a population dividing by fission into two equal parts.  The formula may be easily generalized to other types of division kernels, as done for the size-structured equation in~\cite{BDE}. Based on this formula, we designed and implemented a numerical protocol, which we used to investigate numerically which rates of convergence could be expected. We finally tested the method on experimental data; though our results reveal qualitatively satisfactory, they did not reach the precision obtained for simulated data. This highlighted inherent difficulties of the problem, which deserve to be further investigated. These difficulties are linked to two sources of noise, not yet adressed neither theoretically nor numerically: modelling error and single-cell measurement errors. Investigating the influence of modelling error, \emph{e.g.} heterogeneity of cells with respect to growth rate or fragmentation kernel, or yet dependence of the division rate on another trait different from the increment,  could give first insights in this direction. To take into account single-cell measurement errors, we need to add a deconvolution problem to our noise model, assuming for instance that we measure  realizations of an i.i.d. sample $Y_i=X_i + \sigma \xi_i$ where $X_i$ are i.i.d. random variables distributed along $U_{B,x}$ and $\xi_i$ are i.i.d. normally distributed random variable, $\sigma$ being the level of noise.

Let us now discuss some possible variants of the method, and directions to prove estimation inequalities.

\subsubsection*{Possible variants of the method}

 Instead of estimating $\LB$ from $U_{B,x}$ by Lemma~\ref{lem:LB}, using $L^2 (dx)$ for $x<\bar x$ and $L^2(x^4 dx)$ for $x>\bar x,$ and only then take its Fourier transform, we could use the following lemma~\ref{lem:GB}: first define an estimate of the function $\Gamma_B$ from the sample, and then solve Equation~\eqref{eq:xi2xi}. It seems attractive since the Fourier transform then admits a simple and explicit definition from the sample $(X_1,\cdots X_n)$; but in practice, it appeared difficult to handle and would deserve a full study - in particular to determine, as for $\bar x$, a convenient threshold $\bar \xi$ to use one or the other of the spaces $L^2(x^p dx)$.

\begin{lemma}\label{lem:GB}
Let $\GB$ defined by~\eqref{eq:G_def}, and $\GB^*$ its Fourier transform. Then $\GB^*$ is solution of the following equation
\begin{equation} \label{eq:xi2xi}
\GB^*(2\xi)  =  \GB^*(\xi)  + \Gamma_B(\xi)
\end{equation}

with  $\Gamma_B$ defined by
\begin{equation}\label{eq:Gamma_def}
\Gamma_B(\xi) =  \bold{i}\tau\xi \int\limits_0^{\infty} x^2 e^{\bold{i} x \xi}  U_{B,x}(x) dx.
\end{equation}
For $\Gamma_B\in L^2 (\xi^pd\xi)$ with $p\geq 0,$ this functional equation admits a unique solution $\GB^*\in L^2(\xi^p d\xi).$
\end{lemma}
\begin{proof} At the view of \eqref{eq:sizediv2} and \eqref{eq:G_def} one immediately gets
$$
\GB^*(\xi) = \int\limits_0^\infty  \GB(x) e^{\bold{i}x\xi} dx = \int\limits_0^\infty 4 e^{\lambda G(x) } \LB(2x) e^{\bold{i}x\xi} dx = T_1 + T_2
$$
with 
$$
T_1 = \int\limits_0^\infty  e^{\lambda G(x) }  \big( \lambda U_{B,x}(x) + (gU_{B,x})'(x) \big) e^{\bold{i}x\xi} dx 
\quad \quad \textrm{and} \quad \quad
T_2 = \int\limits_0^\infty  e^{\lambda G(x) }  \LB(x) e^{\bold{i}x\xi} dx.
$$
Let us compute the first term:
\begin{align*}
T_1 & =  \int\limits_0^{\infty}  \lambda e^{\lambda G(x) + \bold{i} x \xi} U_{B,x}(x) dx + {\Big[ (gU_{B,x})(x) e^{\lambda G(x) + \bold{i} x \xi} \Big]_{x=0}^{\infty}} - \int\limits_0^\infty  (gU_{B,x})(x) \big( \tfrac{\lambda}{g(x)}+\bold{i}\xi \big) e^{\lambda G(x) + \bold{i} x \xi} dx \\
& = - \bold{i}\xi \int\limits_0^{\infty} e^{\lambda G(x) + \bold{i} x \xi}  g(x)U_{B,x}(x) dx.
\end{align*}
%
In order to treat the second term, {we assume the growth is exponential {\it i.e.} $g(x) = \tau x$}. In this special case we have $G(x)=\ln(x)/\tau$ and thus $G(2x)=G(2)+G(x)$.
Then
\begin{align*}
T_2  = 2 \int\limits_0^\infty  e^{\lambda G(2x) }  \LB(2x) e^{\bold{i}2x\xi} dx =  \tfrac{1}{2} e^{\lambda G(2)} \int\limits_0^\infty  4 e^{\lambda G(x) }  \LB(2x) e^{\bold{i}x (2\xi)} dx = \GB^*(2\xi)
\end{align*}
using at last $e^{\lambda G(2)}=2$ since $\lambda = \tau$. Gathering the two terms we obtain \eqref{eq:xi2xi}.
The existence and uniqueness in $L^2(x^pdx)$ for $p\neq -1$ directly follows from~\cite{DPZ} Proposition A.1.applied to the equation transformed for $u(\xi)=\f{1}{\xi^2} \GB^*(\xi),$ which satisfies 
$$4u(2\xi)-u(\xi)=\f{\Gamma_B(\xi)}{\xi^2},$$
which admits a unique solution for $\f{\Gamma_B}{x^2}\in L^2 (\xi^q d\xi)$ for $q\neq 3,$ which is equivalent to $\Gamma_B \in L^2(\xi^p d\xi)$ with $p\neq -1.$ Since we have $\GB^*(0)=\int\limits_0^\infty 4y\LB(2y)dy>0$ (it represents the average size of dividing cells), we are only interested in $p\geq 0.$
\end{proof}

The  function  $\Gamma_B$ can be easily estimated by $\widehat \Gamma_n(\xi)  = \frac{\bold{i}\tau\xi }{n} \sum_{j=1}^n X_j^2 e^{\bold{i} X_j \xi}$ truncated for $\vert \xi\vert \leq \overline{\xi}$. Then, in the same spirit as for the solution of~\eqref{eq:sizediv2}, we could solve Equation~\eqref{eq:xi2xi} by writing
$$\GB^*(\xi)=-\sum\limits_{k=0}^\infty \Gamma_B (2^k \xi),\qquad \implies \qquad \widehat \GG_n = -\sum\limits_{k=0}^N \widehat \Gamma_n (2^k\xi),\qquad \forall \xi\in \R,$$
but numerically it happens not to give better  estimation results and to be less easily compared with the original function in the space state, so that we preferred the  method explained above.

\

Other variants and improvements of the method would be to constraint the space of solutions for $\widehat f_{n,h}$ to the space of probability measures. This is done manually in our procedure, taking the positive and real part of the estimated inverse Fourier transform (see Section~\ref{sec:estim:fB}), but constrained optimization or projection on a finite-dimension space approximating probability measures could improve the results. This will be investigated in future work.

\subsection*{Estimation inequalities}

To prove estimation inequalities, several difficulties appear, which give a roadmap for further investigation of the problem.

First, we have to prove that the denominator of our ratios in our inverse Fourier transforms, namely $\GB^*,$ does not vanish. This has been done for related problems (estimation of the fragmentation kernel) in two recent papers, by using complex analysis methods (Lemma 1.iii in~\cite{hoang:hal-01623403}, Theorem 2.i in~\cite{doumic:hal-01501811}). In~\cite{doumic:hal-01501811}, it was the central and most technical point of the study. Here however, the proof carried out in~\cite{hoang:hal-01623403} based on the argument principal could probably be adapted.

Second, the deconvolution problem~\eqref{eq:convol1} appears as a deconvolution problem with unknown "noise", since $\GB$ plays the role of the noise. The difficulty is thus to investigate whether $\GB$ is  {\it ordinary smooth} or 
{\it super smooth}, with the following definitions:
\begin{itemize}
\item Ordinary smooth of order $\beta$: $c_1|t|^{-\beta} \leq |\GB^*(t)| \leq c_2 |t|^{-\beta}$ for any $|t|\geq M$, for positive constants.
\item Super smooth of order $\beta$: $c_1|t|^{\gamma_1}e^{-c_0 |t|^\beta} \leq |\GB^*(t)| \leq c_2 |t|^{\gamma_2}e^{-c_0 |t|^\beta} $ for any $|t|\geq M$, for positive constants.
\end{itemize}
The smoother the "noise", the more ill-posed the problem. Once  a given order of magnitude for the decay of $\GB^*$ assumed, speeds of convergence and orders of magnitude for the choice of the regularization constant $h_3$ may be classically obtained, see for instance~\cite{Baumeister} (ch.4, Section 4.2.2). However, assuming a given decay for $\GB^*$ means that we assume a certain degree of regularity - and no more - for $\GB$, \emph{i.e.} for $\LB,$ \emph{i.e.} for the unknown $B$ itself.  If regularity results exist and can be extended to higher regularity by the chain rule for our equation, see e.g.~\cite{D,gabriel:hal-01742140}, the reverse is false - results such as: if $B$ is not derivable, then $U_{B,x}$ cannot be twice derivable. This shows the importance of  designing \emph{a posteriori} and adaptive methods.

\

{\bf Acknowledgments.} 
A.O. was on leave ("d\'el\'egation") at the french National Research Centre for Science (CNRS) during the 
finalization of this work. M.D. has been partly supported by the ERC Starting Grant SKIPPER$^{AD}$ (number 306321). We thank Albert Cohen, Marc Hoffmann and Beno\^it Perthame  for very fruitful discussions.

\section{Appendix: Figures and tables}

Figures~\ref{fig:etape1}, \ref{fig:etape2}, \ref{fig:etape3} and \ref{fig:etape4} use the notation $\mathcal F$ for the Fourier transform and $\mathcal F^{-1}_h$ for the following operator. For a suitable function $f$,
$$
\mathcal F^{-1}_h(f)(a) = \frac{1}{2\pi} \int\limits_{-1/h}^{1/h} f(\xi) e^{- i a\xi} d\xi.
$$



\begin{figure}[h]
\centering
\includegraphics[width=0.5\textwidth]{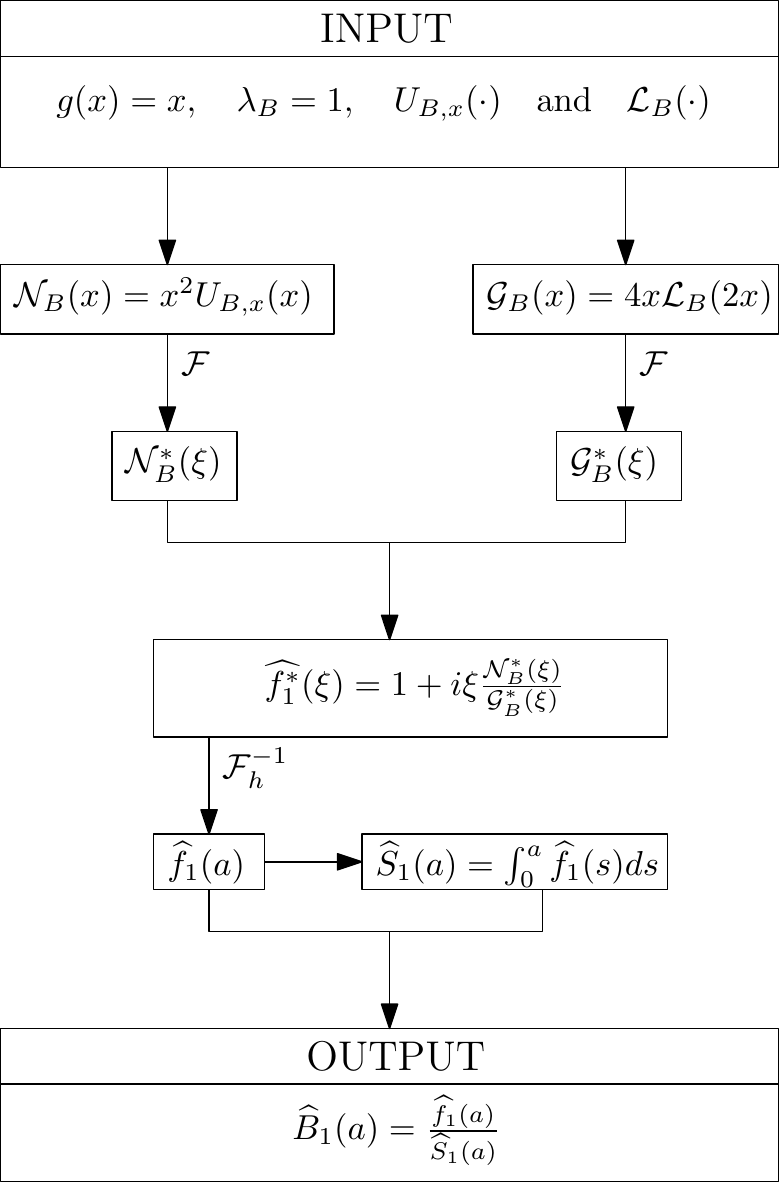} 
\caption{Protocol 1 -- Reconstruction of $B$ when both $U_{B,x}$ and $\LB$ are (almost) exactly known. The oracle choice for $h$ gives us the value 1/4.75.
}%
\label{fig:etape1}%
\end{figure}

\begin{figure}[h]
\centering
\includegraphics[width=0.70\textwidth]{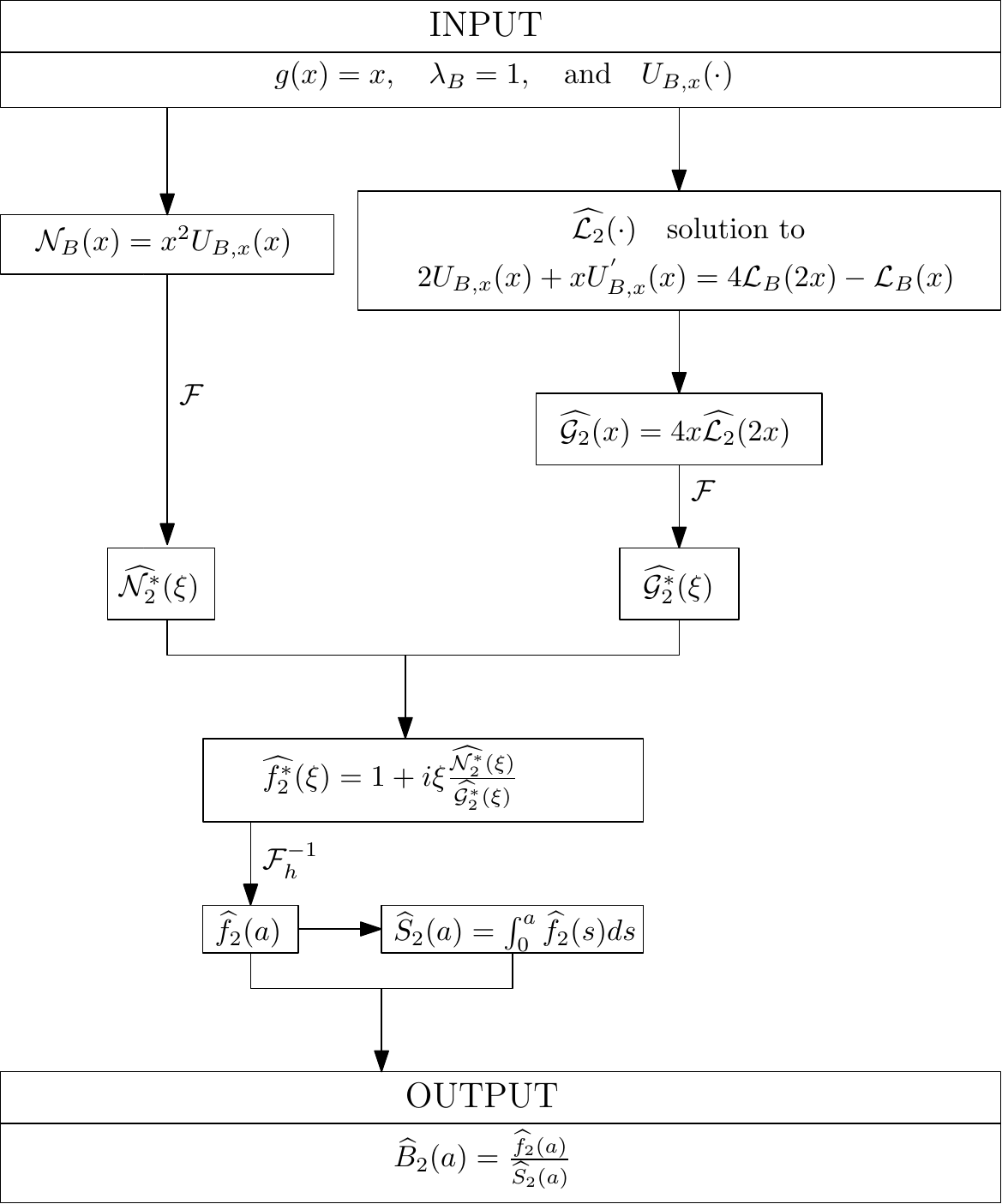} 
\caption{Protocol 2 -- Reconstruction of $B$ when $U_{B,x}$ is (almost) exactly known but not $\LB$.  The oracle choice for $h$ gives us the value 1/5.}%
\label{fig:etape2}%
\end{figure}

\begin{figure}[h]
    \begin{subfigure}[b]{0.5\textwidth} 
        \centering \includegraphics[width=\textwidth]{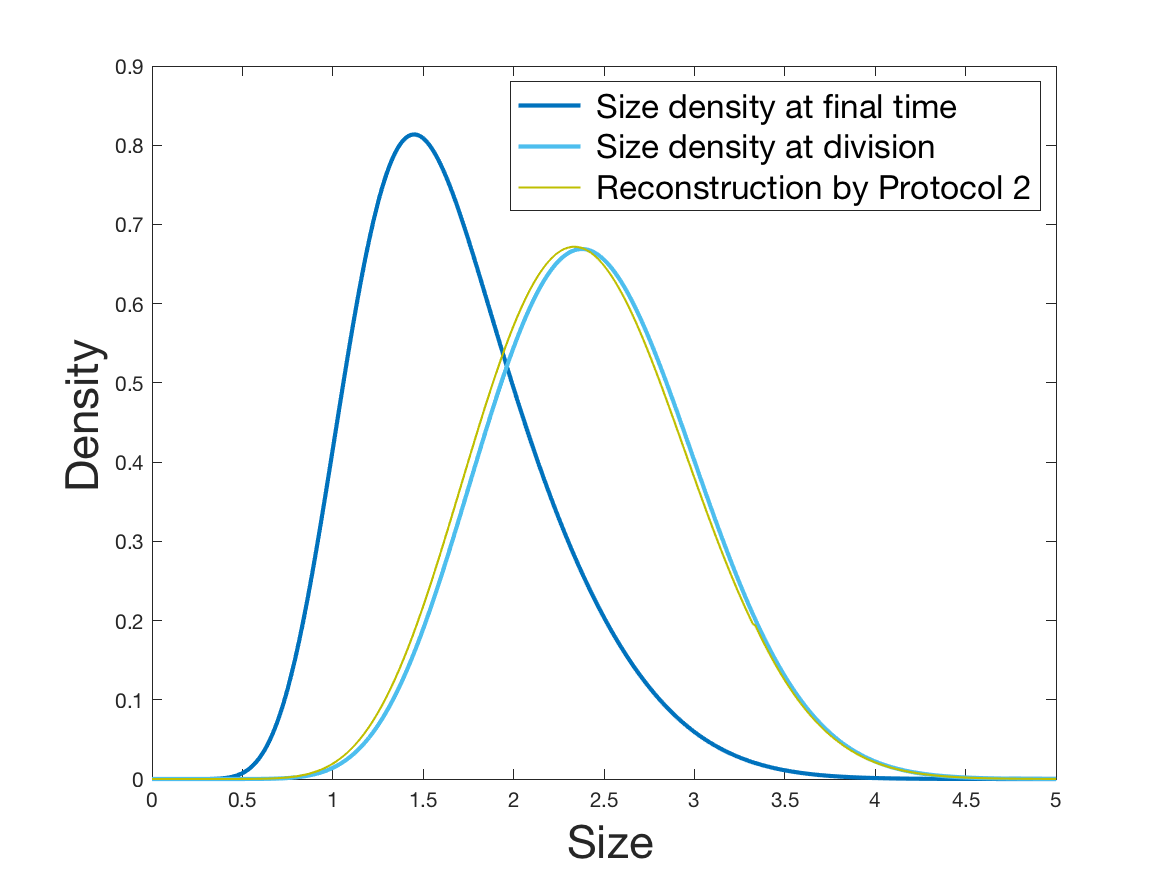} 
        \caption{$U_{B,x}$, $\mathcal L_B$ and $\widehat{\mathcal L}_2$ in function of $x$}
    \end{subfigure}
    \begin{subfigure}[b]{0.5\textwidth} 
        \centering \includegraphics[width=\textwidth]{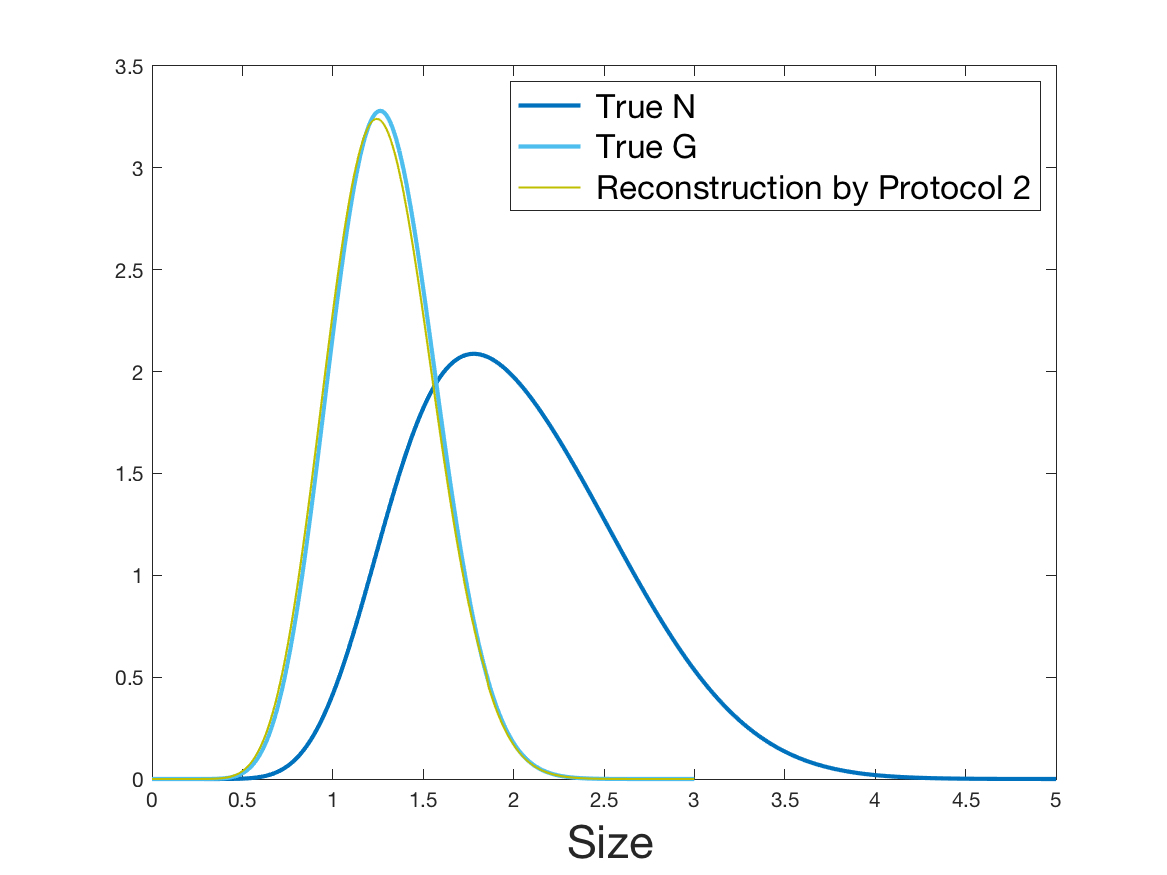} 
        \caption{$\mathcal N_B$, $\mathcal G_B$ and $\widehat{\mathcal G}_2$ in function of $x$}
    \end{subfigure}
    
    \begin{subfigure}[b]{0.5\textwidth} 
        \centering \includegraphics[width=\textwidth]{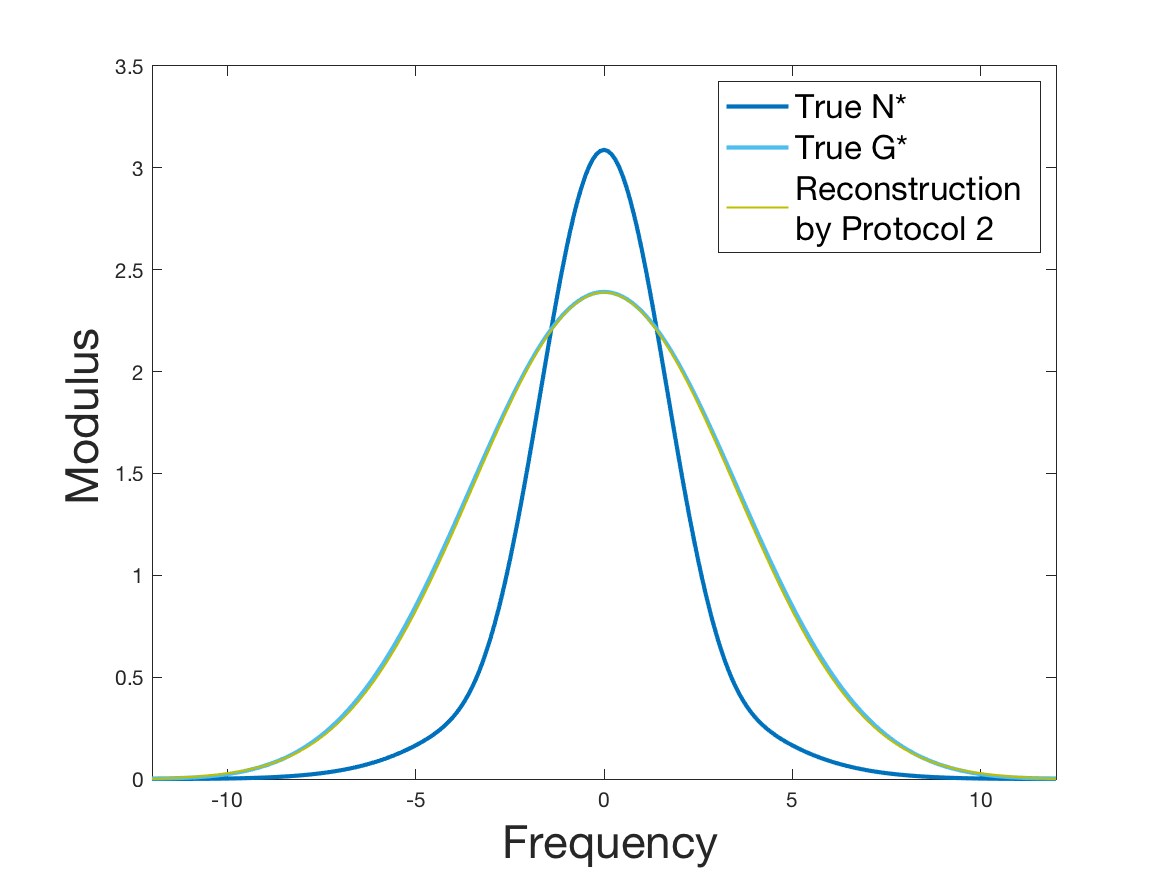} 
        \caption{$|\mathcal N_B^*|$, $|\mathcal G_B^*|$ and $|\widehat{\mathcal G}_2^*|$ in function of $\xi$}
    \end{subfigure}
    \begin{subfigure}[b]{0.5\textwidth}  
        \centering \includegraphics[width=\textwidth]{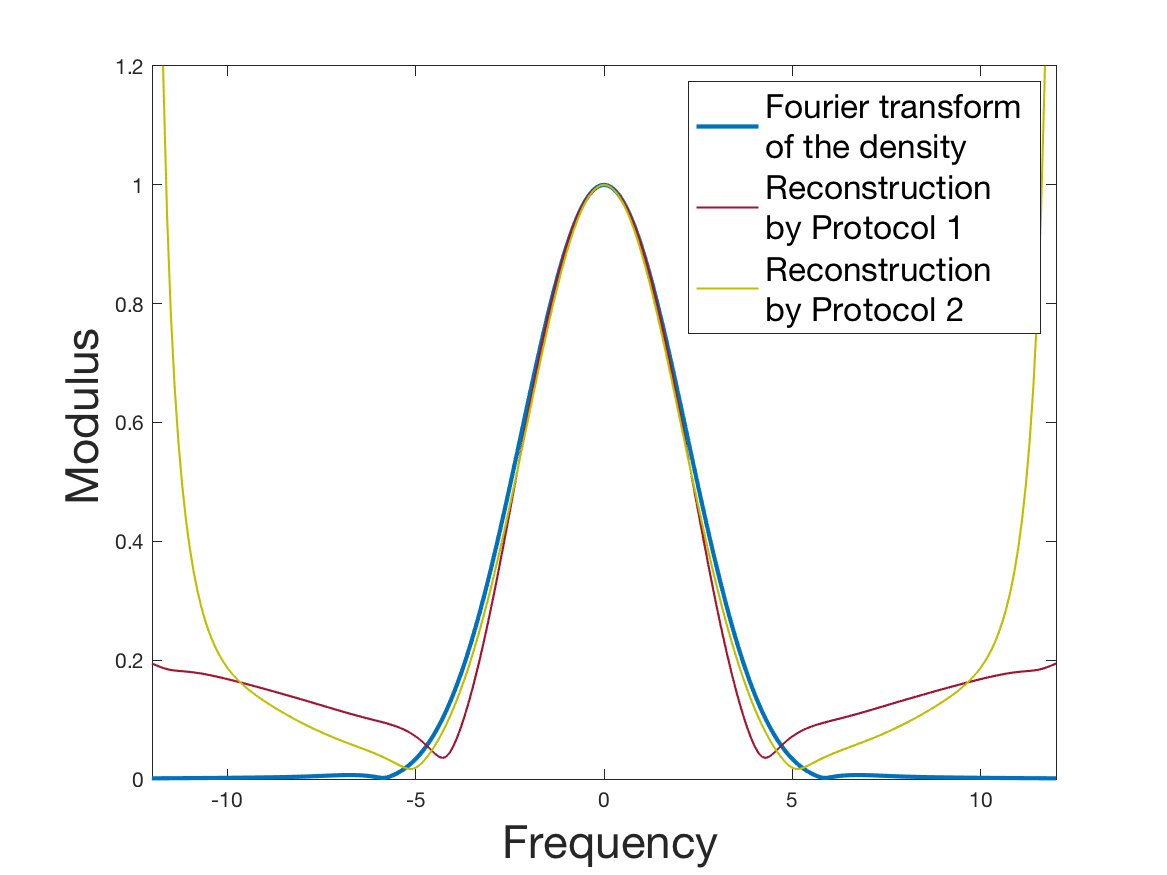} 
        \caption{$|f_B^*|$, $|\widehat{f}_1^*|$ and $|\widehat{f}_2^*|$ in function of $\xi$}
    \end{subfigure}
    
    \begin{subfigure}[b]{0.5\textwidth} 
        \centering \includegraphics[width=\textwidth]{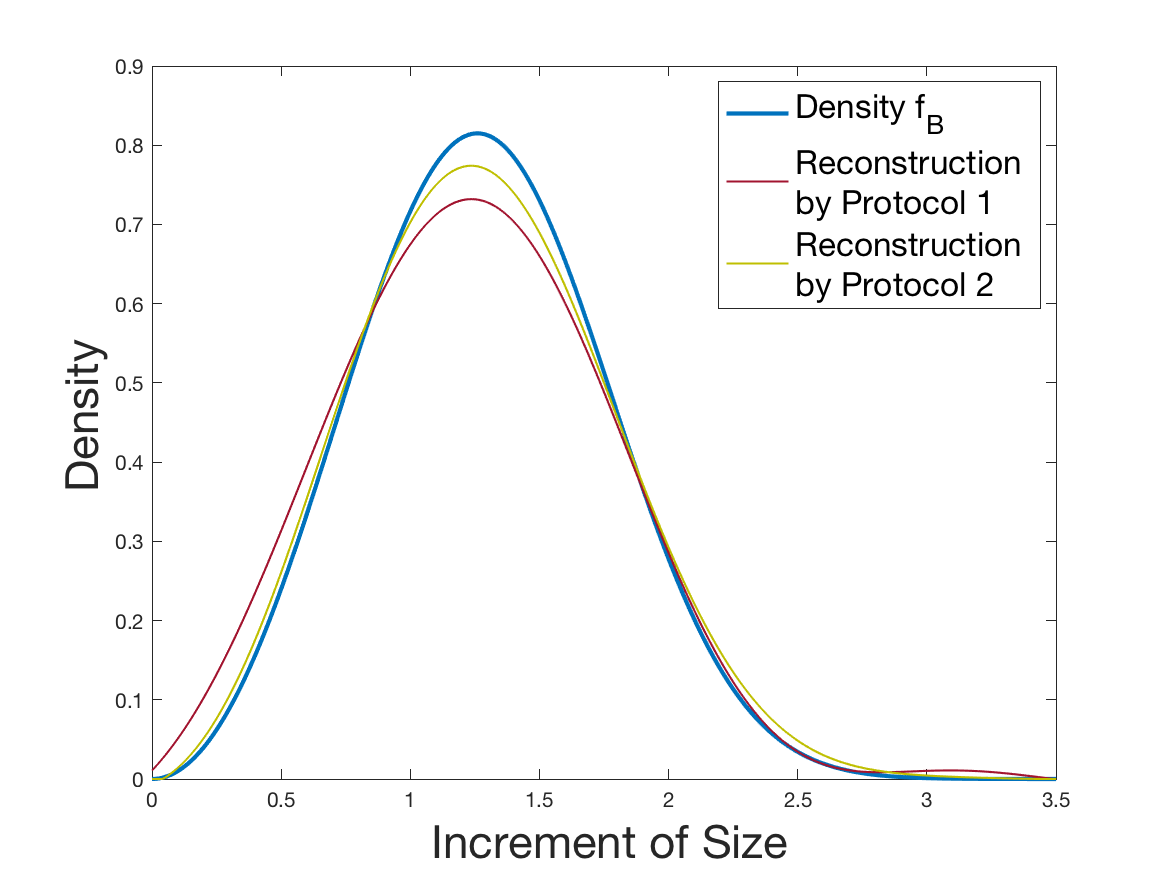} 
        \caption{$f_B$, $\widehat{f}_1$ and $\widehat{f}_2$ in function of $a$}
    \end{subfigure}
    \begin{subfigure}[b]{0.5\textwidth} 
        \centering \includegraphics[width=\textwidth]{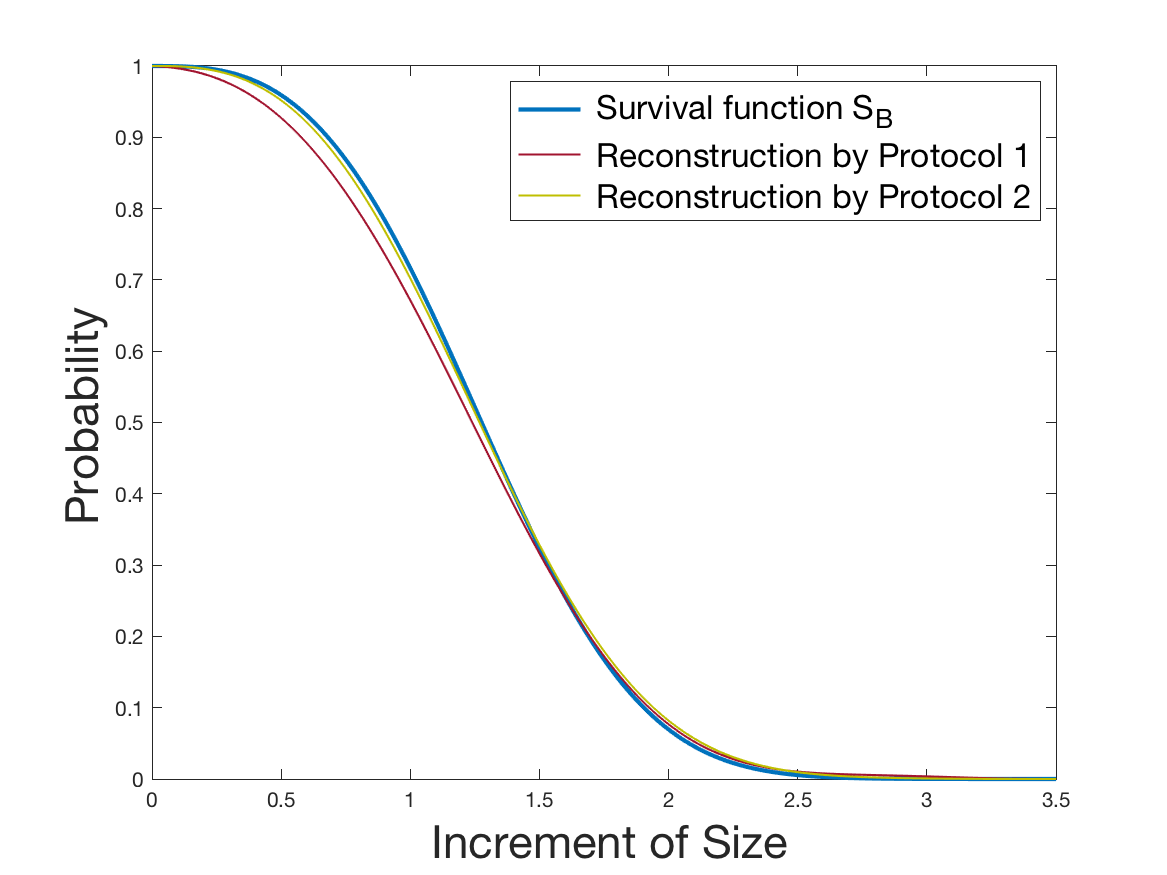}
        \caption{$S_B$, $\widehat{S}_1$ and $\widehat{S}_2$ in function of $a$}
    \end{subfigure}

\end{figure}

\begin{figure}[t] \ContinuedFloat 
    \begin{subfigure}[b]{0.5\textwidth} 
        \includegraphics[width=\textwidth]{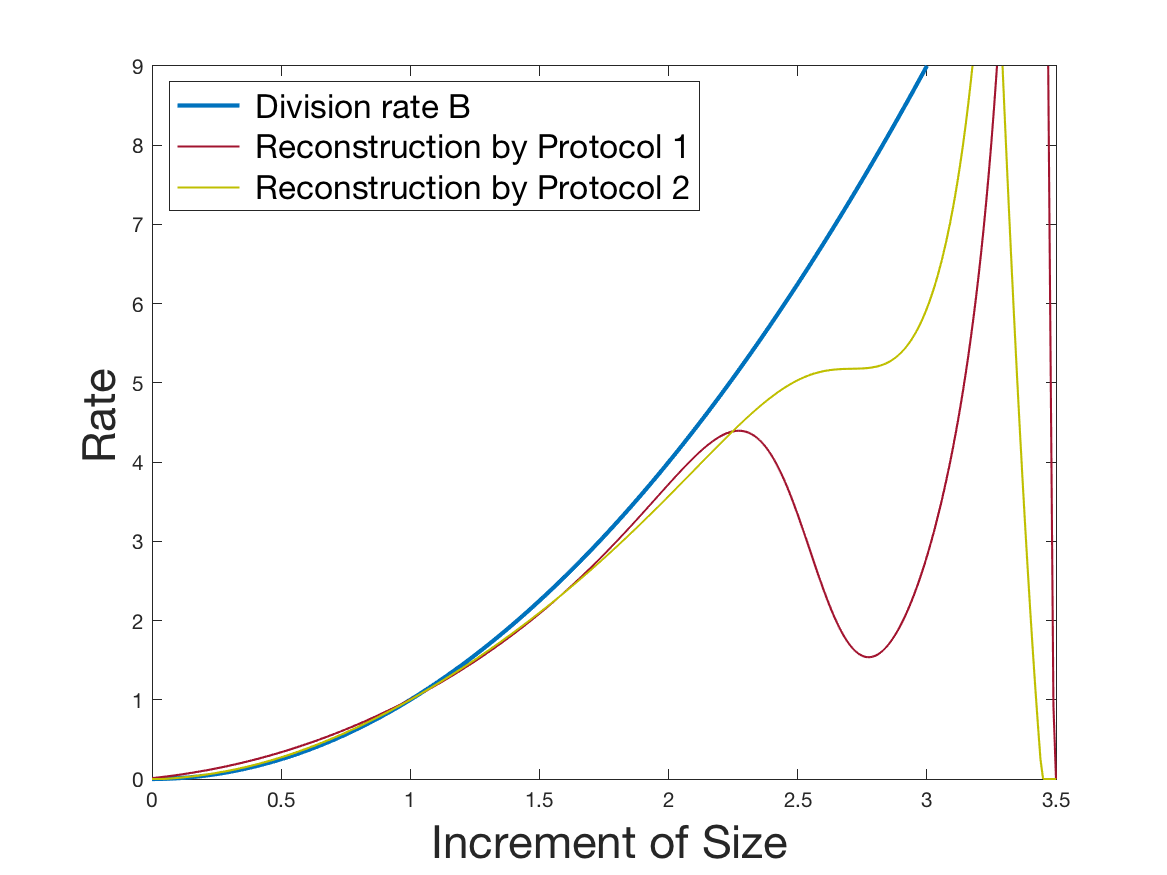} 
        \caption{$B$, $\widehat{B}_1$ and $\widehat{B}_2$ in function of $a$}
    \end{subfigure}
\caption{Results of Protocols 1 and 2. ($x$ stands for size, $\xi$ for frequency and $a$ for increment of size)}%
\label{fig:resultats1et2}%
\end{figure}


\begin{table}[h!]
\centering
\begin{tabular}{|c|c|c|c|c|c|c|}
\hline
\textit{Reconstruction of} & $\mathcal L_B$ & $\mathcal G_B$ & $\mathcal G^*_B$ & $f^*_B$ & $f_B$ & $S_B$  \\ \hline
Numerical & [0;6] & [0;6] & [-50;50] & $[\tfrac{-1}{4.75};\tfrac{1}{4.75}]$ & [0;5] & [0;5]  \\ 
sampling & $\Delta x = \tfrac{6}{500}$ & $\Delta x = \tfrac{6}{500}$ & $\Delta \xi = 0.05$ & $\Delta \xi = 0.05$ & $\Delta a = 0.01$ & $\Delta a = 0.01$   \\ \hline
\textbf{Protocol 1}        & -              & -              & -                &   0.1062      &    0.1043   &   0.0395     \\ \hline
\textbf{Protocol 2}        &       0.0478         &      0.0417          &      0.0417            &   0.0470      &   0.0482    &  0.0149       \\ \hline
\end{tabular}
\caption{Errors of Protocols 1 and 2 for the intermediate steps.}
\label{tab:resultats1et2}
\end{table}

\begin{table}[h!]
\centering
\begin{tabular}{|c|c|c|}
\hline
\textit{Reconstruction of} & $B$ & $B$ \\ \hline
Numerical & [0;2]& [0;2.5 ]\\ 
sampling & $\Delta a = 0.01$ & $\Delta a = 0.01$ \\ \hline
\textbf{Protocol 1}        &  0.0730 & 0.2065 \\ \hline
\textbf{Protocol 2}        &  0.0849 &  0.1321   \\ \hline
\end{tabular}
\caption{Errors of Protocols 1 and 2 for $B$ in function of the numerical sampling.}
\label{tab:resultats1et2_B}
\end{table}

\begin{figure}[h]
\centering
\includegraphics[width=0.80\textwidth]{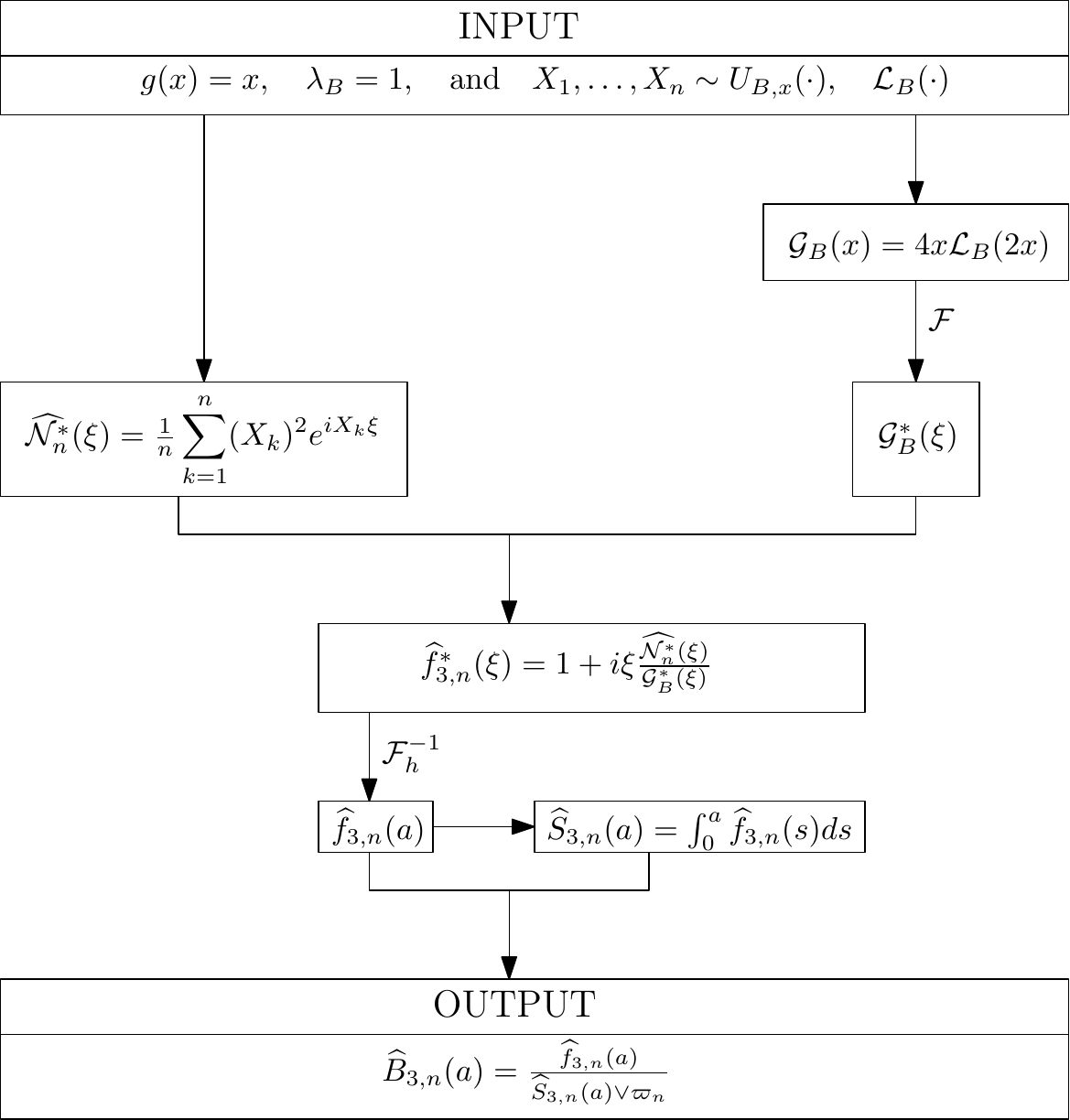} 
\caption{Protocol 3 -- Reconstruction of $B$ when $U_{B,x}$ is reconstructed from $X_1,\ldots,X_n$ i.i.d. $\sim$ $U_{B,x}$ but $\LB$ is (almost) exactly known.  The oracle choice for $h_3$ gives us values that range between $1/3.25$ for $n=500$ and $1/4.75$ for $n = 50~000$. We set $\varpi_n = 1/n$.}%
\label{fig:etape3}%
\end{figure}

\begin{figure}[h]
    \begin{subfigure}[b]{0.5\textwidth} 
        \centering \includegraphics[width=\textwidth]{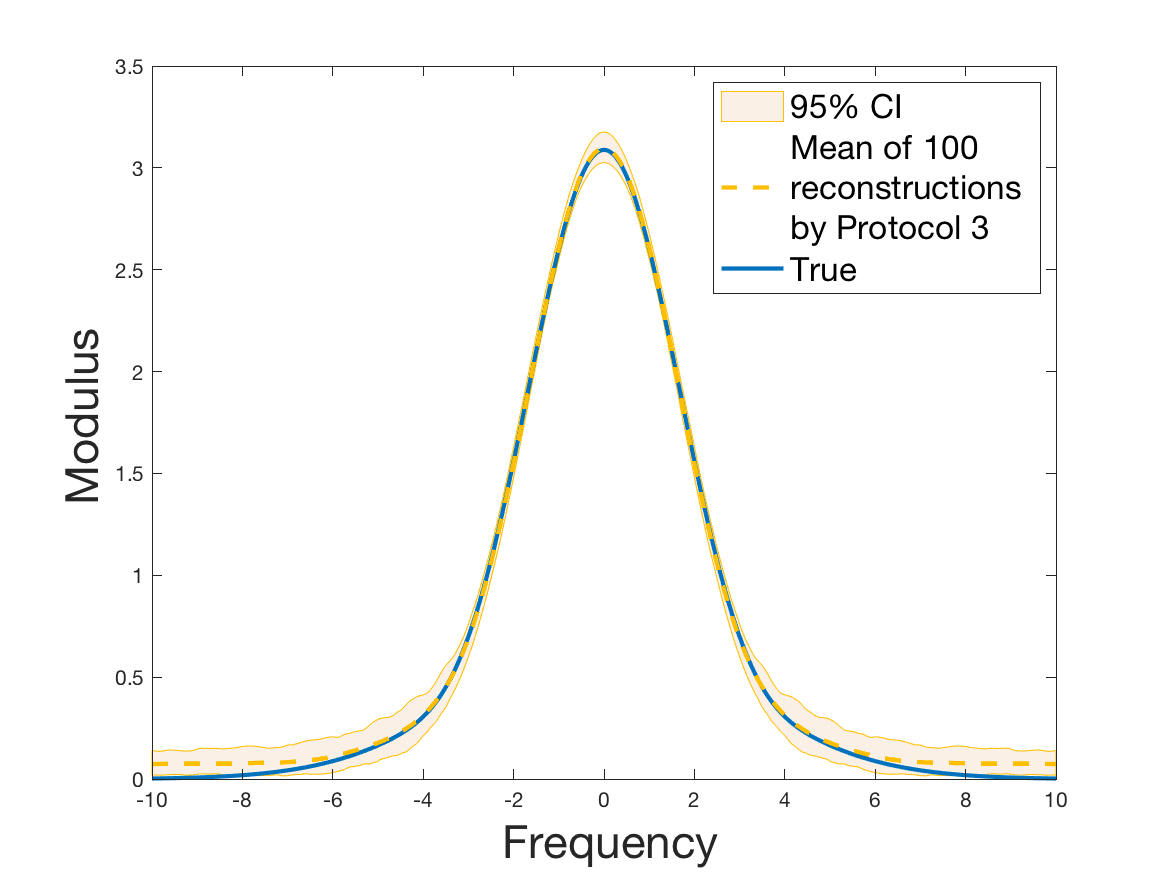}  
        \caption{$|\mathcal N^*_B|$, $|\widehat{\mathcal N^*_{n}}|$ in function of $\xi$}
    \end{subfigure}
    \begin{subfigure}[b]{0.5\textwidth} 
        \centering \includegraphics[width=\textwidth]{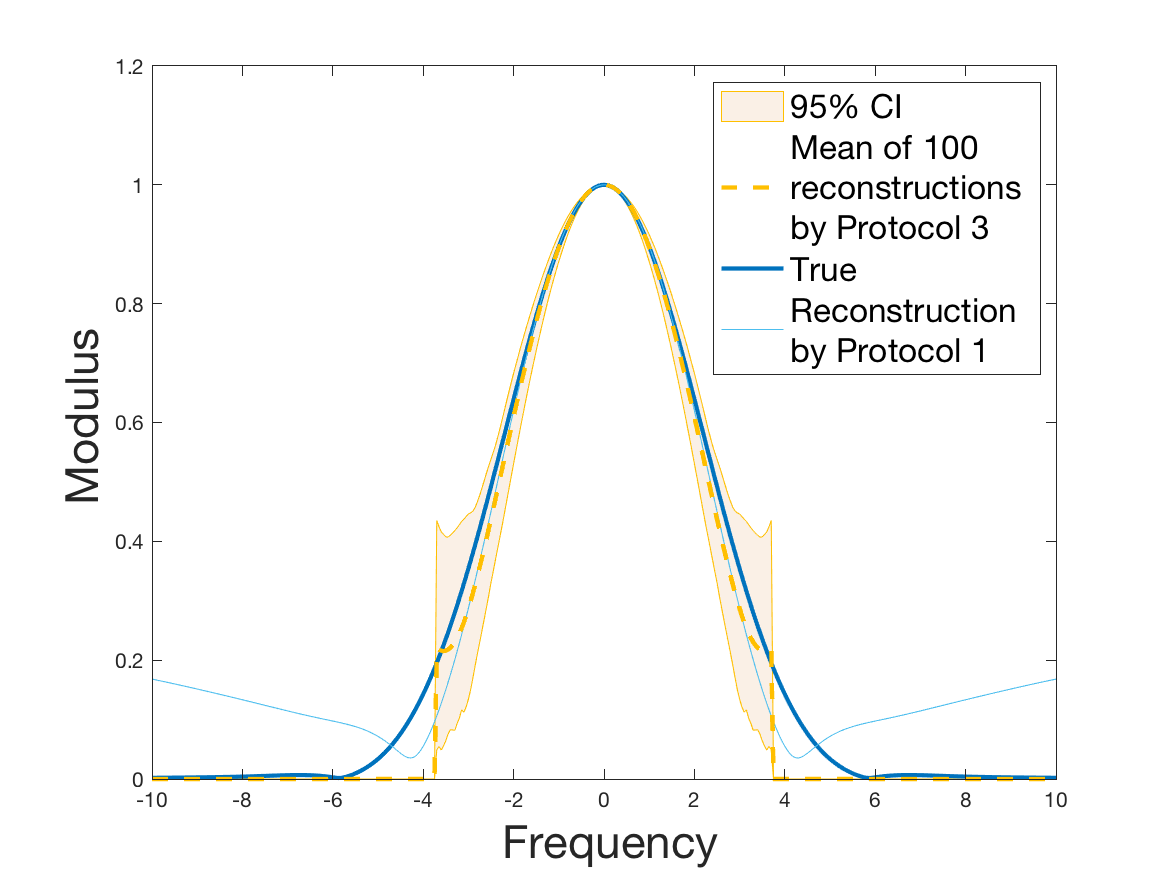}
        \caption{$|f_B^*|$, $|\widehat{f}^*_1|$ and $|\widehat{f}^*_{3,n}|$ in function of $\xi$}
    \end{subfigure}
    
    \begin{subfigure}[b]{0.5\textwidth} 
        \centering \includegraphics[width=\textwidth]{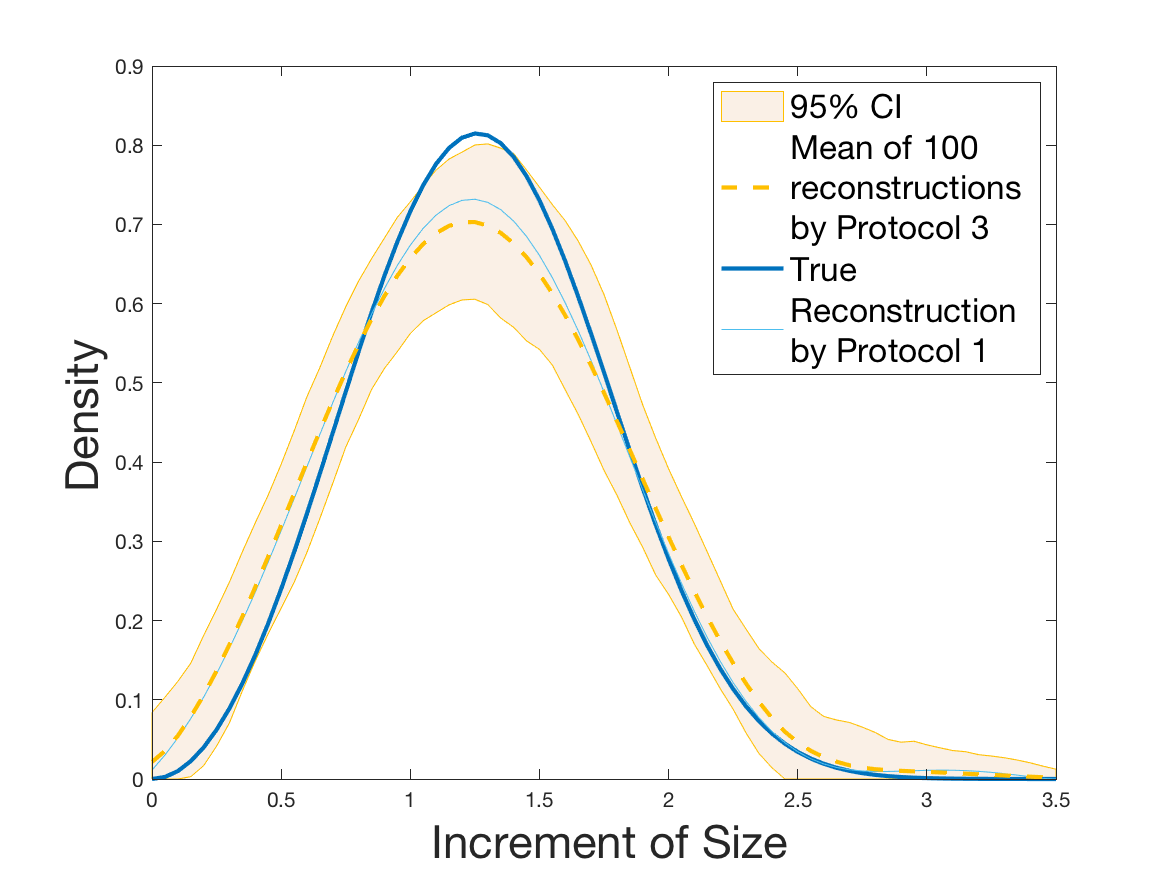}  
        \caption{$f_B$, $\widehat{f}_1$ and $\widehat{f}_{3,n}$ in function of $a$}
    \end{subfigure}
    \begin{subfigure}[b]{0.5\textwidth} 
        \centering \includegraphics[width=\textwidth]{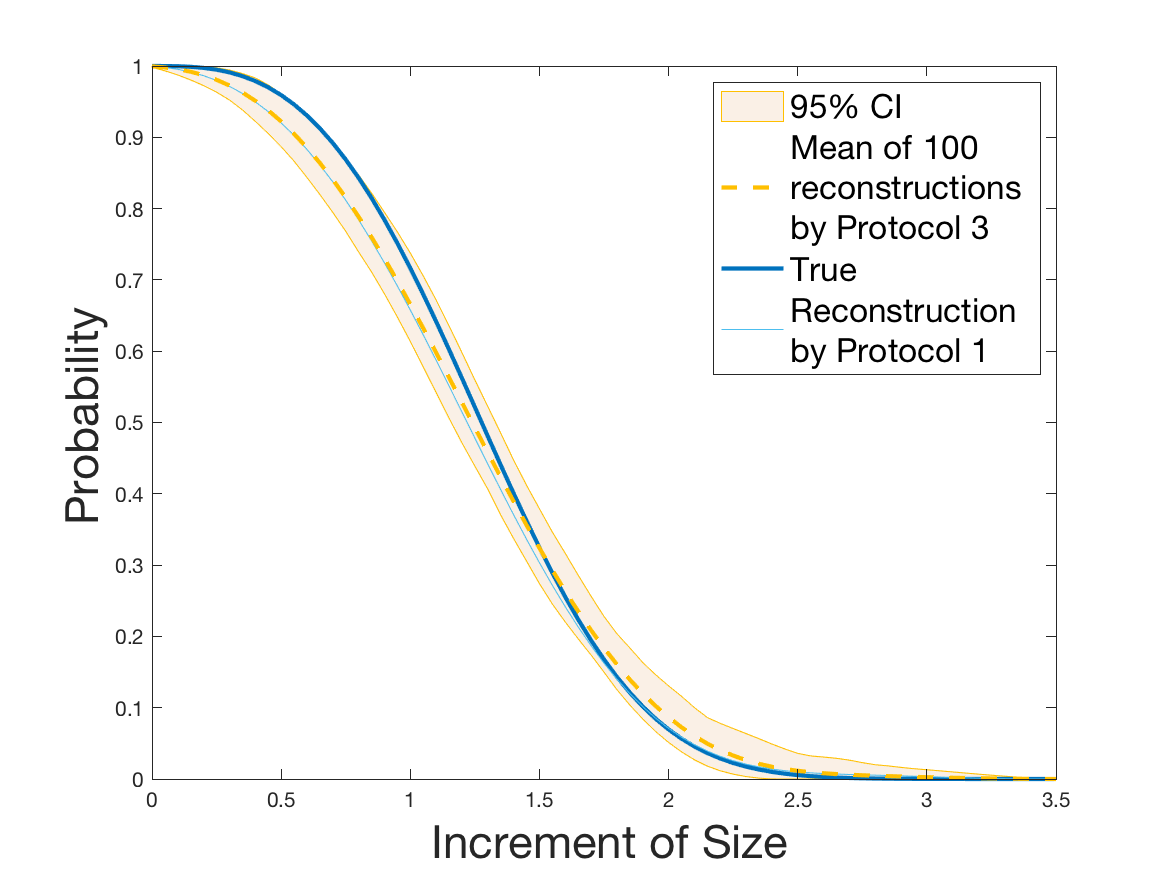}
        \caption{$S_B$, $\widehat{S}_1$ and $\widehat{S}_{3,n}$ in function of $a$}
    \end{subfigure}
 
     \begin{subfigure}[b]{0.5\textwidth} 
        \centering \includegraphics[width=\textwidth]{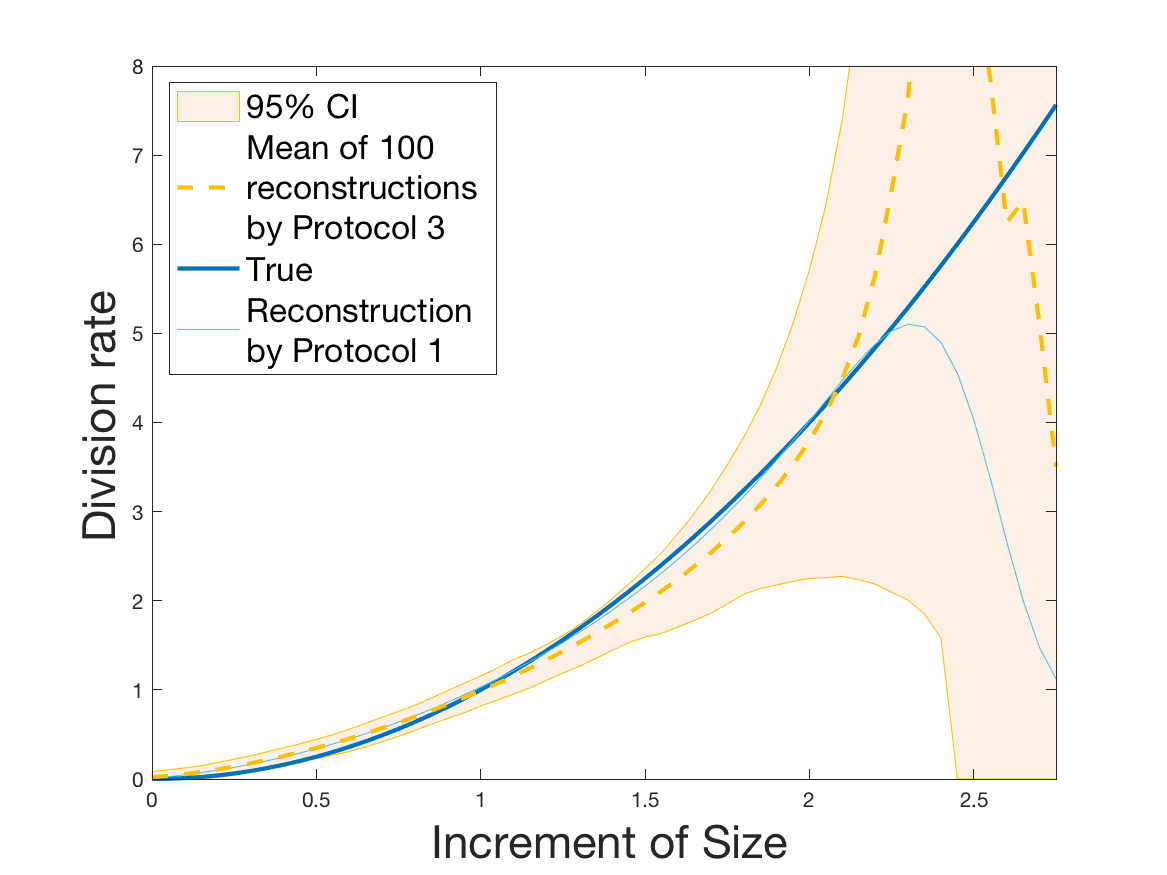} 
        \caption{$B$, $\widehat{B}_1$ and $\widehat{B}_{3,n}$ in function of $a$}
    \end{subfigure}
 
\caption{Results of Protocol 3 for $n=2000$ and $M=100$ Monte Carlo samples. ($x$ stands for size, $\xi$ for frequency and $a$ for increment of size)}%
\label{fig:resultats3_IC}%
\end{figure}

%
%

\begin{figure}[h] 
\centering
\includegraphics[width=0.80\textwidth]{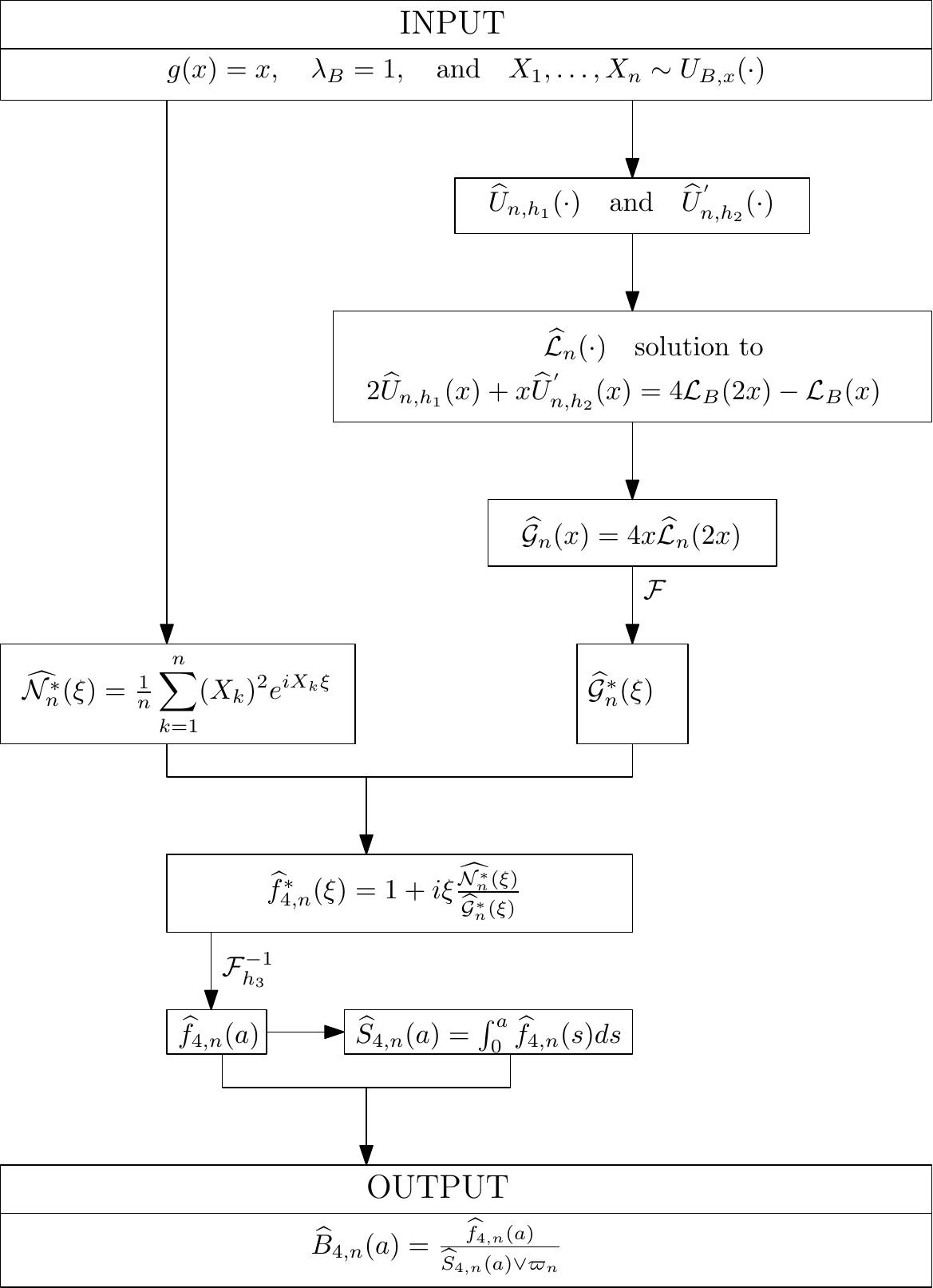} 
\caption{Protocol 4 -- Reconstruction of $B$ when both $U_{B,x}$ and $\LB$ are reconstructed from $X_1,\ldots,X_n$ i.i.d. $\sim$ $U_{B,x}$. The parameter $h_1$ is automatically chosen by the kernel smoothing function 
{\tt ksdensity}; $h_2$ is deduced from $h_1$. The oracle choice for $h_3$ gives us values that range between $1/3.25$ for $n=500$ and $1/4.5$ for $n = 50~000$. We set $\varpi_n = 1/n$.}%
\label{fig:etape4}%
\end{figure}


\begin{figure}[h]

    \begin{subfigure}[b]{0.5\textwidth} 
        \centering \includegraphics[width=\textwidth]{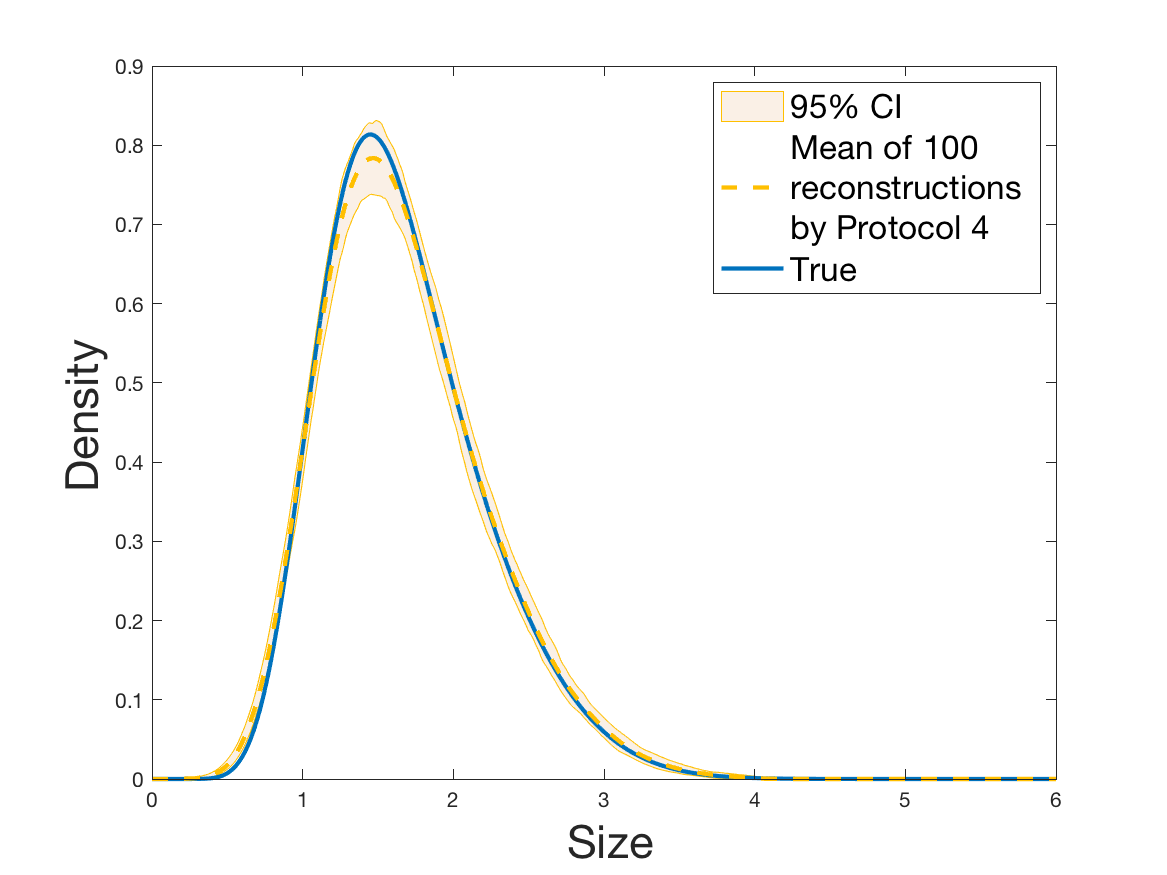}
        \caption{$U_{B,x}$ and $\widehat{U}_{n,h}$}
    \end{subfigure}
    \begin{subfigure}[b]{0.5\textwidth} 
        \centering \includegraphics[width=\textwidth]{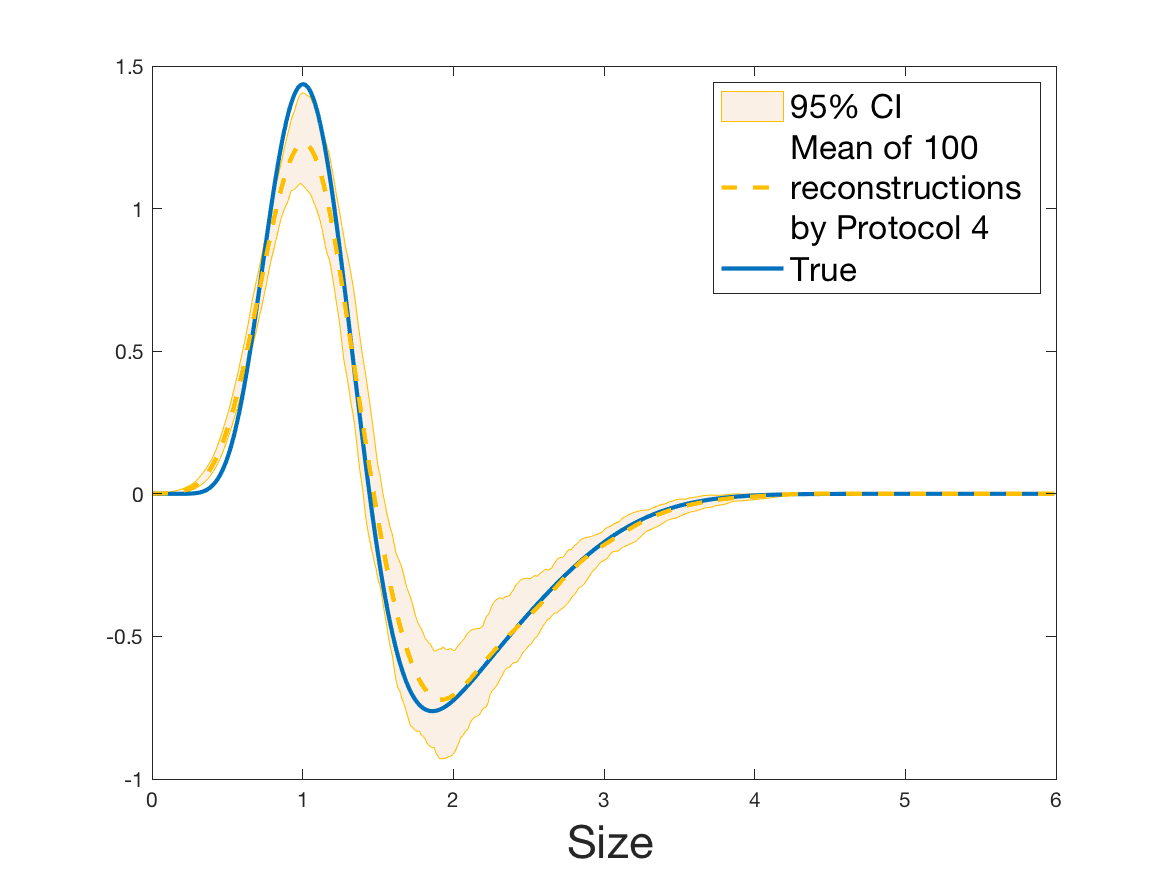}
        \caption{$U^{'}_{B,x}$ and $\widehat{U}^{'}_{n,h^{'}}$}
    \end{subfigure}

    \begin{subfigure}[b]{0.5\textwidth} 
        \centering \includegraphics[width=\textwidth]{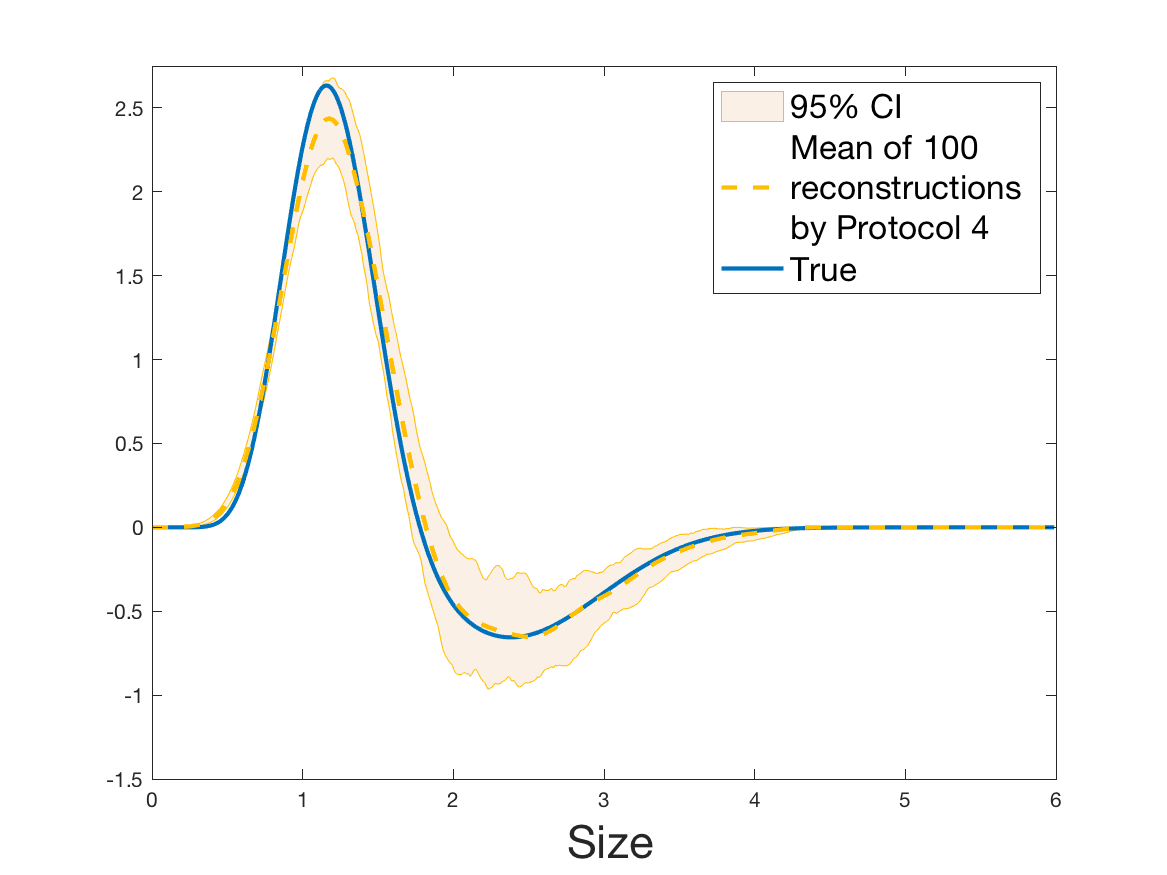}
        \caption{$2 U_{B,x} + x U^{'}_{B,x}$ and $2 \widehat{U}_{n,h} + \widehat{U}^{'}_{n,h^{'}}$}
    \end{subfigure}
    \begin{subfigure}[b]{0.5\textwidth} 
        \centering \includegraphics[width=\textwidth]{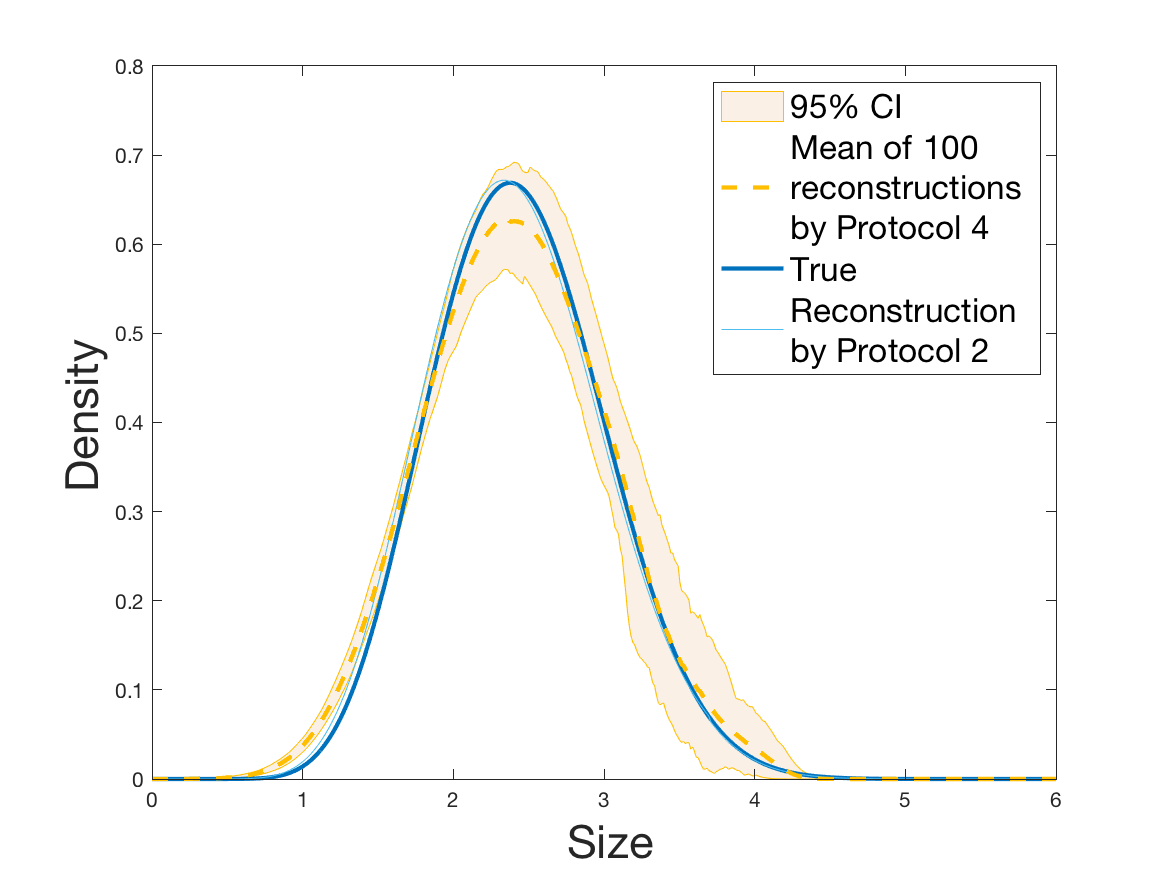} 
        \caption{$\mathcal L_B$ and $\widehat{\mathcal L}_n$ in function of $x$}
    \end{subfigure}
    
    \begin{subfigure}[b]{0.5\textwidth} 
        \centering \includegraphics[width=\textwidth]{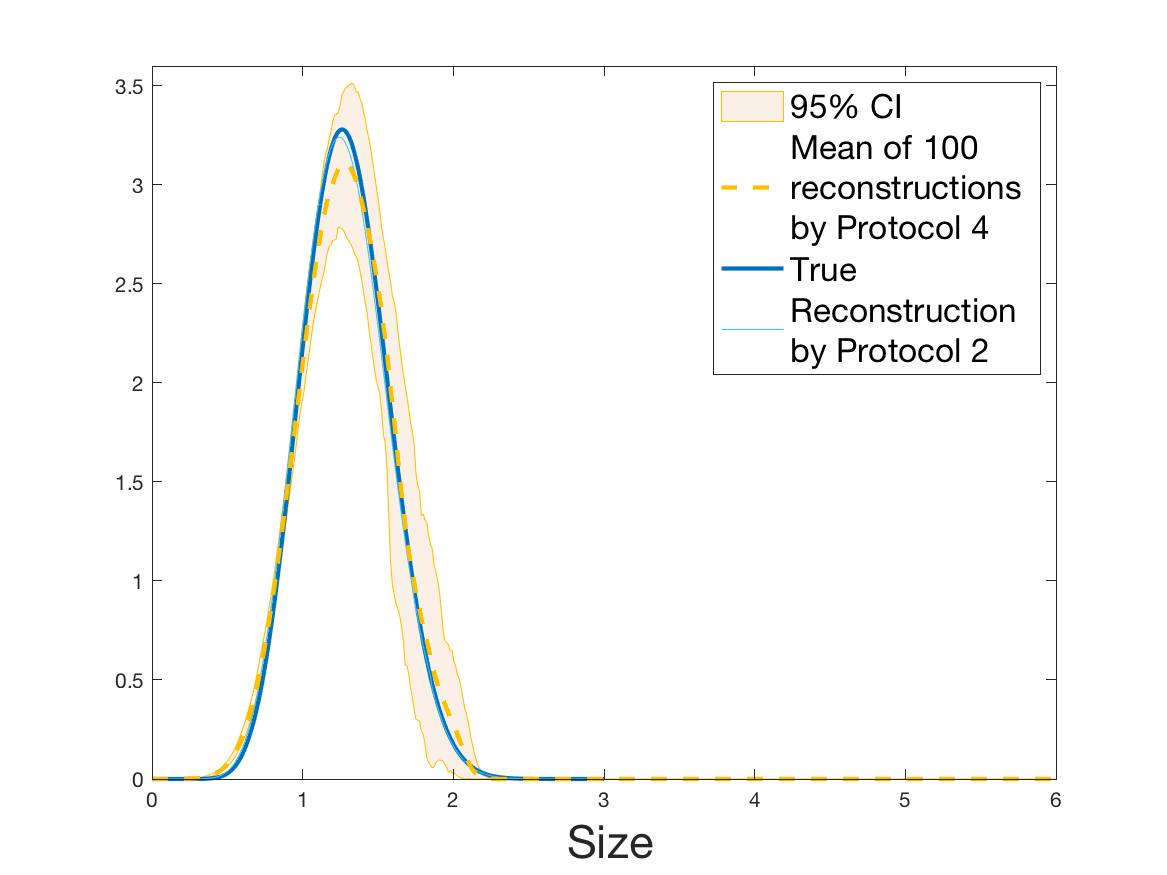}
        \caption{$\mathcal G_B$ and $\widehat{\mathcal G}_n$ in function of $x$}
    \end{subfigure}
    \begin{subfigure}[b]{0.5\textwidth} 
        \centering \includegraphics[width=\textwidth]{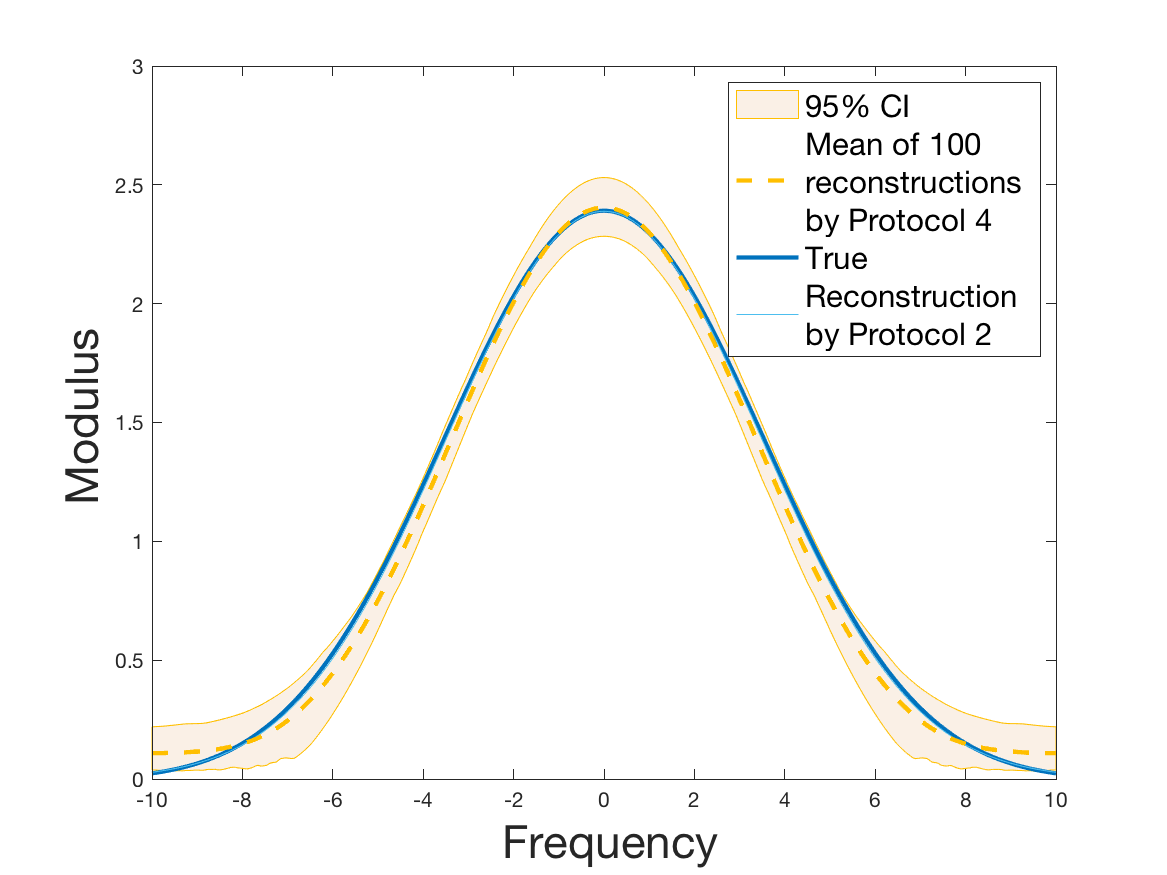}
        \caption{$|\mathcal G^*_B|$, $|\widehat{\mathcal G}^*_{2}|$ and $|\widehat{\mathcal G}^*_{n}|$ in function of $\xi$}
    \end{subfigure}

\end{figure}

\begin{figure}[htb]\ContinuedFloat 
    
    \begin{subfigure}[b]{0.5\textwidth} 
        \centering \includegraphics[width=\textwidth]{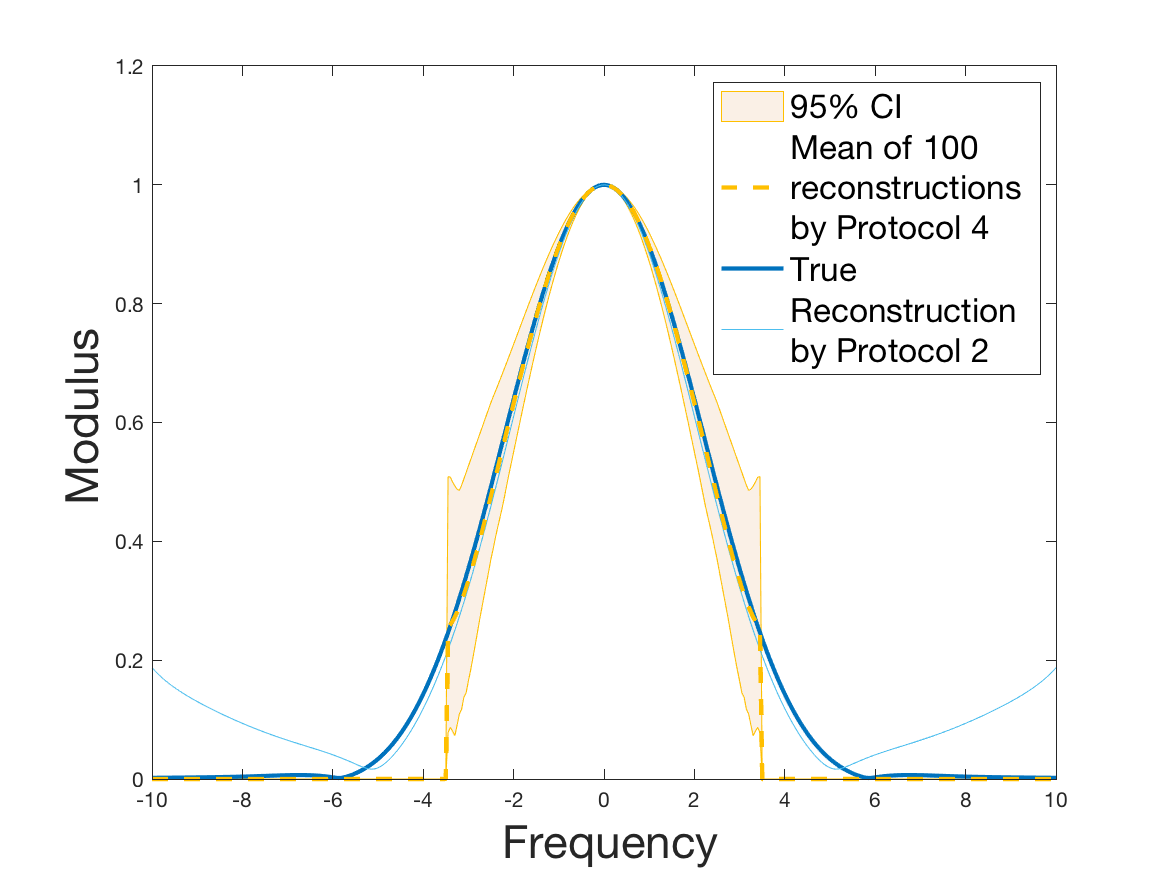}
        \caption{$|f_B^*|$, $|\widehat{f}^*_2|$ and $|\widehat{f}^*_{4,n}|$ in function of $\xi$}
    \end{subfigure}
    \begin{subfigure}[b]{0.5\textwidth} 
        \centering \includegraphics[width=\textwidth]{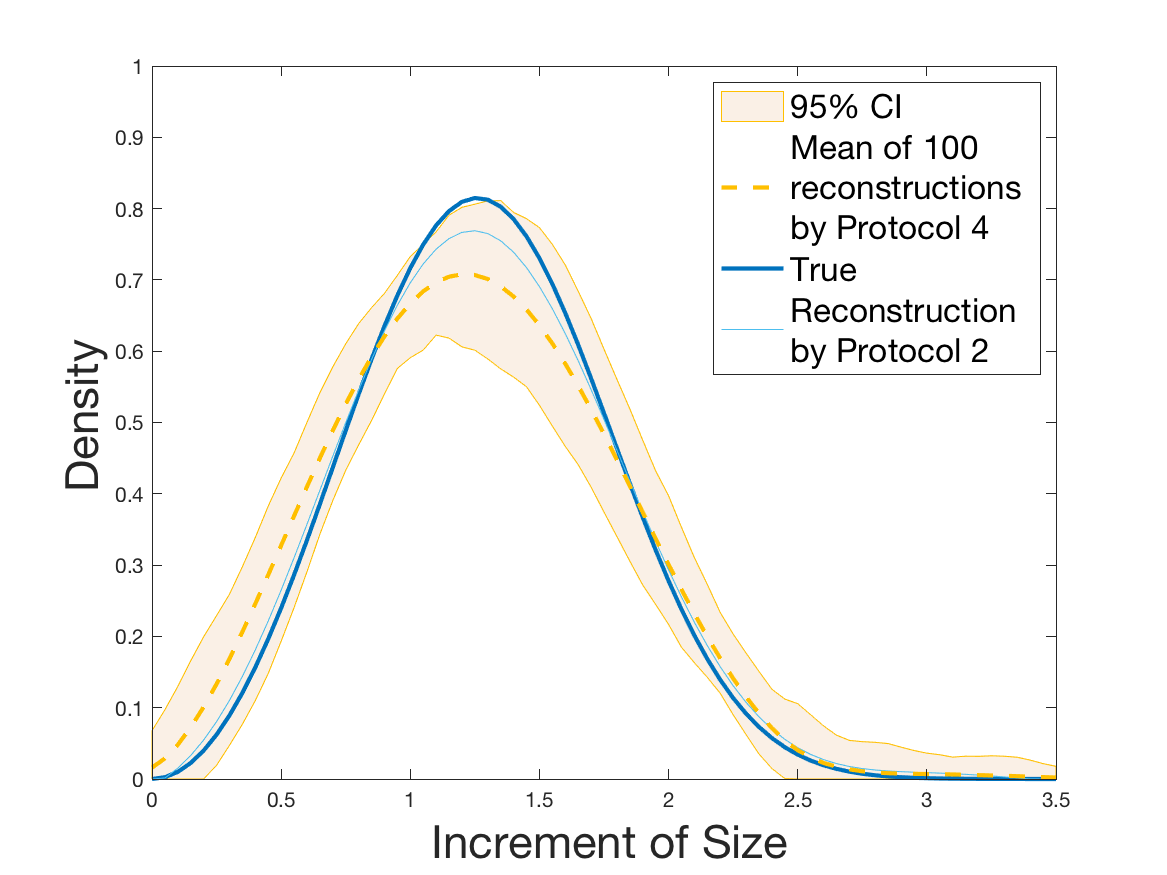}
        \caption{$f_B$, $\widehat{f}_2$ and $\widehat{f}_{4,n}$ in function of $a$}
    \end{subfigure}
    
    \begin{subfigure}[b]{0.5\textwidth} 
        \centering \includegraphics[width=\textwidth]{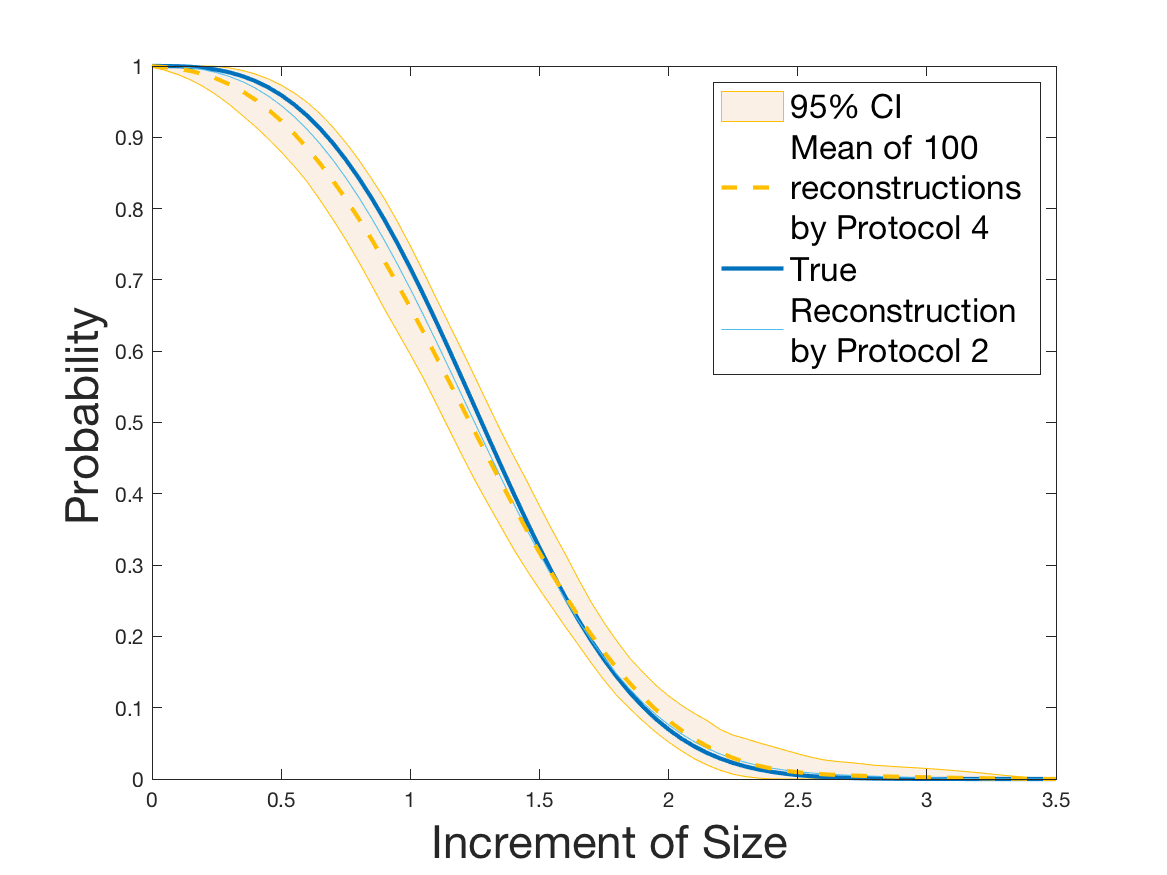}
        \caption{$S_B$, $\widehat{S}_2$ and $\widehat{S}_{4,n}$ in function of $a$}
    \end{subfigure}
    \begin{subfigure}[b]{0.5\textwidth} 
        \centering \includegraphics[width=\textwidth]{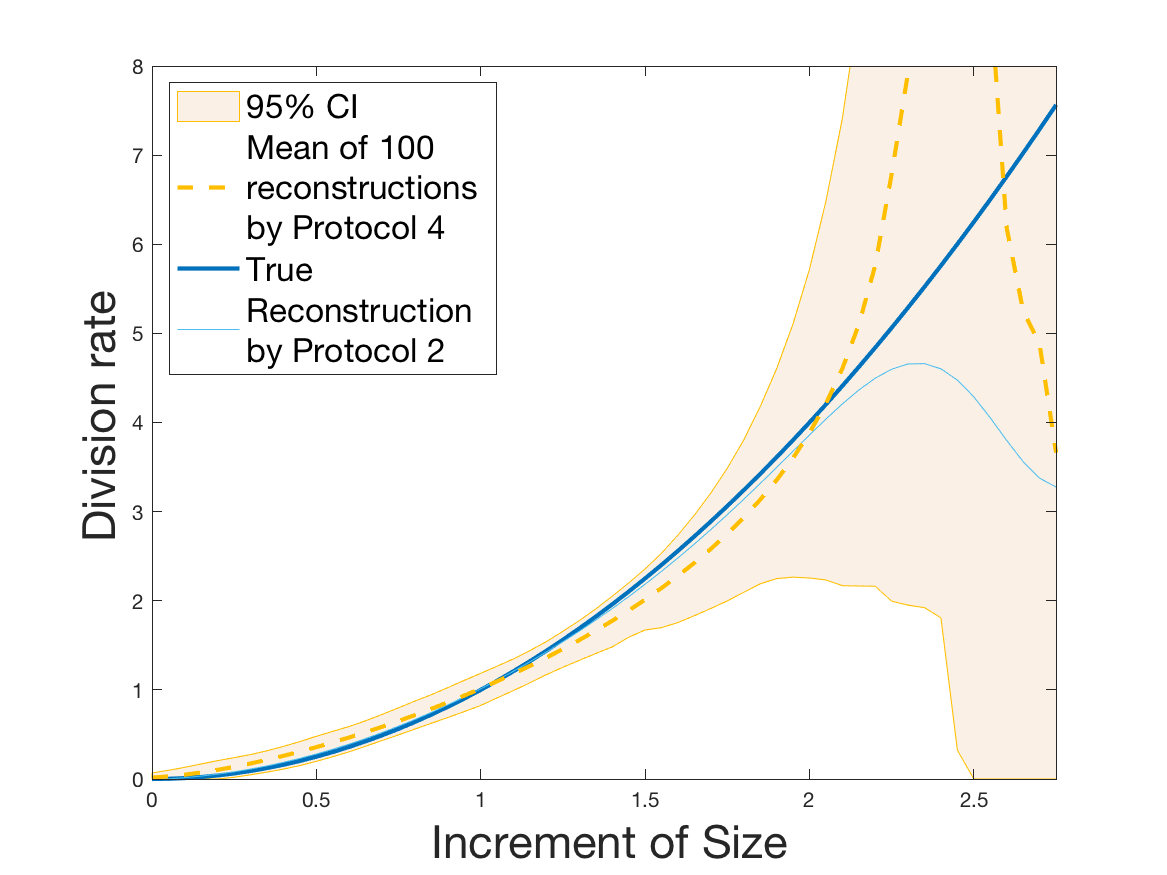}
        \caption{$B$, $\widehat{B}_2$ and $\widehat{B}_{4,n}$ in function of $a$}
    \end{subfigure}
\caption{Results of Protocol 4 for $n=2000$ and $M=100$ Monte Carlo samples.}%
\label{fig:resultats4_IC}%
\end{figure}

\begin{figure}[h]
    \begin{subfigure}[b]{0.5\textwidth} 
        \centering \includegraphics[width=\textwidth]{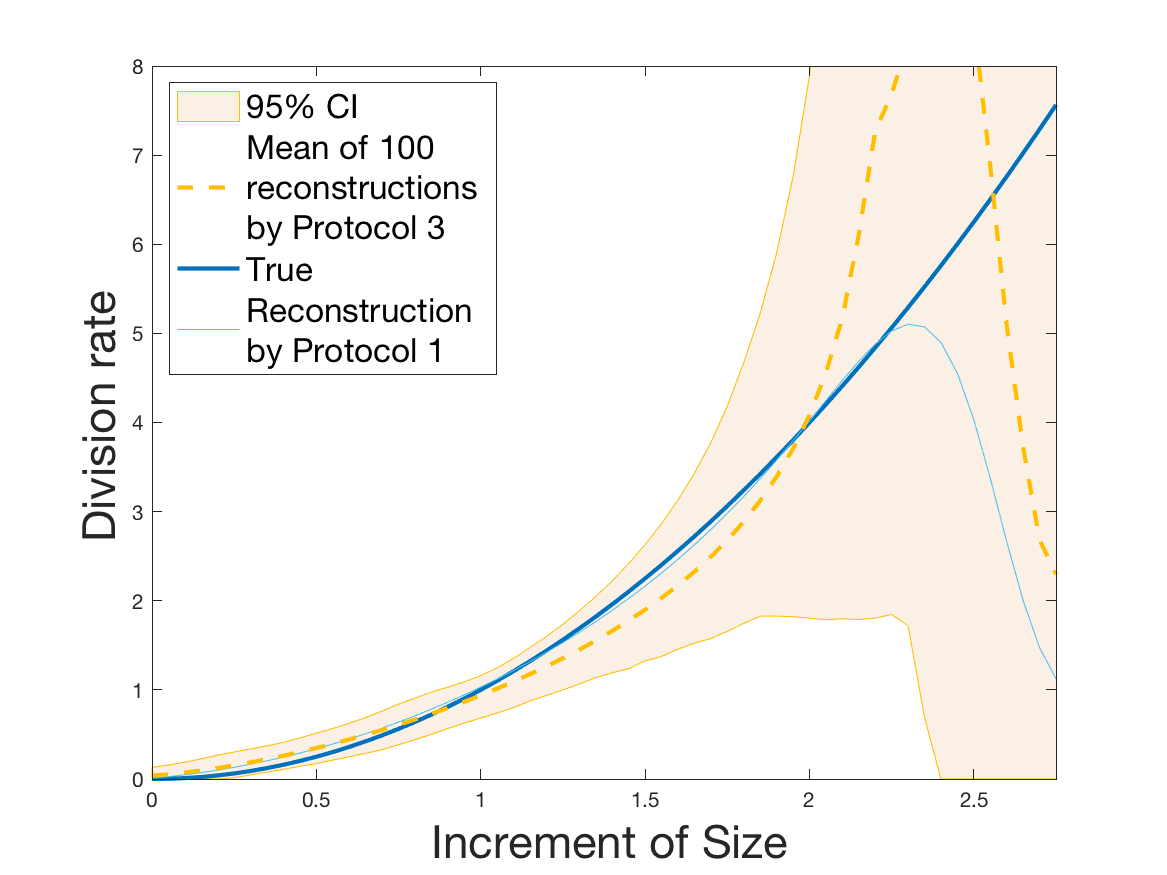}  
        \caption{$\widehat{B}_{3,n}$ with $n=500$}
    \end{subfigure}
    \begin{subfigure}[b]{0.5\textwidth} 
        \centering \includegraphics[width=\textwidth]{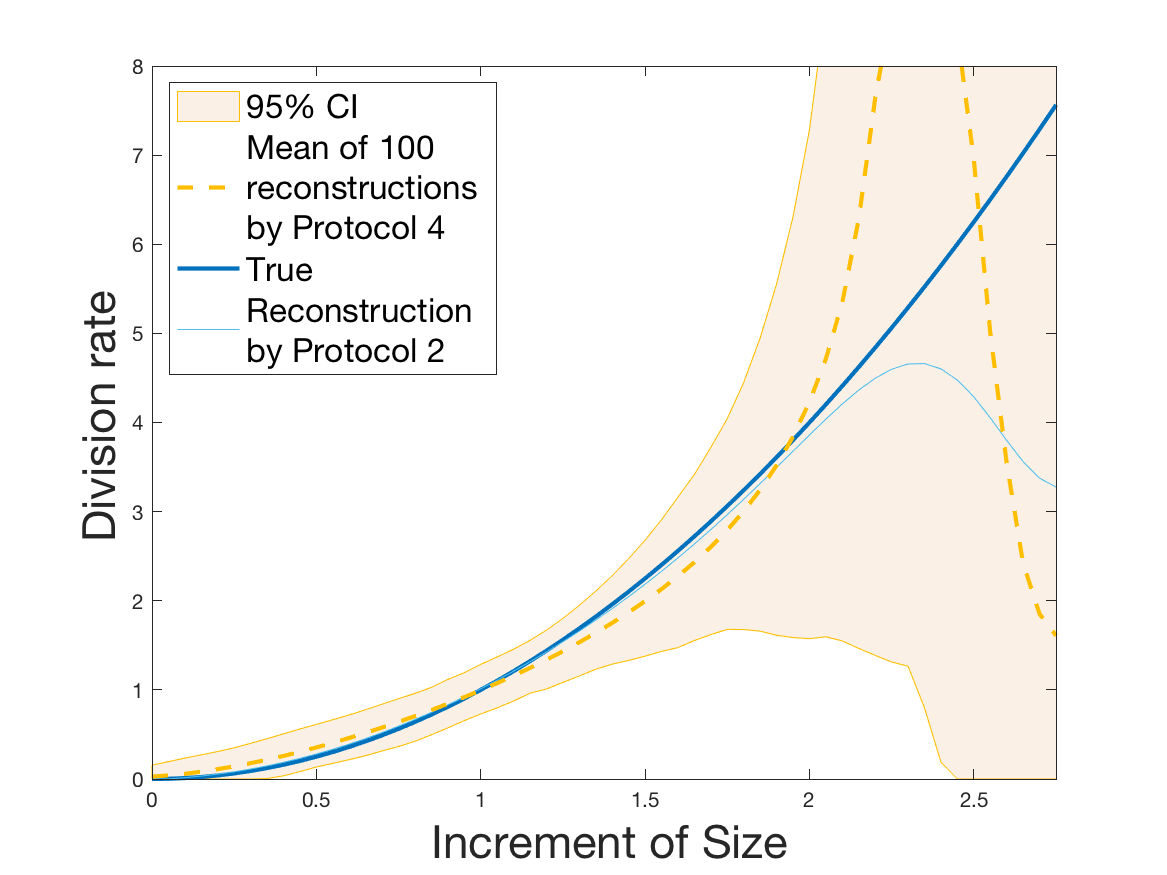} 
        \caption{$\widehat{B}_{4n}$ with $n=500$}
    \end{subfigure}
    
     \begin{subfigure}[b]{0.5\textwidth} 
        \centering \includegraphics[width=\textwidth]{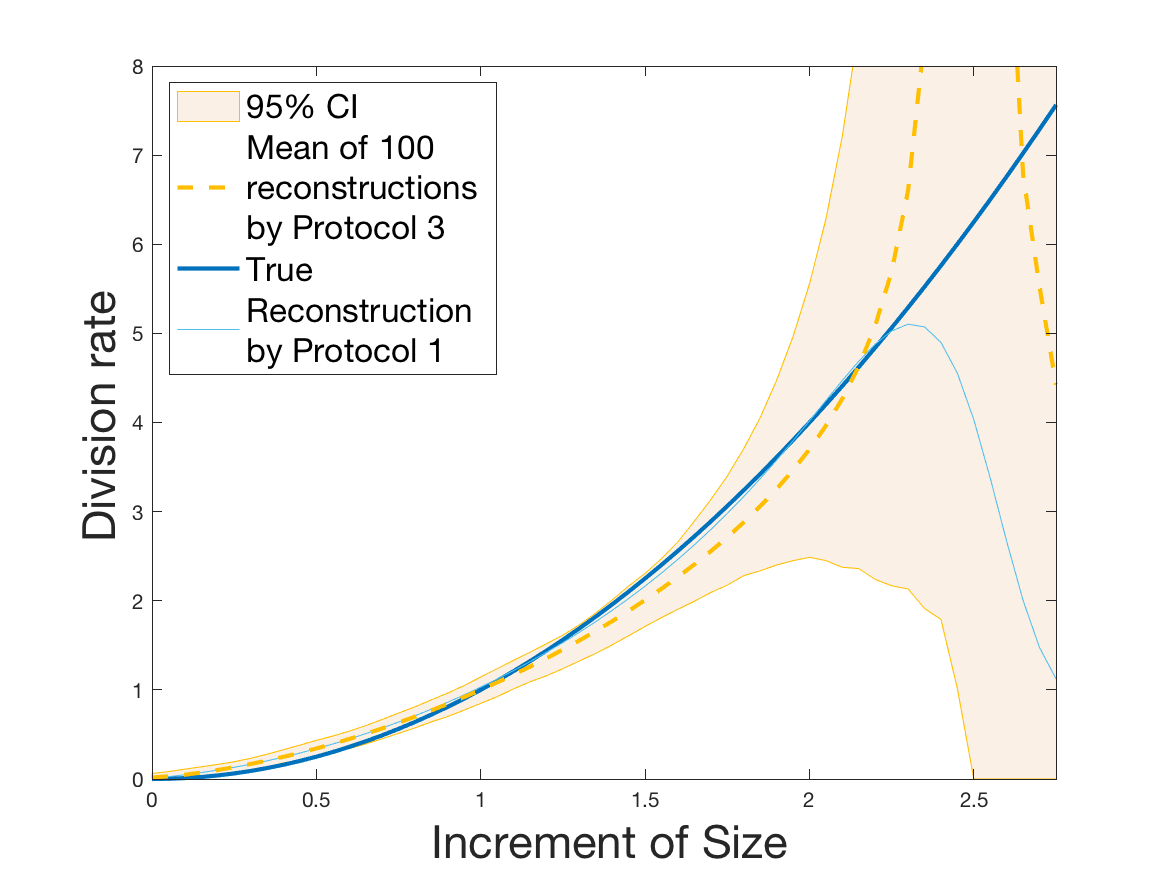}  
        \caption{$\widehat{B}_{3,n}$ with $n=5~000$}
    \end{subfigure}
    \begin{subfigure}[b]{0.5\textwidth} 
        \centering \includegraphics[width=\textwidth]{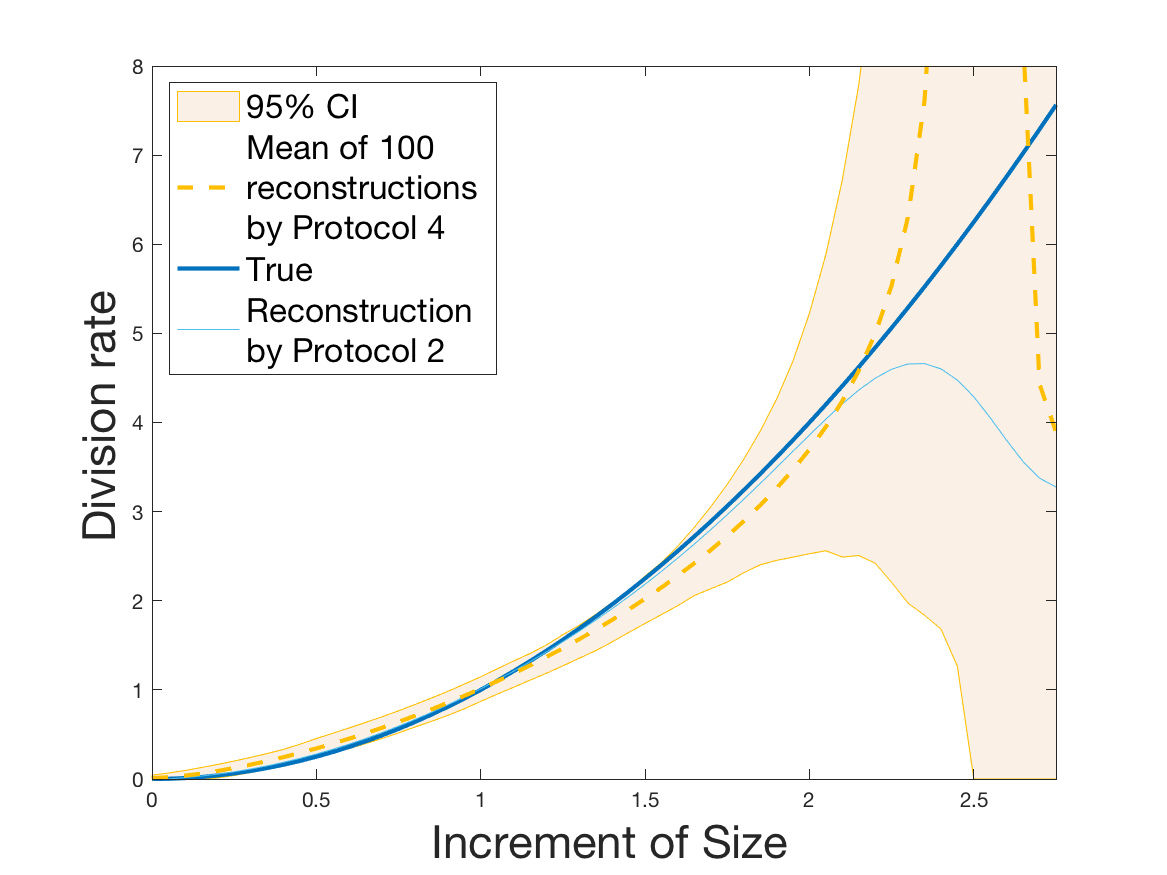} 
        \caption{$\widehat{B}_{4,n}$ with $n=5~000$}
    \end{subfigure}
  
    \begin{subfigure}[b]{0.5\textwidth} 
        \centering \includegraphics[width=\textwidth]{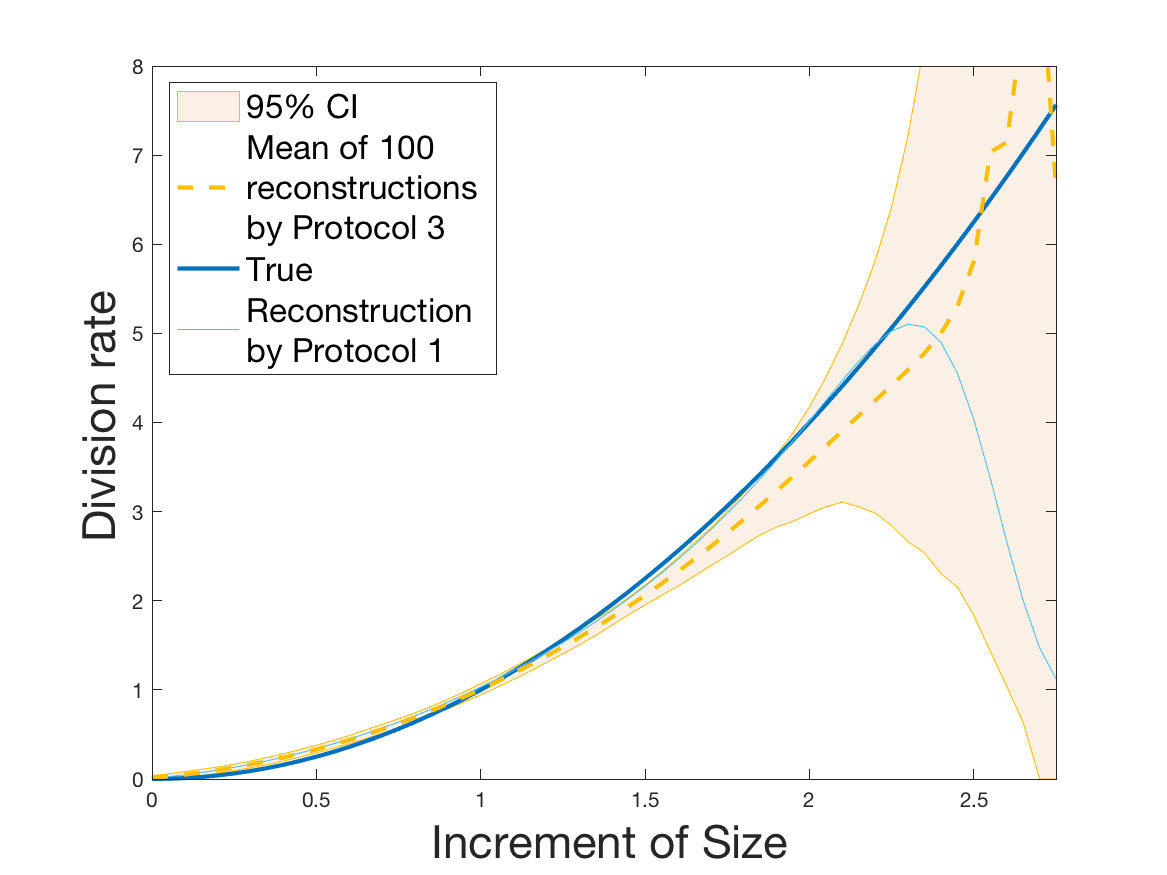}  
        \caption{$\widehat{B}_{3,n}$ with $n=50~000$}
    \end{subfigure}
    \begin{subfigure}[b]{0.5\textwidth} 
        \centering \includegraphics[width=\textwidth]{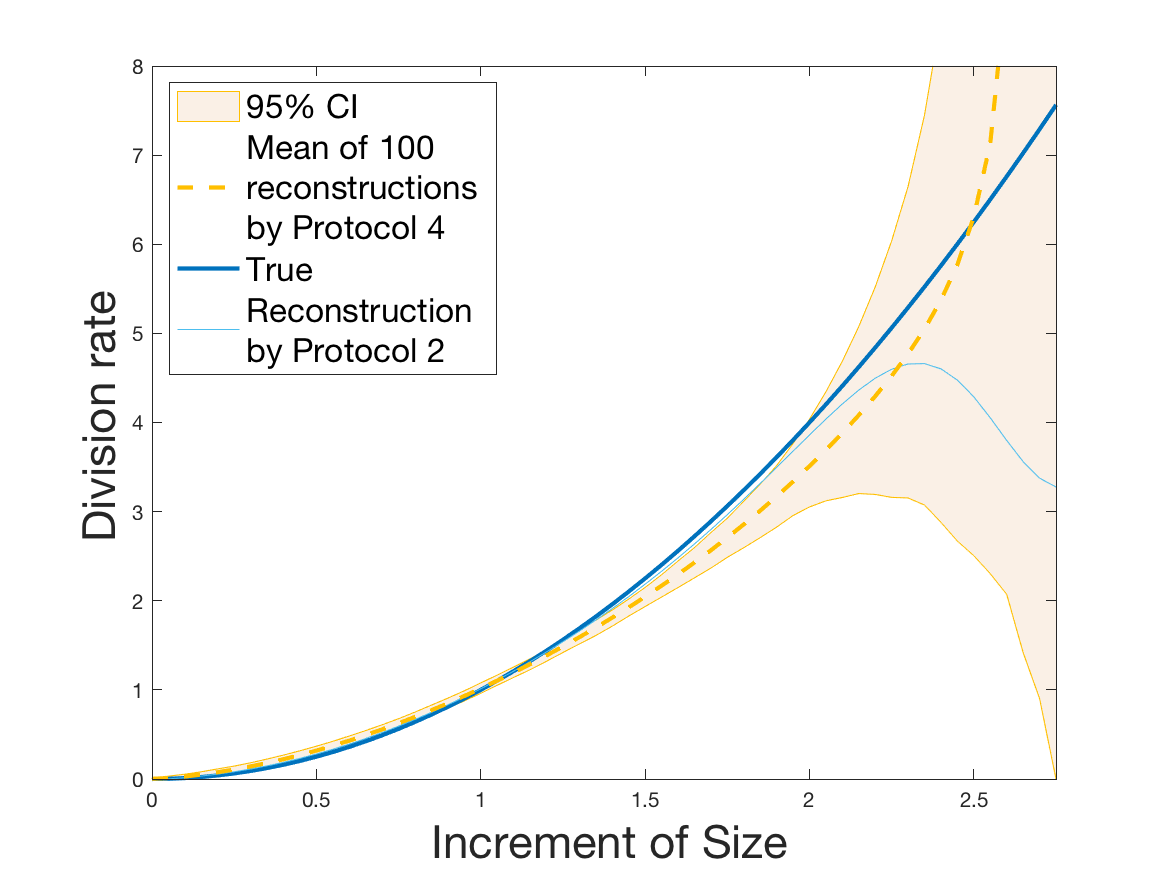} 
        \caption{$\widehat{B}_{4,n}$ with $n=50~000$}
    \end{subfigure} 
\caption{Results of Protocols 3 and 4 -- Estimation of the division rate $B(a) = a^2$ in function of the increment of size $a$ for different $n$ (500, 5~000, 10~000) and $M=100$ Monte Carlo samples. \label{fig:resultats34_IC_n}} %
\end{figure}

\begin{figure}[h]
    \begin{subfigure}[b]{0.5\textwidth} 
        \centering \includegraphics[width=\textwidth]{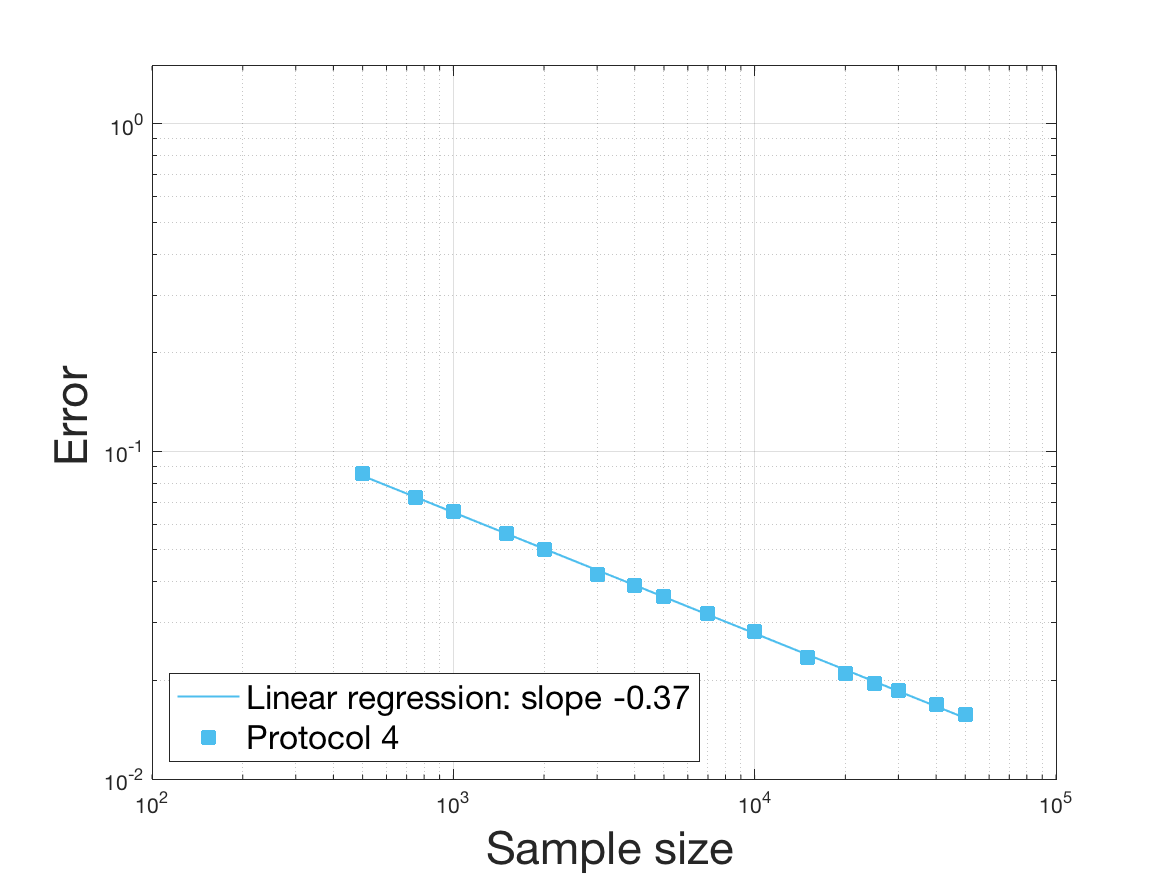}
        \caption{Estimation of $U_{B,x}$}
    \end{subfigure}
    \begin{subfigure}[b]{0.5\textwidth} 
        \centering \includegraphics[width=\textwidth]{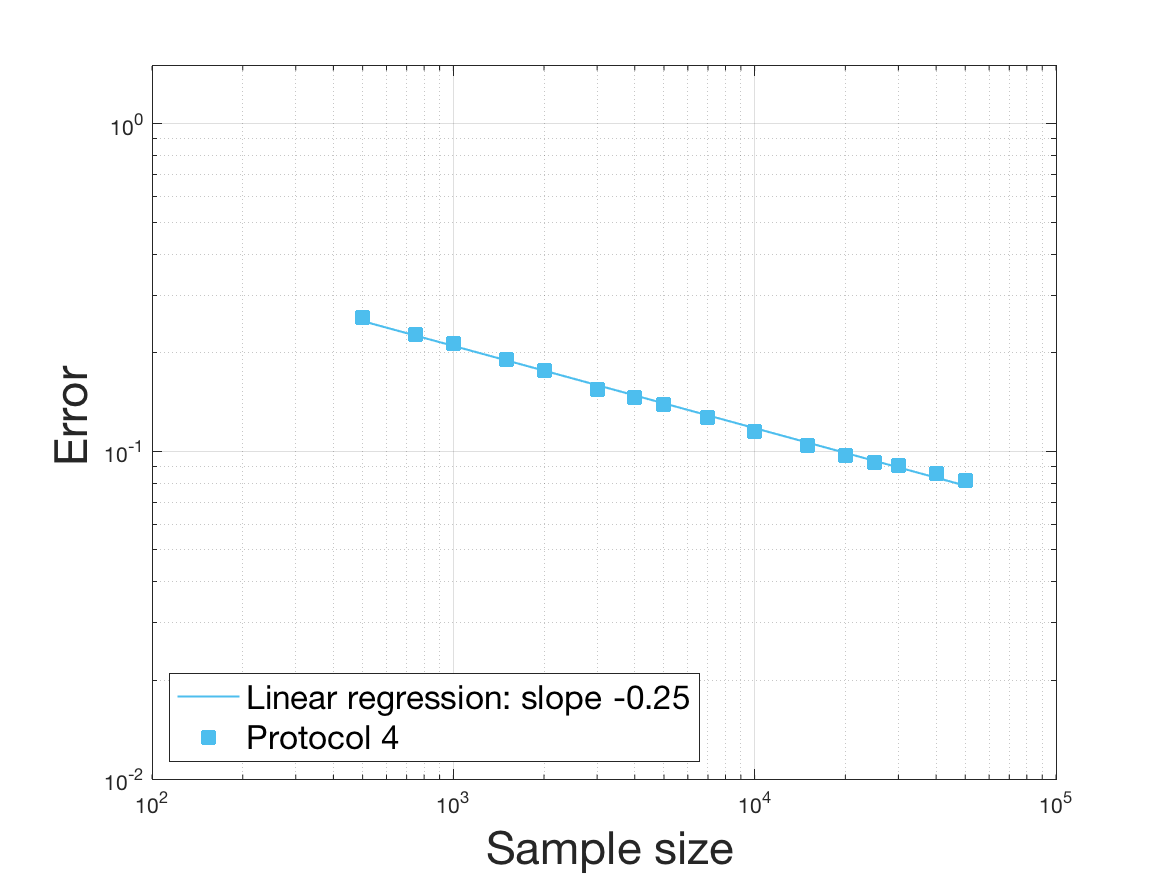} 
        \caption{Estimation of $U^{'}_{B,x}$}
    \end{subfigure}
    
    \begin{subfigure}[b]{0.5\textwidth} 
        \centering \includegraphics[width=\textwidth]{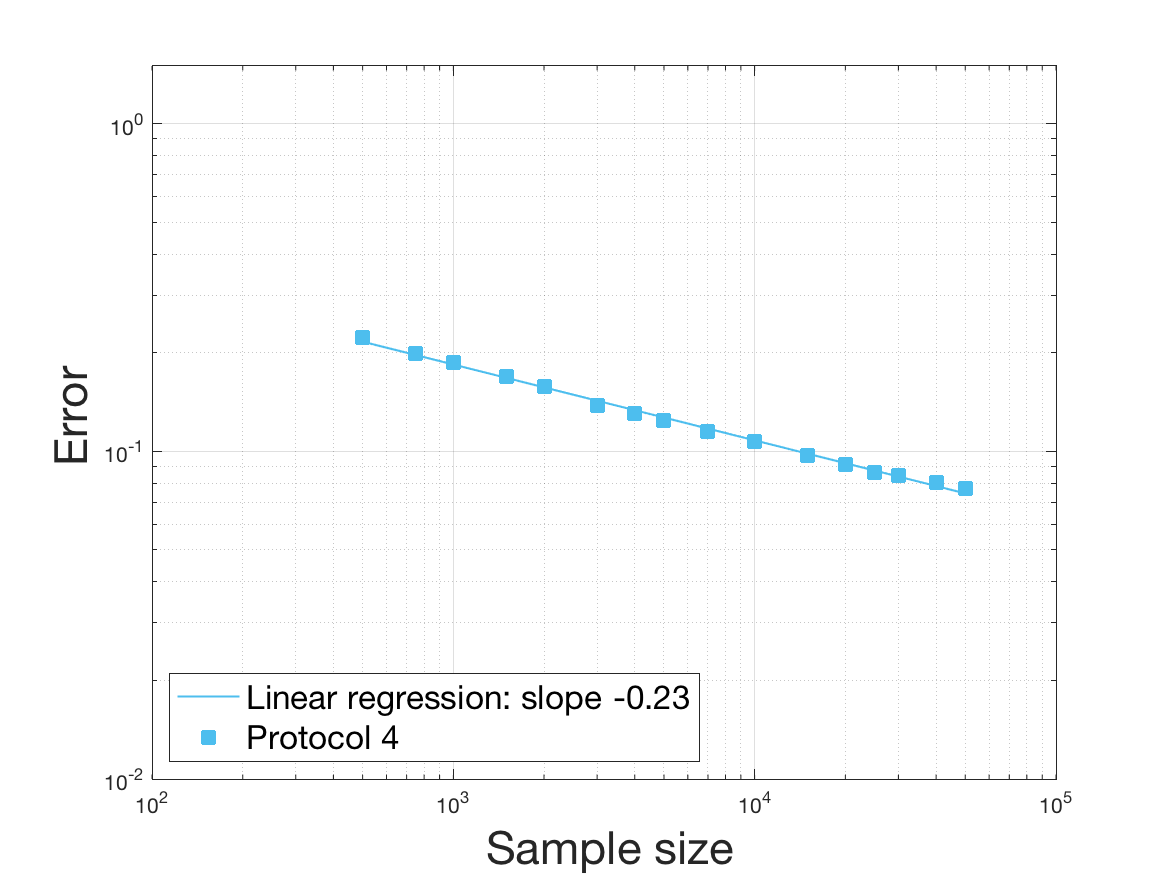}
        \caption{Estimation of $2 U_{B,x} + x U^{'}_{B,x}$}
    \end{subfigure}
    \begin{subfigure}[b]{0.5\textwidth} 
        \centering \includegraphics[width=\textwidth]{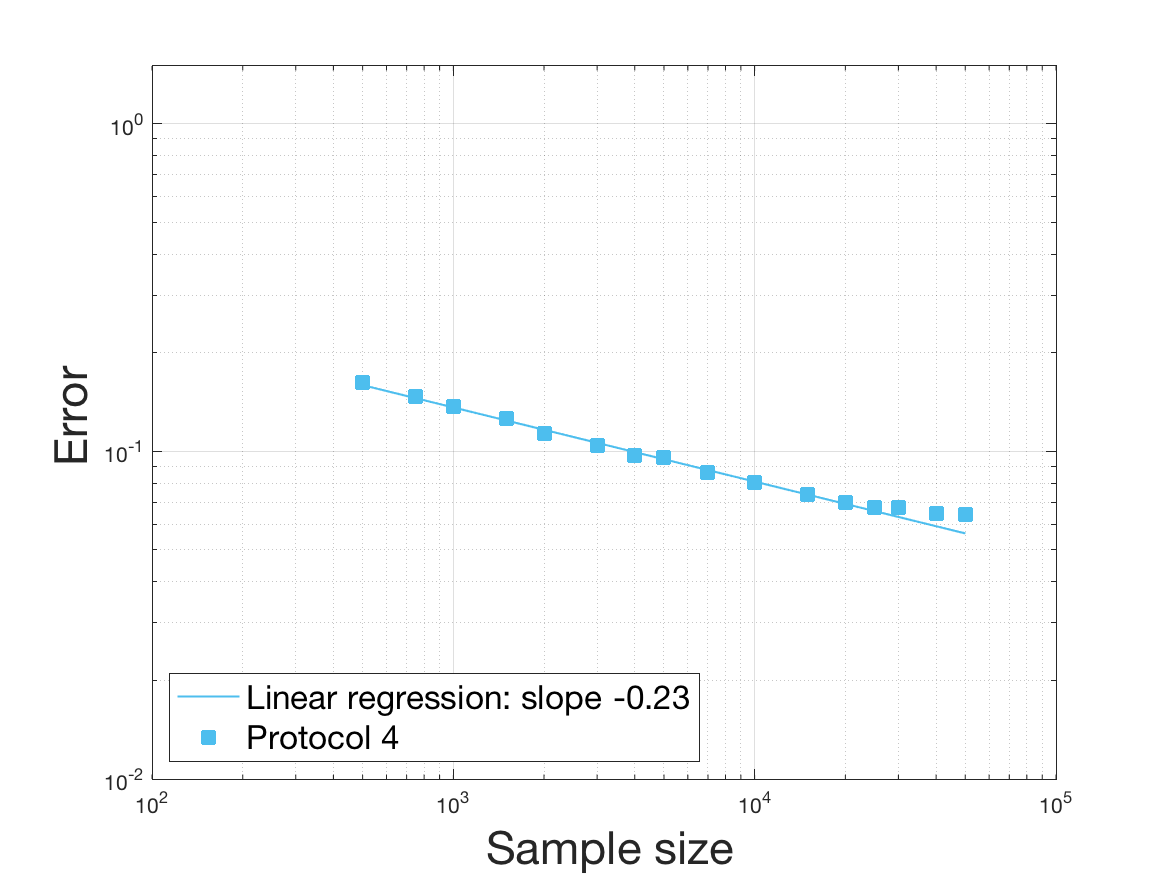}
        \caption{Estimation of $\mathcal L_B$}
    \end{subfigure}
    
    \begin{subfigure}[b]{0.5\textwidth} 
        \centering \includegraphics[width=\textwidth]{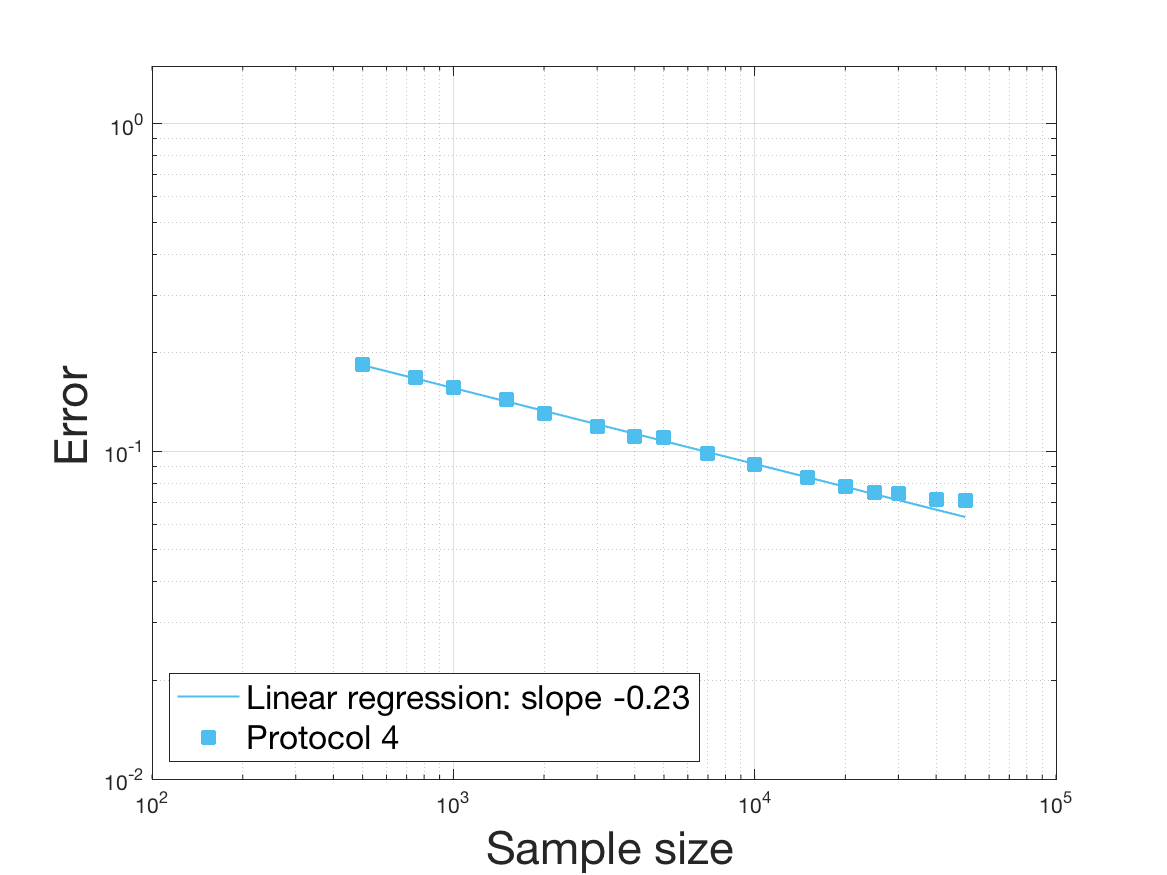}
        \caption{Estimation of $\mathcal G_B$}
    \end{subfigure}
    \begin{subfigure}[b]{0.5\textwidth} 
        \centering \includegraphics[width=\textwidth]{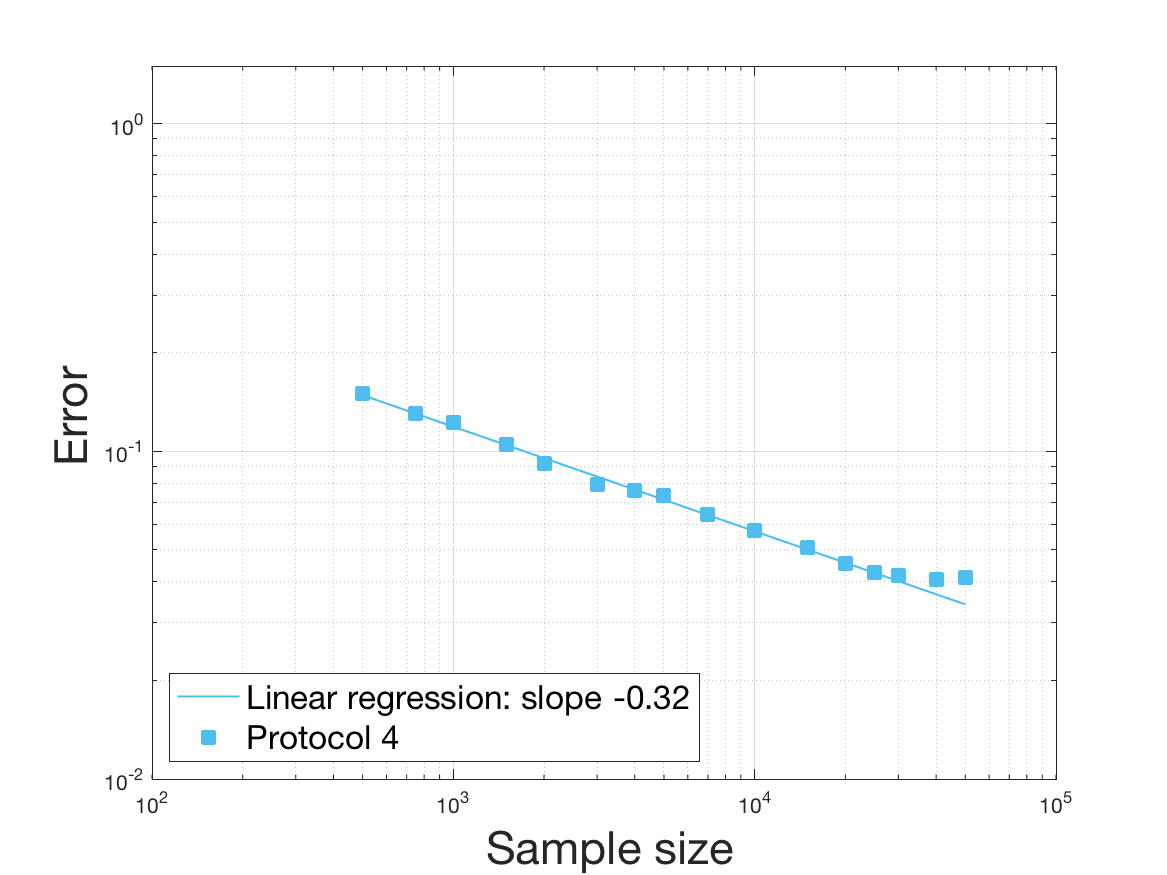}
        \caption{Estimation of $\mathcal G^*_B$}
    \end{subfigure}

\end{figure}

\begin{figure}[h]
    \begin{subfigure}[b]{0.5\textwidth} 
        \centering \includegraphics[width=\textwidth]{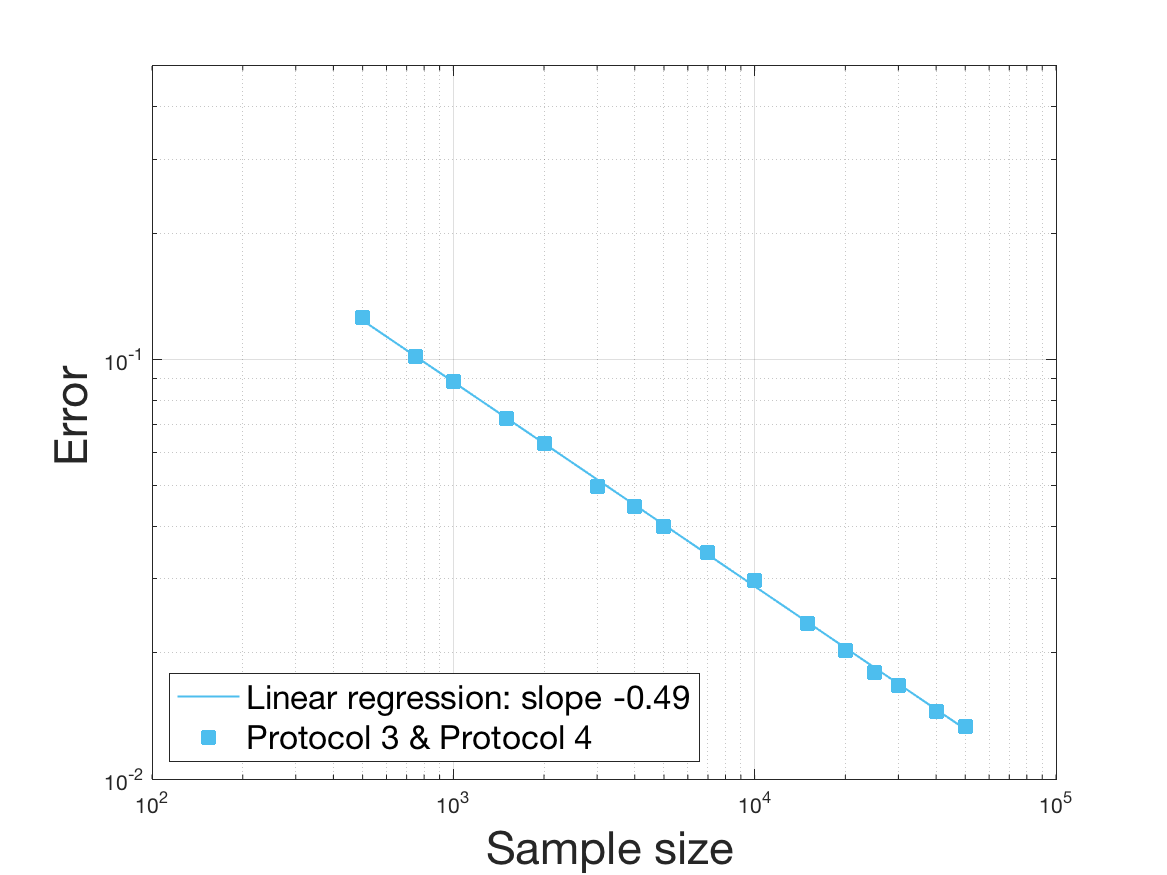} 
        \caption{Estimation of $\mathcal N^*_B$}
    \end{subfigure}
    \begin{subfigure}[b]{0.5\textwidth} 
        \centering \includegraphics[width=\textwidth]{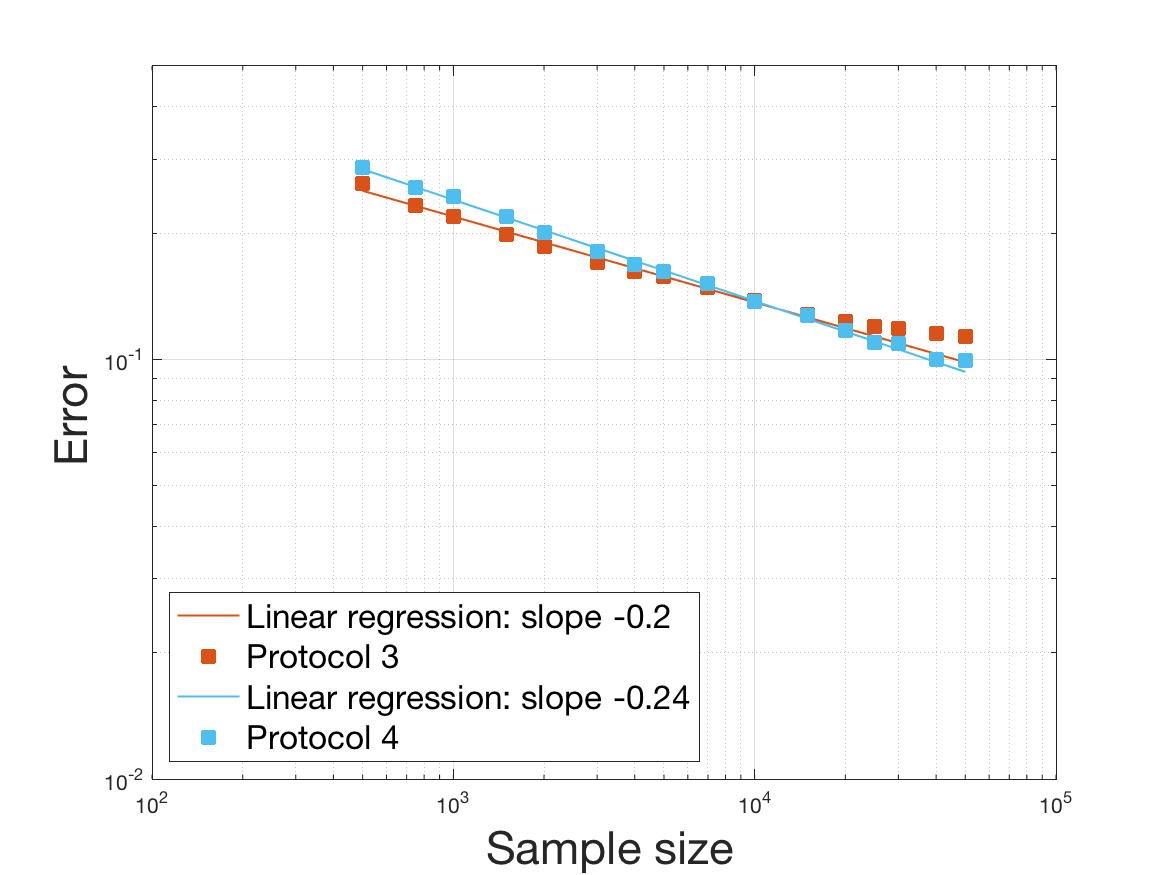}
        \caption{Estimation of $f_B^*$}
    \end{subfigure}
    
    \begin{subfigure}[b]{0.5\textwidth} 
        \centering \includegraphics[width=\textwidth]{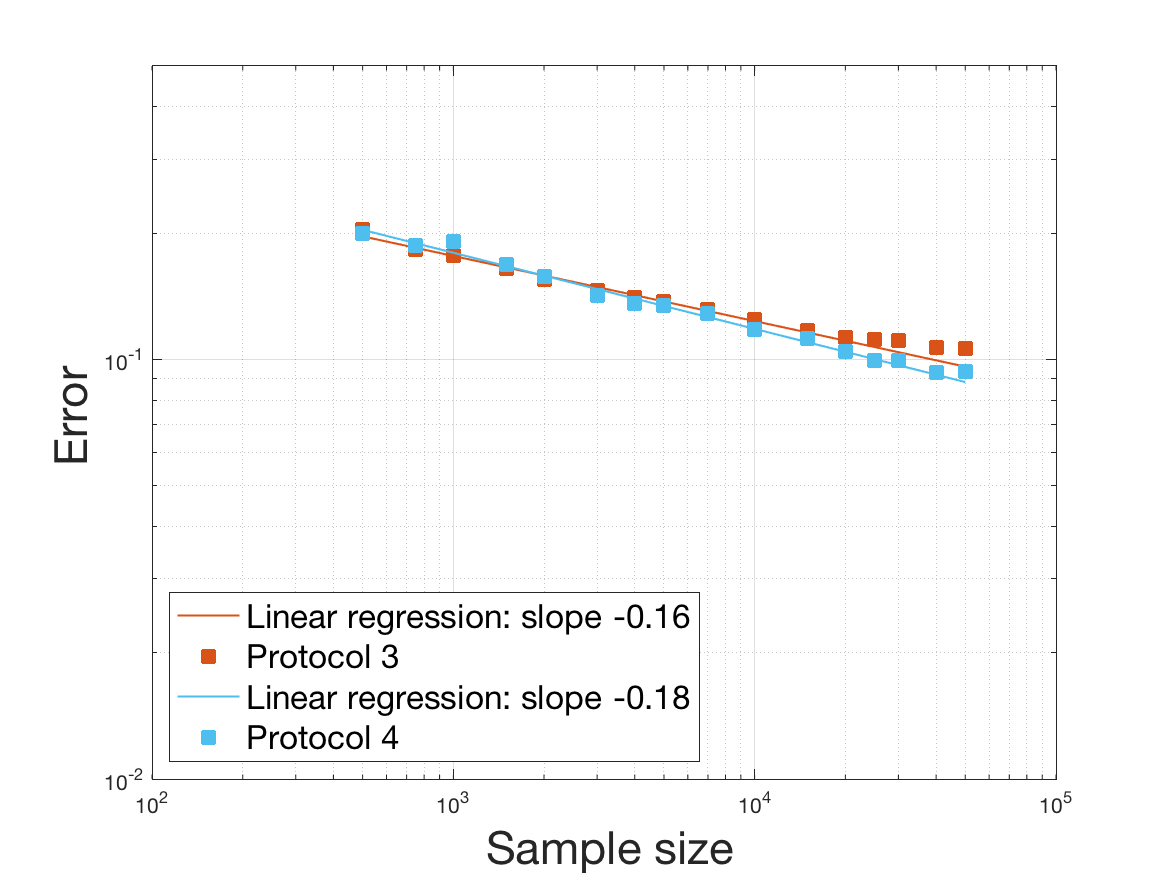} 
        \caption{Estimation of $f_B$}
    \end{subfigure}
    \begin{subfigure}[b]{0.5\textwidth} 
        \centering \includegraphics[width=\textwidth]{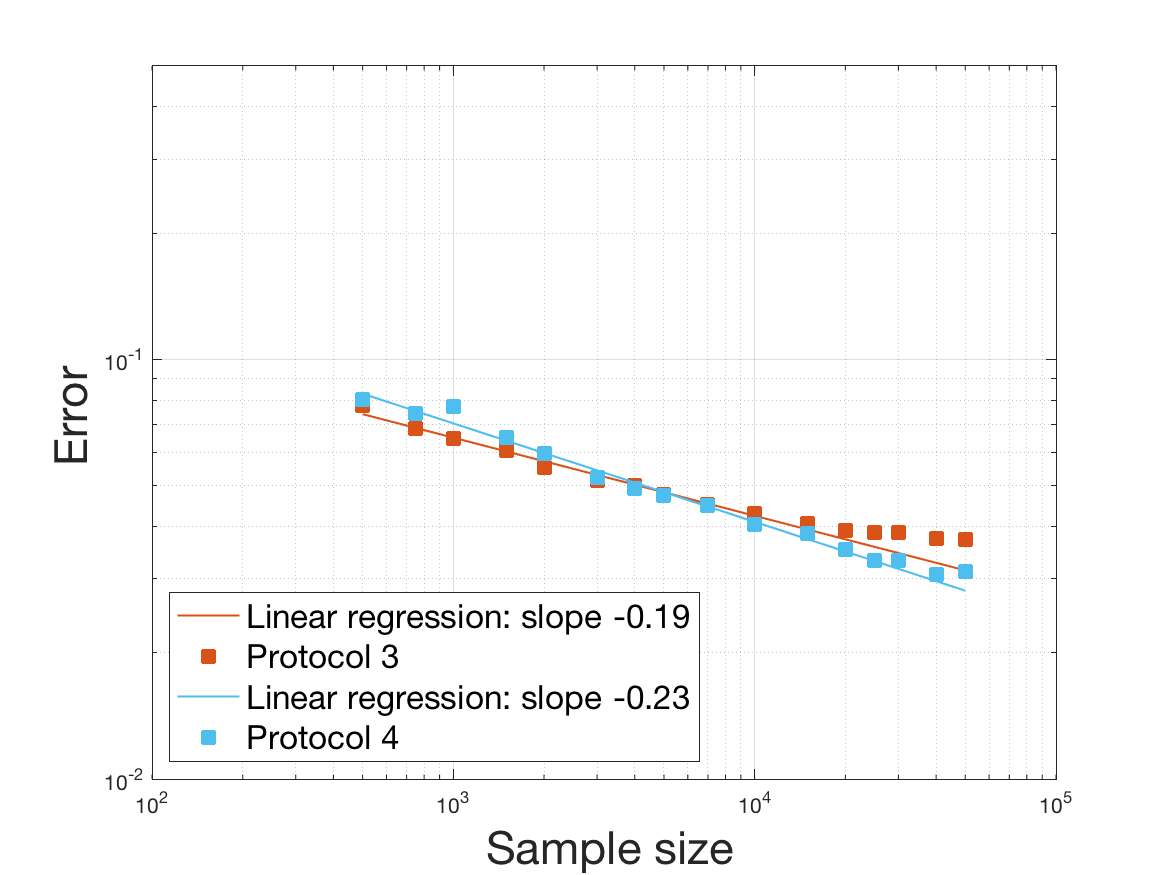} 
        \caption{Estimation of $S_B$}
    \end{subfigure}
    
     \begin{subfigure}[b]{0.5\textwidth} 
        \centering \includegraphics[width=\textwidth]{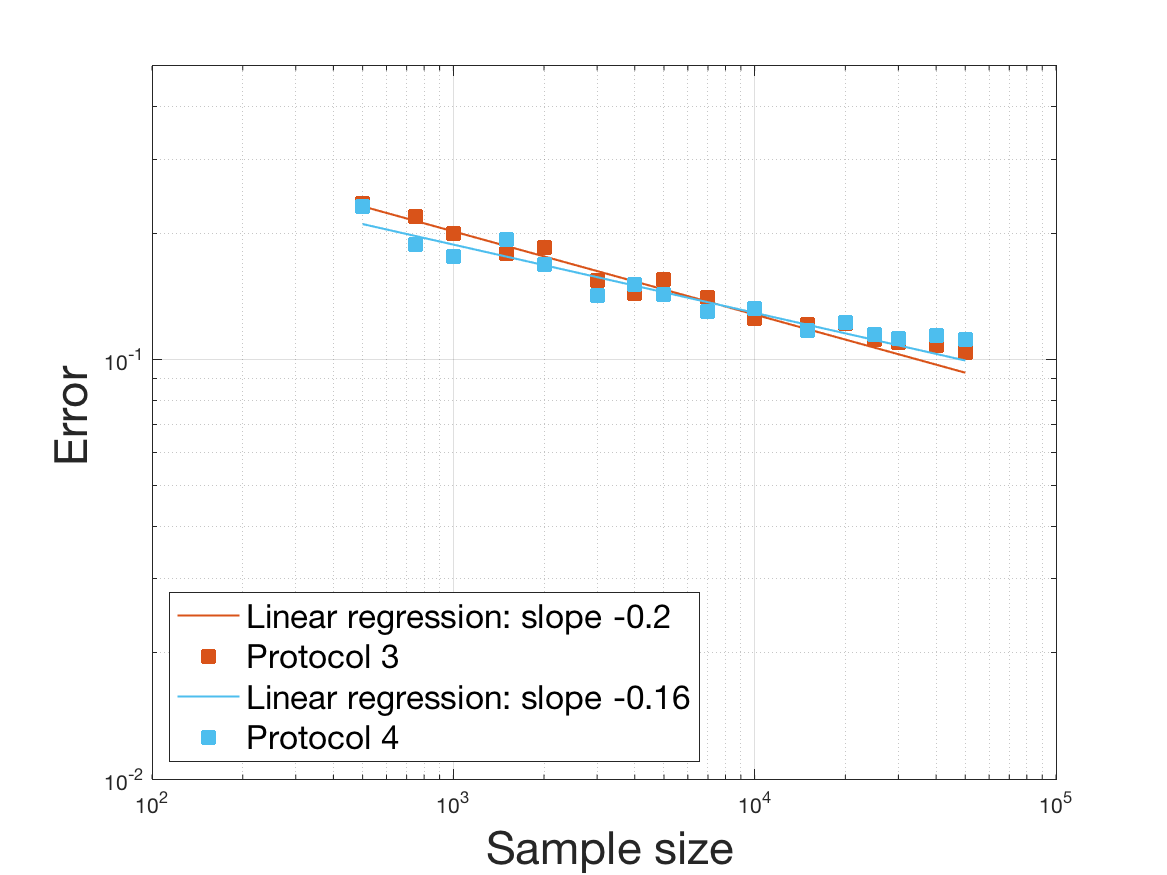} 
        \caption{Estimation of $B$}
    \end{subfigure}
    
\caption{Results of Protocols 3 and 4 -- Reduction of the mean error over $M=100$ samples (in log-scale) in function of the sample size (from $n=500$ to $n=50~000$). {\small Empirical errors are computed over the following regular grids: (a)-(e) $[0;6]$, $\Delta x = \tfrac{6}{500}$; (f)-(h) $[-10;10]$, $\Delta \xi = 0.05$; (i)-(j) $[0;2.25]$, $\Delta a = \tfrac{1}{\sqrt{n}}$; (k) $[0;2]$, $\Delta a = \tfrac{1}{\sqrt{n}}$. }\label{fig:resultats34_vitesse}}%
\end{figure}

\begin{figure}[h]
    \begin{subfigure}[b]{0.5\textwidth} 
        \centering \includegraphics[width=\textwidth]{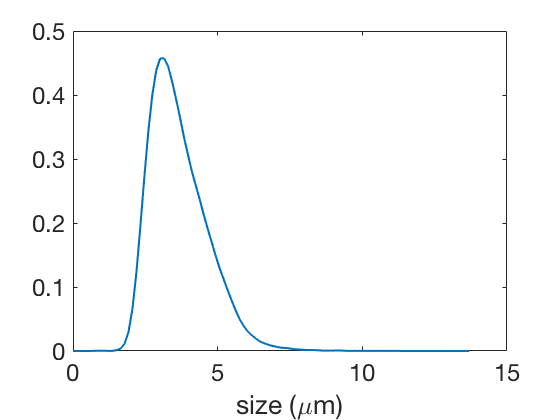} 
        \caption{\label{fig:Nexper}{\bf Estimation of the size distribution} \\ $\widehat U_{n,x}$ of $ U_{B,x}$  from an experimental sample \\ taken from~\cite{Eric}, $n=31,333$, $h=0.125.$ \vspace{0.8cm}}
    \end{subfigure}
    \begin{subfigure}[b]{0.5\textwidth} 
        \centering \includegraphics[width=\textwidth]{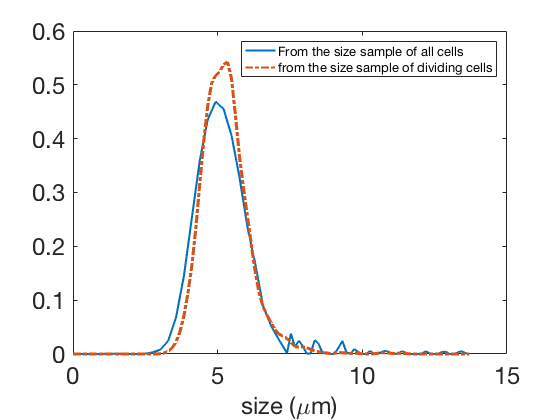}
        \caption{{\label{fig:BNexp}\bf Estimation of the size distribution of dividing cells ${\cal L}_B$:} 1/ estimation ${\cal L}_B$ from $\widehat U_{n,x}$ (plain blue line, Step 2 of Section~\ref{sec:estim:GB}), 2/ estimation from the experimental sample of dividing cells, for which $n_d=1,679$ and $h_d=0.167$ (dotted-dashed red line).}
    \end{subfigure}
    \caption{{\bf Testing the procedure on experimental data.} \\ Initial step: estimation of the size distribution}
    \end{figure}
\begin{figure}[h]
    \begin{subfigure}[b]{0.5\textwidth} 
        \centering \includegraphics[width=\textwidth]{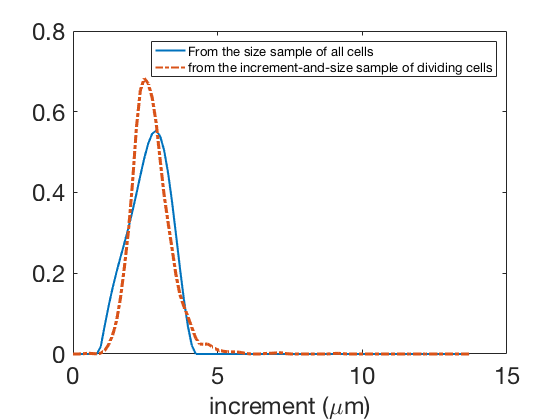} 
        
        \caption{\label{fig:fBexper}{\bf Estimation of $f_B(a)$}}
    \end{subfigure}
    \begin{subfigure}[b]{0.5\textwidth} 
        \centering \includegraphics[width=\textwidth]{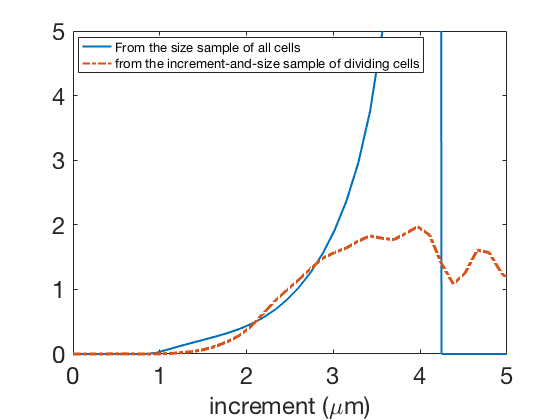}
        \caption{{\label{fig:Bexper}\bf Estimation of the division rate $B(a)$}}
    \end{subfigure}
    \caption{{\bf Testing the procedure on experimental data.} \\ Final step: estimation of the increment-structured division rate}
    \end{figure}


\clearpage

\bibliographystyle{plain}       
\bibliography{AdderModel_EstimateB}  



\end{document}